\documentclass[12pt]{article}

\usepackage{a4}
\usepackage{epsfig} 
\usepackage{amsmath} 
\usepackage{amssymb}  
\usepackage{subfigure}
\usepackage{color}
\usepackage{dsfont,hyperref}

\usepackage[title]{appendix}

\newtheorem{lemma}{Lemma}

\newtheorem{theorem}{Theorem}

\newtheorem{remark}{Remark}

\newtheorem{example}{Example}
\newtheorem{procedure}{Procedure}

\newenvironment{proof}{~~\textit{Proof:}}{\hfill$\blacksquare$}

\begin{document}

\begin{center}
\LARGE ``Passive Mechanical Realizations of Bicubic Impedances with No More Than Five Elements for Inerter-Based Control Design'' with the Supplementary Material
\end{center}

\vspace{1.5cm}

\noindent This report includes the original manuscript (pp.~2--40) and the supplementary material (pp.~41--48) of ``Passive Mechanical Realizations of Bicubic Impedances with No More Than Five Elements for Inerter-Based Control Design''.

\vspace{1cm}

\noindent Authors: Kai Wang and Michael Z. Q. Chen

\newpage



\title{Passive Mechanical Realizations of Bicubic Impedances with No More Than Five Elements for Inerter-Based Control Design}
\author{Kai~Wang$^{1}$  ~ and ~Michael~Z.~Q.~Chen$^{2,}$\footnote{Corresponding author:  Michael~Z.~Q.~Chen, mzqchen@outlook.com. \newline
$^{1}$ Key Laboratory of Advanced Process Control for Light Industry (Ministry of Education), School of Internet of Things Engineering, Jiangnan University, Wuxi 214122, P. R. China (e-mail: kaiwang@jiangnan.edu.cn). \newline
$^{2}$ School of Automation, Nanjing University of Science and Technology, Nanjing 210094, P. R. China (e-mail: mzqchen@outlook.com).
\newline
This work was supported by the National Natural Science Foundation of China under grants 61873129 and 61703184.
} }
\date{}
\maketitle

\begin{abstract}
This paper mainly investigates the passive realization problems of bicubic (third-order) impedances as  damper-spring-inerter networks consisting of no more than five elements.
First, the special case where a bicubic impedance contains a pole or a zero on the imaginary axis or at infinity is discussed.
Then, assuming that there is no pole or zero on the imaginary axis or at infinity, the realizations of bicubic impedances as five-element networks are investigated.
Necessary and sufficient conditions for the realizability
as five-element series-parallel networks and as five-element non-series-parallel networks are derived, respectively, where 22 series-parallel configurations and 11 non-series-parallel configurations are presented to cover the conditions. Finally, two numerical examples together with positive-real controller designs for a quarter-car suspension system
are presented for illustrations. The results of this paper can contribute to the synthesis of low-complexity passive mechanical (or electrical) networks, which are motivated by the synthesis and design of inerter-based vibration control systems.

\medskip

\end{abstract}








\section{Introduction}

As an important branch of system theory,
\emph{passive network synthesis} \cite{AV73,CWC19,You15} is  to
realize   passive systems, described by impedances, admittances, etc., as   electrical (or mechanical) networks consisting of passive elements.
Any network consisting of passive elements must be passive, and the \emph{impedance} of any two-terminal linear  time-invariant passive  network is \emph{positive-real} \cite{AV73}. The impedance is defined as
$Z(s) := V(s)/I(s)$, where $V(s)$ and $I(s)$ are Laplace transforms of port voltage and current, respectively, and a real-rational function $Z(s)$ is defined to be positive-real if $\Re(Z(s)) \geq 0$ for $\Re(s) > 0$
(see \cite{AV73,CWC19}). By the \emph{Bott-Duffin synthesis procedure} \cite{BD49}, any positive-real impedance is realizable as a two-terminal passive network consisting of
resistors, inductors, and capacitors (RLC network).
However, a large number of redundant elements are generated by the synthesis procedure in many cases. So far, the passive realizations of positive-real impedances using the least number of elements have been essential  problems in the field of passive network synthesis, which remain unsolved even for low-order positive-real impedances.

The analogy between passive electrical and mechanical networks has been completed since the invention of \emph{inerters}, where   the resistors, inductors, capacitors, and transformers are analogous to the dampers, springs, inerters, and levers, respectively, through the force-current analogy \cite{Smi02}.  For a two-terminal mechanical network, the   impedance $Z(s)$ (resp. admittance $Y(s)$) is defined to be the ratio of the Laplace transform  of the relative velocity (resp. force) to the Laplace transform  of the force (resp. relative velocity).
Therefore, passive network synthesis can be completely applied to designing passive mechanical circuits, and damper-spring-inerter networks realizing positive-real impedance (or admittance) controllers
have been widely applied to a series of passive or semi-active vibration control systems, such as vibration isolation systems \cite{ACWJ18,JW19,SZ19}, vehicle suspension systems \cite{CLLNSC17,HC18,NSYZDZL19,PS06,SW04,ZJN17}, train suspension systems \cite{WLLSC09}, building vibration   systems \cite{CL19,PRRK19,YS16}, wind turbine towers \cite{WNZ19}, etc. The inerter-based control approach has low cost and high reliability, and introducing inerters can provide better system performances.
Therefore, after designing a suitable positive-real impedance (or admittance) controller by optimizing the system performances, the positive-real function can be further realized as a passive mechanical network consisting of dampers, springs, and inerters, by utilizing the approach of passive network synthesis. Moreover, the realizability conditions in network synthesis can be applied to the optimization of control systems, in order to satisfy the  network complexity constraint (see \cite{CLLNSC17,HC18,ZJN17}). Therefore, it is practically essential to investigate the area of passive network synthesis, especially the minimal realizations of low-order impedances.

Moreover, passive network synthesis can be applied to many other fields, such as
microwave antenna circuit design \cite{LBHD11},    self-assembling circuit design \cite{DGYM20},   supercapacitor model synthesis~\cite{DZHD17},
 acoustics simulation \cite{BHBS16}, biometric image processing \cite{Sae14},
frequency control  \cite{PM19},   positive-real and negative imaginary   systems \cite{Hak20,LX18}, etc.
In recent years, there have been a series of new results on passive network synthesis (see \cite{CWC19,CS09(2),CS09,CWSL13,Hug17,Hug20,JS11,LQ19,ST17,WCH14,WJ19,WC18,WC20,YKKP14,ZJWN17}).  Specifically, Kalman has made an independent call for a renewed attempt in passive network synthesis as an important branch of system theory \cite{Smi17}.

It is essential to investigate the passive  damper-spring-inerter realization problems of low-order impedances using the least number of elements, due to the practical constraints on space, cost, weight, etc., for mechanical  systems.
Many existing works have focused on the realization problems of biquadratic (second-order) impedances \cite{Hug17,JS11,WCH14}.
Recently, some investigations on bicubic (third-order) impedances have been made \cite{CWSL13,Hug20,WJ19,ZJWN17}, where the bicubic impedance is more general and can provide better system performances in vibration control systems with respect to the biquadratic case (see \cite{HC18,ZJN17,WLLSC09}).
The minimal realizations of some specific classes of bicubic positive-real impedances have been investigated in \cite{CWSL13,WJ19,ZJWN17}. The realization results of a bicubic positive-real impedance as a series-parallel network consisting of three energy storage
elements and a finite number of resistors (dampers) have been derived in \cite{Hug20}.
Since a  bicubic positive-real impedance is realizable with no more than 13 elements (resp. 12 elements) by the Bott-Duffin synthesis procedure (resp. \emph{Pantell's modified Bott-Duffin procedure} \cite[Section~2.4]{CWC19}), it is necessary to obtain the realization results of bicubic impedances with no more than $k$ elements for $k = 1, 2, ..., 11$, in order to completely solve the minimal realizations of bicubic impedances. Moreover, the realization results including the realizability conditions and covering configurations can be
applied to the optimization designs of third-order positive-real impedance controllers in  inerter-based vibration systems,
such that the third-order positive-real controller to be obtained can always be realized as a passive damper-spring-inerter network satisfying the required complexity.

This paper is concerned with the realization problem of bicubic impedances as damper-spring-inerter networks containing no more than five elements, which is a critical starting point of solving minimal realizations of bicubic impedances.
First, the realization problems of the bicubic impedances containing a pole or zero on $j \mathbb{R} \cup \infty$ with no more than five elements  are investigated in Section~\ref{sec: Poles or Zeros on Imaginary Axis}.
It is shown that the realization results of the bicubic impedance containing a pole or zero    at the origin $s = 0$   or infinity
$s = \infty$
can be referred to the existing results in \cite{WJ19}, and any bicubic positive-real impedance containing a finite pole or zero on
$j \mathbb{R} \setminus \{0\}$ is realizable as a series-parallel  network containing three energy storage elements and no more than two dampers (Theorem~\ref{theorem: third-degree five-elements special case}). Then, under the assumption that there is no pole or zero on $j \mathbb{R} \cup \infty$, the main realization results of this paper are derived in Section~\ref{sec: main results}, where it can be proved that the least number of elements for realizations is five.
A necessary and sufficient condition is derived  for the realizability of such a bicubic impedance as a five-element series-parallel network (Theorem~\ref{theorem: third-degree five-element series-parallel network}), by obtaining 22 covering configurations classified into six quartets (Figs.~\ref{fig: N1}--\ref{fig: N6})
and investigating their realizability conditions. Furthermore, a necessary and sufficient condition is derived for such a bicubic impedance to be realizable as a five-element non-series-parallel network (Theorem~\ref{theorem: third-degree five-element non-series-parallel network}), by obtaining  11 covering configurations classified into five quartets (Figs.~\ref{fig: N7}--\ref{fig: N11}) and investigating their realizability conditions. Finally, two numerical examples together with positive-real controller designs for a quarter-car suspension system  are presented for illustrations in  Section~\ref{sec: examples}, where it is shown  that the third-order positive-real controller    realizable as a five-element network using the results of this paper can provide better ride comfort performances than the second-order positive-real controller $K(s)$  realizable as a series-parallel (resp. non-series-parallel) network containing no more than nine (resp. eight) elements.

The contributions of this paper are summarized as follows. Since the least number of elements to realize the bicubic impedance with positive coefficients is five,
the methodology and realization results of this paper can provide the guidance on investigating the realizations   of bicubic impedances as passive networks with higher complexity, in order to finally solve the minimal realization problems of positive-real bicubic impedances.
The realizability conditions and the network element values in explicit forms are derived in this paper, which makes it more convenient
to obtain the network realizations compared with the classical Bott-Duffin synthesis procedure.
Moreover, the realization results in this paper can be utilized in the optimization design of positive-real controllers realizable as five-element damper-spring-inerter networks in many vibration systems.
As illustrated in this paper, for the quarter-car suspension systems, the optimal positive-real controller in the bicubic  form realizable as five-element networks can provide  both lower
physical complexity and better ride comfort performances than the optimal
positive-real controller in the biquadratic form.
The numerical examples also show that using the five-element realization results in this paper,
the positive-real bicubic impedance satisfying the corresponding realizability conditions  can be realized with much fewer elements than the Bott-Duffin realizations.

The  networks in this paper are assumed to be two-terminal linear  time-invariant passive damper-spring-inerter networks, whose element values are positive and finite. The realization results can be directly applied to those of electrical RLC networks based on the analogy between passive electrical and mechanical systems (see \cite{Smi02}). For the brevity of this paper, the detailed proofs of some results can be referred to the supplementary material \cite{WC_sup}.

\section{Problem Formulation} \label{sec: problem formulation}

A \emph{bicubic impedance} $Z(s)$ is a real-rational impedance function whose \emph{McMillan degree}\footnote{For any real-rational function $Z(s) = a(s)/d(s)$ with polynomials $a(s)$ and $d(s)$ being coprime, the McMillan degree (or called degree) of $Z(s)$ is equal to the maximum degree of $a(s)$ and $d(s)$, that is, $\delta(Z(s)) = \max\{\deg(a(s)), \deg(d(s)) \}$ \cite[Section~3.6]{AV73}.} is three, which is denoted as $\delta(Z(s)) = 3$.
The general form of a bicubic impedance can be expressed as
\begin{equation}  \label{eq: three-degree impedance}
Z(s) = \frac{a_3 s^3 + a_2 s^2 + a_1 s + a_0}{d_3 s^3 + d_2 s^2 + d_1 s + d_0} =: \frac{a(s)}{d(s)},
\end{equation}
where $a_i, d_j \geq 0$ for $i, j = 0, 1, 2, 3$, and there is no common factor between $a(s)$ and $d(s)$.
A necessary and sufficient condition for a  bicubic impedance  to be positive-real has been presented in  \cite[Theorem~13]{CS09(2)}, which is shown as follows.




\begin{lemma} \cite{CS09(2)} \label{lemma: bicubic positive-real}
{Consider a third-degree impedance $Z(s)$ in the form of \eqref{eq: three-degree impedance}, where $\delta(Z(s)) = 3$ and $a_i, d_j \geq 0$ for $i, j = 0, 1, 2, 3$. Then, $Z(s)$ is positive real, if and only if $(a_1+d_1)(a_2+d_2)\geq (a_0+d_0)(a_3+d_3)$, and one of the following conditions holds with $f_0 := a_0 d_0$, $f_1 := a_1d_1 - a_0d_2 - a_2d_0$, $f_2 := a_2d_2 - a_1d_3 - a_3d_1$, and $f_3 := a_3d_3$:

(a) $f_3 = 0$, $f_2 \geq 0$, $f_0 \geq 0$, and $-f_1 \leq 2\sqrt{f_0f_2}$;

(b) $f_3 > 0$, $f_0 \geq 0$, and (b1) or (b2) holds: (b1) $f_1 \geq 0$ and $-f_2 \leq \sqrt{3f_1f_3}$; (b2) $f_2^2 > 3f_1f_3$ and $2f_2^3 - 9f_1 f_2 f_3 + 27f_0f_3^2 \geq 2(f_2^2 - 3f_1f_3)^{3/2}$.  }
\end{lemma}

By applying the \emph{Bott-Duffin synthesis procedure} \cite{BD49} (resp. \emph{Pantell's modified Bott-Duffin procedure} \cite[Section~2.4]{CWC19}), any positive-real bicubic impedance $Z(s)$ in the form of \eqref{eq: three-degree impedance} is realizable as a series-parallel (resp. non-series-parallel) damper-spring-inerter network containing no more than
13 elements (resp. 12 elements).  To
simplify the complexity of mechanical network realizations, it is essential to investigate the restricted-complexity
realization problems of bicubic impedances  as  damper-spring-inerter networks.

This paper aims to derive necessary and sufficient conditions for a bicubic impedance $Z(s)$ in the form of \eqref{eq: three-degree impedance} to be realizable as
a  damper-spring-inerter network containing no more than five elements, and to present the realization configurations  to cover  the conditions.

\section{Notations and Preliminaries} \label{sec: notations}

This section will introduce the notations utilized in the remaining part of this paper.

Following the definition in \cite[Definition~8.24]{Fuh12}, the \emph{Bezoutian matrix} $\mathcal{B}(a,d)$ of two third-degree polynomials $a(s)$ and $d(s)$ in \eqref{eq: three-degree impedance} is a real symmetric $3\times 3$ matrix whose entries $\mathcal{B}_{ij}$ for $i, j = 1, 2, 3$ satisfy
\[
\frac{a(z)d(w)-d(z)a(w)}{z-w} =  \sum_{i=1}^{3}\sum_{j=1}^{3} \mathcal{B}_{ij}z^{i-1}w^{j-1}.
\]
Then, the following notations are introduced as follows:
\begin{equation*}
\begin{split}
\mathcal{B}_{11} := a_1d_0 - a_0d_1, ~~ \mathcal{B}_{12} := a_2d_0-a_0d_2, ~~ \mathcal{B}_{13} := a_3d_0 - a_0d_3,  \\
\mathcal{B}_{22} := \mathcal{B}_{13} + a_2 d_1 - a_1d_2, ~~ \mathcal{B}_{23} := a_3d_1 - a_1d_3, ~~ \mathcal{B}_{33} := a_3d_2 - a_2d_3.
\end{split}
\end{equation*}
Moreover, one denotes
\begin{equation*}
\begin{split}
\Delta_1 := a_1a_2 - a_0a_3, ~~ \Delta_2 := d_1 d_2 - d_0 d_3, ~~ \mathcal{M}_{11} := a_1d_0 + a_0d_1, \\ \mathcal{M}_{33} := a_3d_2 + a_2d_3,  ~~
\mathcal{M}_{12} := a_2d_0 + a_0d_2, ~~ \mathcal{M}_{23} := a_3d_1 + a_1d_3, ~~ \mathcal{M}_{13} := a_3d_0 + a_0d_3.
\end{split}
\end{equation*}

Consider any two-terminal damper-spring-inerter network $N$ whose two terminals are labeled as $a$ and $a'$.  A linear graph whose edges correspond to all the elements of $N$ is called the \emph{network graph} \cite{CWSL13}, \cite[pg.~28]{CWC19}.  Then, let $\mathcal{P}(a,a')$ denote the \emph{path} \cite[pg.~14]{SR61} whose  \emph{terminal vertices} \cite[pg.~14]{SR61} are $a$ and $a'$, and let $\mathcal{C}(a,a')$ denote the \emph{cut-set} \cite[pg.~28]{SR61} that separates the network graph into two connected subgraphs containing terminals $a$ and $a'$, respectively.
Furthermore, a path $\mathcal{P}(a,a')$   whose all edges correspond to springs (resp. inerters) is denoted as $k$-$\mathcal{P}(a,a')$ (resp. $b$-$\mathcal{P}(a,a')$);  a cut-set $\mathcal{C}(a,a')$  whose all edges correspond to springs (resp. inerters) is denoted as $k$-$\mathcal{C}(a,a')$ (resp. $b$-$\mathcal{C}(a,a')$).

In addition to network graphs, any two-terminal damper-spring-inerter network $N$ can be also described by a \emph{one-terminal-pair labeled graph} $\mathcal{N}$
(see \cite{WC20}, \cite[pg.~14]{SR61}), where each label designate a passive element regardless of the element value. The dampers, springs, and inerters are labeled as $c_i$, $k_i$, and $b_i$, respectively.

The  notations of the maps acting on the labeled graph are  as follows:\footnote{Such an approach of defining the notations $\text{GDu}$, $\text{Inv}$, and $\text{Dual}$ was suggested by Professor Rudolf E. Kalman (see \cite[Section~2.7]{CWC19}).}
\begin{itemize}
  \item[1.] $\text{GDu} :=$ Graph duality, which takes the graph into its dual (see \cite[Definition~3-12]{SR61}) without changing the labels.
  \item[2.] $\text{Inv} :=$ Inversion, which preserves the graph but interchanges the labels of springs $k_i$ and inerters $b_i$, that is, springs to inerters and inerters to springs, with their labels $k_i$ to $b_i$ and $b_i$ to $k_i$.
  \item[3.] $\text{Dual} :=$ Network duality of one-terminal-pair labeled graph $:= \text{GDu} \circ \text{Inv} = \text{Inv} \circ  \text{GDu}$.
\end{itemize}

An example to illustrate $\text{GDu}$, $\text{Inv}$, and $\text{Dual}$ can be referred to the four configurations in Fig.~\ref{fig: N2}. Their one-terminal-pair labeled graphs can be  denoted  as $\mathcal{N}_{2a}$,  $\mathcal{N}_{2b}$,  $\mathcal{N}_{2c}$, and  $\mathcal{N}_{2d}$, respectively, satisfying  $\mathcal{N}_{2b} = \text{Dual}(\mathcal{N}_{2a})$, $\mathcal{N}_{2c} = \text{Inv}(\mathcal{N}_{2a})$, and $\mathcal{N}_{2d} = \text{GDu}(\mathcal{N}_{2a})$.

As shown in \cite[Section~2.7]{CWC19}, \cite{WC20},    $Z(s)$ is realizable as the impedance of a network whose one-terminal-pair labeled graph is $\mathcal{N}$, if and only if $Z(s^{-1})$ is realizable as the impedance of a network whose one-terminal-pair labeled graph is $\text{Inv}(\mathcal{N})$ (\emph{principle of frequency inversion}), if and only if
$Z(s^{-1})$ is realizable as the admittance of a network whose one-terminal-pair labeled graph is $\text{GDu}(\mathcal{N})$ (\emph{principle of frequency-inverse duality}), if and only if
$Z(s)$ is realizable as the admittance of a network whose one-terminal-pair labeled graph is $\text{Dual}(\mathcal{N})$ (\emph{principle of duality}).

\section{Impedances With Poles or Zeros on Imaginary Axis or at Infinity} \label{sec: Poles or Zeros on Imaginary Axis}

\subsection{Pole or Zero at Origin or Infinity}

For the case when a bicubic impedance  $Z(s)$ in the form  \eqref{eq: three-degree impedance}, where  $a_i, d_j \geq 0$ for $i, j = 0, 1, 2, 3$, contains a zero at $s = 0$ (origin),  it is clear that $a_0 = 0$.
Then, the realization results of $Z(s)$ with no more than five elements have been presented in \cite{WJ19}.

Moreover,   the realization results for the case when $Z(s)$ contains a zero at $s = \infty$ can be directly obtained through $a_k \leftrightarrow a_{3-k}$  and $d_k \leftrightarrow d_{3-k}$ for $k = 0, 1$ (the \emph{principle of frequency inversion} \cite[Section~2.7]{CWC19}, \cite{WC20});   the realization results for the case when $Z(s)$ contains a pole at $s = 0$ can be directly obtained through $a_k \leftrightarrow d_k$ for $k = 0, 1, 2, 3$ (the \emph{principle of duality}~\cite[Section~2.7]{CWC19}, \cite{WC20});  the  realization results for the case when $Z(s)$ contains a pole at $s = \infty$ can be directly obtained through $a_k \leftrightarrow d_{3-k}$ for $k = 0, 1, 2, 3$ (the \emph{principle of frequency-inverse duality}~\cite[Section~2.7]{CWC19}, \cite{WC20}).

\subsection{Non-Zero Finite Pole or Zero on Imaginary Axis}

Consider a  bicubic impedance $Z(s)$ in the form of \eqref{eq: three-degree impedance}, where  $a_i, d_j \geq 0$ for $i, j = 0, 1, 2, 3$, and $\delta(Z(s)) = 3$.
The following theorem (Theorem~\ref{theorem: third-degree five-elements special case}) presents the realization results for the case when $Z(s)$ contains a finite pole or zero on $j \mathbb{R} \setminus \{0\}$.

\begin{theorem} \label{theorem: third-degree five-elements special case}
Consider a bicubic positive-real impedance $Z(s)$ in the form of \eqref{eq: three-degree impedance}, where  $a_i, d_j \geq 0$ for $i, j = 0, 1, 2, 3$, and  $\delta(Z(s)) = 3$. If $Z(s)$ contains a finite pole or zero on $j \mathbb{R} \setminus \{0\}$, then $Z(s)$ is realizable as a
series-parallel network containing three energy storage elements and no more than two dampers.
\end{theorem}
\begin{proof}
Assuming that $Z(s)$ contains a finite pole on $j \mathbb{R} \setminus \{0\}$,  the impedance can be expressed as
\begin{equation} \label{eq: finite pole on imaginary axis}
Z(s) = \frac{a_3 s^3 + a_2 s^2 + a_1 s + a_0}{(s^2 + \omega_1^2)(d_3 s+d_0/\omega_1^2)},
\end{equation}
where $\omega_1 > 0$. If $Z(s)$ is positive-real, then based on \cite[pg.~13]{CWC19}, it follows from \eqref{eq: finite pole on imaginary axis} that
\begin{equation} \label{eq: finite pole on imaginary axis 02}
Z(s) = \frac{2K_1 s}{s^2 + \omega_1^2} + \frac{c_1 s + c_0}{d_3 s + d_0/\omega_1^2} =: Z_1(s) + Z_2(s),
\end{equation}
where $K_1 > 0$, $c_0 \geq 0$, and $c_1 \geq 0$. It is clear that $Z_1(s)$ in \eqref{eq: finite pole on imaginary axis 02}
is realizable as the parallel connection of a spring and an inerter.
Based on the results in \cite{WCH14}, $Z_2(s)$ in \eqref{eq: finite pole on imaginary axis 02} is realizable as a series-parallel subnetwork containing one energy storage element and no more than two dampers. Therefore, $Z(s)$  is realizable as a  series-parallel network containing three energy storage elements and no more than two dampers. Together with the principle of duality ($a_k \leftrightarrow d_k$ for $k = 0, 1, 2, 3$), a similar discussion can be applied to the case when $Z(s)$ contains a finite zero on $j \mathbb{R} \setminus \{0\}$.
\end{proof}

\begin{remark}
It can be derived that a bicubic impedance $Z(s)$ in the form of \eqref{eq: three-degree impedance}, where $a_i, d_j > 0$ for $i, j = 0, 1, 2, 3$,
contains a finite pole (resp. zero) on $j \mathbb{R} \setminus \{0\}$ if and only if $\Delta_2 = 0$ (resp. $\Delta_1  = 0$).
\end{remark}

\section{Main Results} \label{sec: main results}

This section will investigate the realizations of a bicubic impedance $Z(s)$
in the form of \eqref{eq: three-degree impedance}  without any pole or zero on $j \mathbb{R} \cup  \infty$. Then, it is implied that $a_i, d_j > 0$ for $i, j = 0, 1, 2, 3$.

\subsection{Basic Lemmas}

The following two lemmas (Lemmas~\ref{lemma: graph constraint} and \ref{lemma: lossless passive networks}) present the topological restrictions of the network realizations of bicubic impedances.

\begin{lemma} \cite[Theorem 2]{Ses59} \label{lemma: graph constraint}
Consider a bicubic impedance $Z(s)$ in the form of \eqref{eq: three-degree impedance}, where $a_i, d_j > 0$ for $i, j = 0, 1, 2, 3$, $\delta(Z(s)) = 3$, and there is not any pole or zero on $j \mathbb{R} \cup  \infty$.
Then, the network graph of any network realizing $Z(s)$ cannot contain any of $k$-$\mathcal{P}(a,a')$, $b$-$\mathcal{P}(a,a')$, $k$-$\mathcal{C}(a,a')$, or $b$-$\mathcal{C}(a,a')$.
\end{lemma}

\begin{lemma} \label{lemma: lossless passive networks}
Consider a bicubic impedance $Z(s)$ in the form of \eqref{eq: three-degree impedance}, where $a_i, d_j > 0$ for $i, j = 0, 1, 2, 3$, $\delta(Z(s)) = 3$, and there is not any pole or zero on $j \mathbb{R} \cup  \infty$. Then, $Z(s)$ cannot be realized as the series or parallel connection of a lossless subnetwork\footnote{A lossless network only contains energy storage elements (springs or inerters).} and a general passive subnetwork.
\end{lemma}
\begin{proof}
It can be verified that the impedance $Z(s)$ must contain at least one pole or zero on $j \mathbb{R} \cup  \infty$ due to the lossless subnetwork, which contradicts the assumption.
\end{proof}

The following lemma (Lemma~\ref{lemma: number of elements}) presents the least number of elements and energy storage elements needed to realize a bicubic impedance.

\begin{lemma} \label{lemma: number of elements}
Consider a bicubic impedance $Z(s)$ in the form of \eqref{eq: three-degree impedance}, where $a_i, d_j > 0$ for $i, j = 0, 1, 2, 3$, $\delta(Z(s)) = 3$, and there is not any pole or zero on $j \mathbb{R} \cup  \infty$. Then, any network realizing $Z(s)$ must contain at least five elements, where the number of energy storage elements is at least three.
\end{lemma}
\begin{proof}
Since it is shown in \cite[pg.~370]{AV73} that the   McMillan degree of a given impedance is equal to the least number of energy storage elements, any network realizing $Z(s)$ must contain at least three energy storage elements. It is clear that any non-series-parallel configuration must contain at least five elements.
For any series-parallel  realization of $Z(s)$, the network  can be decomposed as a parallel or series connection of two series-parallel subnetworks. By Lemma~\ref{lemma: lossless passive networks}, each of these two subnetworks must contain at least one damper, which implies that any series-parallel network realizing $Z(s)$ also contains at least five elements.
\end{proof}


\subsection{Realizations as Five-Element Series-Parallel Networks}

In Lemma~\ref{lemma: number of elements}, it is shown that any five-element series-parallel network realizing a bicubic impedance \eqref{eq: three-degree impedance} without any pole or zero on $j \mathbb{R} \cup  \infty$ contains the least number of elements. This subsection will derive the realization results of $Z(s)$ as five-element series-parallel networks. The following Lemmas~\ref{lemma: third-order configurations}--\ref{lemma: realizability condition of N6a} will be utilized to prove Theorem~\ref{theorem: third-degree five-element series-parallel network}.

\begin{lemma}  \label{lemma: third-order configurations}
Consider a bicubic impedance $Z(s)$ in the form of \eqref{eq: three-degree impedance}, where $a_i, d_j > 0$ for $i, j = 0, 1, 2, 3$, $\delta(Z(s)) = 3$, and there is not any pole or zero on $j \mathbb{R} \cup  \infty$. Then, $Z(s)$ is realizable as a five-element series-parallel network, if and only if $Z(s)$ is realizable as one of the configurations in Figs.~\ref{fig: N1}--\ref{fig: N6}.
\end{lemma}
\begin{proof}
See Appendix~A for details, where Lemmas~\ref{lemma: graph constraint}--\ref{lemma: number of elements} are utilized in the proof.
\end{proof}

\begin{figure}[thpb]
      \centering
      \subfigure[]{
      \includegraphics[scale=0.9]{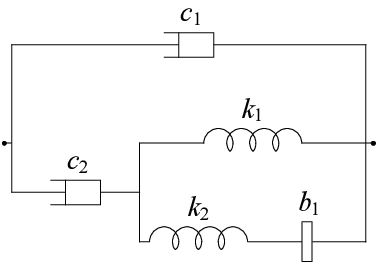}
      \label{fig: N1a}}
      \subfigure[]{
      \includegraphics[scale=0.9]{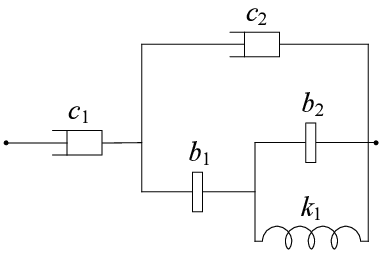}
      \label{fig: N1b}}
      \caption{Five-element series-parallel configurations that can realize the bicubic impedance $Z(s)$ in \eqref{eq: three-degree impedance}, whose \emph{one-terminal-pair labeled graphs} \cite{WC20}, \cite[pg.~14]{SR61}  are (a) $\mathcal{N}_{1a}$ and (b) $\mathcal{N}_{1b}$, respectively, satisfying $\mathcal{N}_{1b} = \text{Dual}(\mathcal{N}_{1a})$. Moreover, the configurations whose one-terminal-pair labeled graphs are $\text{Inv}(\mathcal{N}_{1a})$ and $\text{GDu}(\mathcal{N}_{1a})$ can always be equivalent to the configurations in (b) and (a), respectively.  }
      \label{fig: N1}
\end{figure}

\begin{figure}[thpb]
      \centering
      \subfigure[]{
      \includegraphics[scale=0.9]{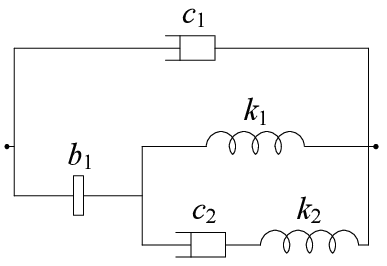}
      \label{fig: N2a}}
      \subfigure[]{
      \includegraphics[scale=0.9]{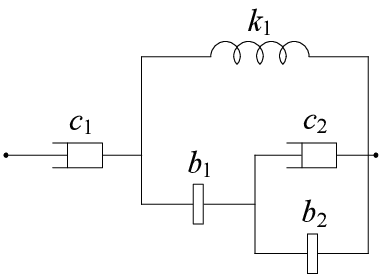}
      \label{fig: N2b}}
      \subfigure[]{
      \includegraphics[scale=0.9]{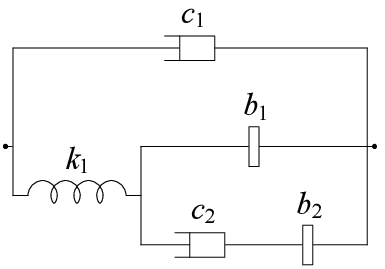}
      \label{fig: N2c}}
      \subfigure[]{
      \includegraphics[scale=0.9]{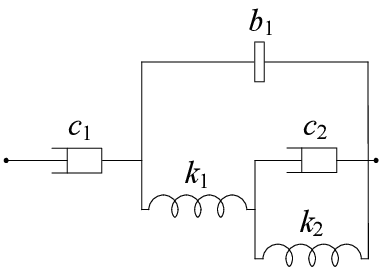}
      \label{fig: N2d}}
      \caption{Five-element series-parallel configurations that can realize the bicubic impedance $Z(s)$ in \eqref{eq: three-degree impedance}, whose one-terminal-pair labeled graphs are (a) $\mathcal{N}_{2a}$, (b) $\mathcal{N}_{2b}$, (c) $\mathcal{N}_{2c}$, and (d) $\mathcal{N}_{2d}$, respectively, satisfying $\mathcal{N}_{2b} = \text{Dual}(\mathcal{N}_{2a})$, $\mathcal{N}_{2c} = \text{Inv}(\mathcal{N}_{2a})$, and $\mathcal{N}_{2d} = \text{GDu}(\mathcal{N}_{2a})$.}
      \label{fig: N2}
\end{figure}

\begin{figure}[thpb]
      \centering
      \subfigure[]{
      \includegraphics[scale=0.9]{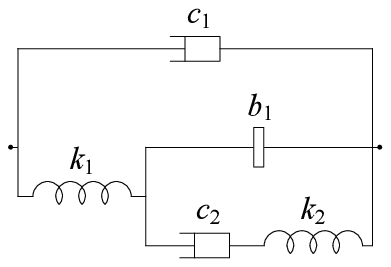}
      \label{fig: N3a}}
      \subfigure[]{
      \includegraphics[scale=0.9]{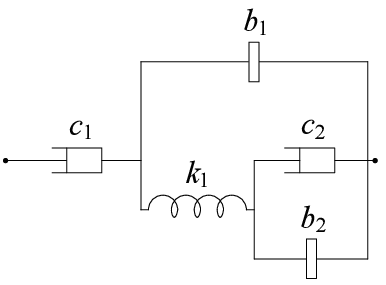}
      \label{fig: N3b}}
      \subfigure[]{
      \includegraphics[scale=0.9]{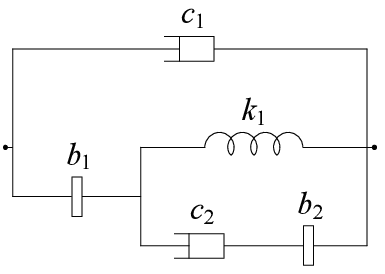}
      \label{fig: N3c}}
      \subfigure[]{
      \includegraphics[scale=0.9]{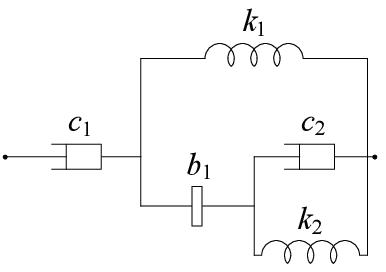}
      \label{fig: N3d}}
      \caption{Five-element series-parallel configurations that can realize the bicubic impedance $Z(s)$ in \eqref{eq: three-degree impedance}, whose one-terminal-pair labeled graphs are (a) $\mathcal{N}_{3a}$, (b) $\mathcal{N}_{3b}$, (c) $\mathcal{N}_{3c}$, and (d) $\mathcal{N}_{3d}$, respectively, satisfying $\mathcal{N}_{3b} = \text{Dual}(\mathcal{N}_{3a})$, $\mathcal{N}_{3c} = \text{Inv}(\mathcal{N}_{3a})$, and $\mathcal{N}_{3d} = \text{GDu}(\mathcal{N}_{3a})$.}
      \label{fig: N3}
\end{figure}

\begin{figure}[thpb]
      \centering
      \subfigure[]{
      \includegraphics[scale=0.9]{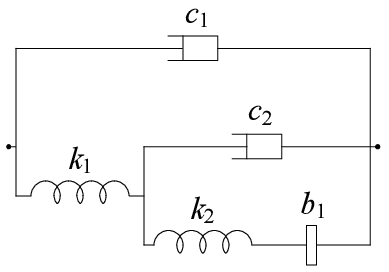}
      \label{fig: N4a}}
      \subfigure[]{
      \includegraphics[scale=0.9]{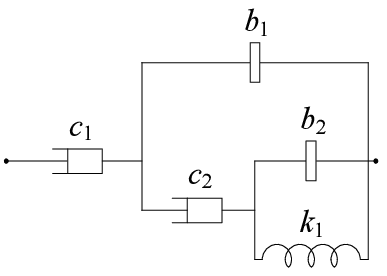}
      \label{fig: N4b}}
      \subfigure[]{
      \includegraphics[scale=0.9]{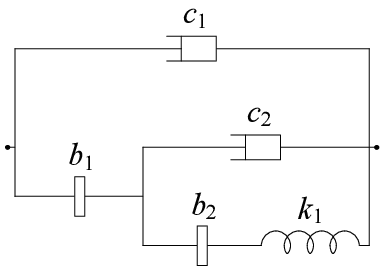}
      \label{fig: N4c}}
      \subfigure[]{
      \includegraphics[scale=0.9]{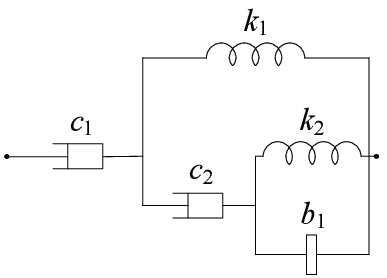}
      \label{fig: N4d}}
      \caption{Five-element series-parallel configurations that can realize the bicubic impedance $Z(s)$ in \eqref{eq: three-degree impedance}, whose one-terminal-pair labeled graphs are (a) $\mathcal{N}_{4a}$, (b) $\mathcal{N}_{4b}$, (c) $\mathcal{N}_{4c}$, and (d) $\mathcal{N}_{4d}$, respectively, satisfying $\mathcal{N}_{4b} = \text{Dual}(\mathcal{N}_{4a})$, $\mathcal{N}_{4c} = \text{Inv}(\mathcal{N}_{4a})$, and $\mathcal{N}_{4d} = \text{GDu}(\mathcal{N}_{4a})$.}
      \label{fig: N4}
\end{figure}

\begin{figure}[thpb]
      \centering
      \subfigure[]{
      \includegraphics[scale=0.9]{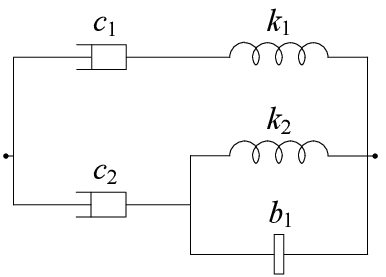}
      \label{fig: N5a}}
      \subfigure[]{
      \includegraphics[scale=0.9]{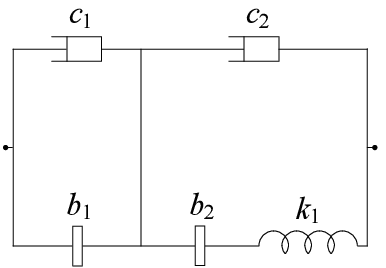}
      \label{fig: N5b}}
      \subfigure[]{
      \includegraphics[scale=0.9]{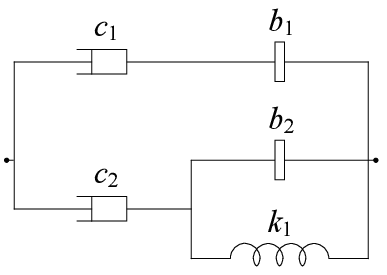}
      \label{fig: N5c}}
      \subfigure[]{
      \includegraphics[scale=0.9]{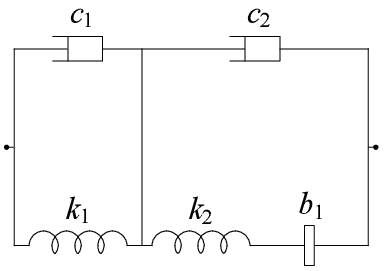}
      \label{fig: N5d}}
      \caption{Five-element series-parallel configurations that can realize the bicubic impedance $Z(s)$ in \eqref{eq: three-degree impedance}, whose one-terminal-pair labeled graphs are (a) $\mathcal{N}_{5a}$, (b) $\mathcal{N}_{5b}$, (c) $\mathcal{N}_{5c}$, and (d) $\mathcal{N}_{5d}$, respectively, satisfying $\mathcal{N}_{5b} = \text{Dual}(\mathcal{N}_{5a})$, $\mathcal{N}_{5c} = \text{Inv}(\mathcal{N}_{5a})$, and $\mathcal{N}_{5d} = \text{GDu}(\mathcal{N}_{5a})$.}
      \label{fig: N5}
\end{figure}

\begin{figure}[thpb]
      \centering
      \subfigure[]{
      \includegraphics[scale=0.9]{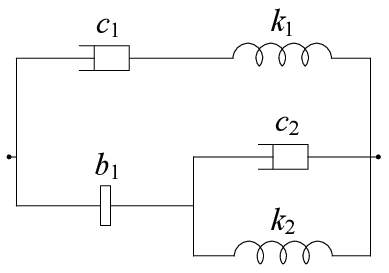}
      \label{fig: N6a}}
      \subfigure[]{
      \includegraphics[scale=0.9]{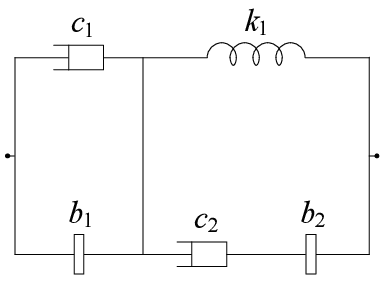}
      \label{fig: N6b}}
      \subfigure[]{
      \includegraphics[scale=0.9]{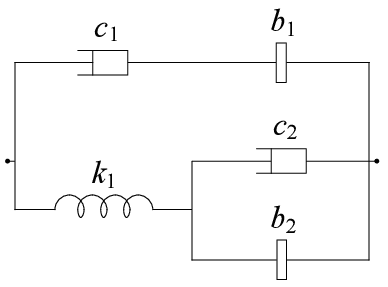}
      \label{fig: N6c}}
      \subfigure[]{
      \includegraphics[scale=0.9]{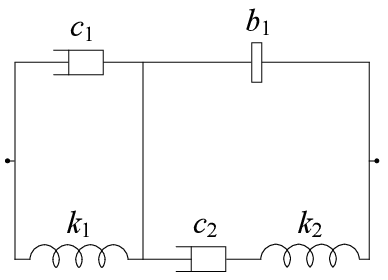}
      \label{fig: N6d}}
      \caption{Five-element series-parallel configurations that can realize the bicubic impedance $Z(s)$ in \eqref{eq: three-degree impedance}, whose one-terminal-pair labeled graphs are (a) $\mathcal{N}_{6a}$, (b) $\mathcal{N}_{6b}$, (c) $\mathcal{N}_{6c}$, and (d) $\mathcal{N}_{6d}$, respectively, satisfying $\mathcal{N}_{6b} = \text{Dual}(\mathcal{N}_{6a})$, $\mathcal{N}_{6c} = \text{Inv}(\mathcal{N}_{6a})$, and $\mathcal{N}_{6d} = \text{GDu}(\mathcal{N}_{6a})$.}
      \label{fig: N6}
\end{figure}

The above lemma (Lemma~\ref{lemma: third-order configurations}) presents a set of realization configurations in Figs.~\ref{fig: N1}--\ref{fig: N6} that can cover all the cases of five-element series-parallel realizations. Then, the realizability conditions of these configurations need to be derived, which will be shown in the following Lemmas~\ref{lemma: realizability condition of N1a}--\ref{lemma: realizability condition of N6a}.

\begin{lemma}   \label{lemma: realizability condition of N1a}
Consider a bicubic impedance $Z(s)$ in the form of \eqref{eq: three-degree impedance}, where $a_i, d_j > 0$ for $i, j = 0, 1, 2, 3$, $\delta(Z(s)) = 3$, and there is not any pole or zero on $j \mathbb{R} \cup  \infty$. Then, $Z(s)$ is realizable as one of the five-element series-parallel configurations in Fig.~\ref{fig: N1} (whose  one-terminal-pair labeled graph is $\mathcal{N}_{1a}$ or $\mathcal{N}_{1b}$), if and only if $\mathcal{B}_{13} \neq 0$,
$\mathcal{B}_{12} = 0$,   $\mathcal{B}_{23} = 0$,  and    $\Delta_1 > 0$.
Moreover, if $\mathcal{B}_{13} > 0$, $\mathcal{B}_{12} = 0$, $\mathcal{B}_{23} = 0$,  and $\Delta_1 > 0$, then
$Z(s)$ is realizable as   in Fig.~\ref{fig: N1a} whose element values can be expressed as
\begin{equation}  \label{eq: N1a element values}
\begin{split}
c_1 = \frac{d_3}{a_3}, ~~ c_2 = \frac{\mathcal{B}_{13}}{a_0 a_3}, ~~
k_1 = \frac{d_3 \mathcal{B}_{13}}{a_3^2 d_1},  ~~
k_2 = \frac{\mathcal{B}_{13} \Delta_1}{a_0a_1a_3^2}, ~~
b_1 = \frac{\mathcal{B}_{13} \Delta_1}{a_0a_1^2a_3}.
\end{split}
\end{equation}
\end{lemma}
\begin{proof}
See Appendix~\ref{appendix: N1a} for details.
\end{proof}

\begin{lemma}   \label{lemma: realizability condition of N2a}
Consider a bicubic impedance $Z(s)$ in the form of \eqref{eq: three-degree impedance}, where $a_i, d_j > 0$ for $i, j = 0, 1, 2, 3$, $\delta(Z(s)) = 3$, and there is not any pole or zero on $j \mathbb{R} \cup  \infty$. Then, $Z(s)$ is realizable as one of the five-element series-parallel configurations in Fig.~\ref{fig: N2} (whose  one-terminal-pair labeled graph is $\mathcal{N}_{2a}$, $\mathcal{N}_{2b}$, $\mathcal{N}_{2c}$, or $\mathcal{N}_{2d}$), if and only if one of the two conditions holds:
\begin{enumerate}
  \item[1.] $\mathcal{B}_{13} = 0$, $\Delta_1 > 0$, and either $a_0 \mathcal{B}_{33} = a_1 \mathcal{B}_{23} > 0$ or
$a_3 \mathcal{B}_{11} = a_2 \mathcal{B}_{12} < 0$ holds;
  \item[2.] $\mathcal{B}_{13} = 0$, $\Delta_2 > 0$, and either $d_0 \mathcal{B}_{33} = d_1 \mathcal{B}_{23} < 0$ or $d_3 \mathcal{B}_{11} = d_2 \mathcal{B}_{12} > 0$ holds.
\end{enumerate}
Moreover, if $\mathcal{B}_{13} = 0$, $\Delta_1 > 0$, and   $a_0 \mathcal{B}_{33} = a_1 \mathcal{B}_{23} > 0$,  then
$Z(s)$ is realizable as  in Fig.~\ref{fig: N2a} whose element values can be expressed as
\begin{equation}  \label{eq: N2a element values}
\begin{split}
c_1 = \frac{d_3}{a_3}, ~~
c_2 = \frac{\mathcal{B}_{33} \Delta_1}{a_1a_2^2a_3}, ~~
k_1 = \frac{d_0 \mathcal{B}_{33}}{a_1a_2d_3}, ~~
k_2 = \frac{\mathcal{B}_{33} \Delta_1}{a_1a_2a_3^2}, ~~
b_1 = \frac{\mathcal{B}_{33}}{a_1a_3}.
\end{split}
\end{equation}
\end{lemma}
\begin{proof}
The method is similar to that of Lemma~\ref{lemma: realizability condition of N1a}, which can be referred to \cite[Section~2]{WC_sup} for details.
\end{proof}

\begin{lemma}   \label{lemma: realizability condition of N3a}
Consider a bicubic impedance $Z(s)$ in the form of \eqref{eq: three-degree impedance}, where $a_i, d_j > 0$ for $i, j = 0, 1, 2, 3$, $\delta(Z(s)) = 3$, and there is not any pole or zero on $j \mathbb{R} \cup  \infty$.  Then, $Z(s)$ is realizable as one of the five-element series-parallel configurations in Fig.~\ref{fig: N3} (whose  one-terminal-pair labeled graph is $\mathcal{N}_{3a}$, $\mathcal{N}_{3b}$, $\mathcal{N}_{3c}$, or $\mathcal{N}_{3d}$), if and only if one of the following four conditions holds:
\begin{enumerate}
  \item[1.] $\mathcal{B}_{33} \Delta_1 = a_2a_3 \mathcal{B}_{13} > 0$ and $a_2 \mathcal{B}_{33} = a_3 \mathcal{B}_{23} > 0$;
  \item[2.] $\mathcal{B}_{33} \Delta_2 = d_2d_3 \mathcal{B}_{13} < 0$ and $d_2 \mathcal{B}_{33} = d_3 \mathcal{B}_{23} < 0$;
  \item[3.] $\mathcal{B}_{11} \Delta_1 = a_0a_1 \mathcal{B}_{13} < 0$ and $a_1 \mathcal{B}_{11} = a_0 \mathcal{B}_{12} < 0$;
  \item[4.] $\mathcal{B}_{11} \Delta_2 = d_0d_1 \mathcal{B}_{13} > 0$ and $d_1 \mathcal{B}_{11} = d_0 \mathcal{B}_{12} > 0$.
\end{enumerate}
Moreover, if Condition~1 holds, then $Z(s)$ is realizable as the configuration in  Fig.~\ref{fig: N3a} whose element values can be expressed as
\begin{equation} \label{eq: N3a element values}
\begin{split}
c_1  = \frac{d_3}{a_3}, ~~
c_2  = \frac{\mathcal{B}_{13}}{a_0 a_3},  ~~
k_1  = \frac{\mathcal{B}_{23}}{a_2 a_3},  ~~
k_2  = \frac{a_2 \mathcal{B}_{13}}{a_0 a_3^2},  ~~
b_1  = \frac{\mathcal{B}_{23}}{a_0 a_3}.
\end{split}
\end{equation}
\end{lemma}
\begin{proof}
The method is similar to that of Lemma~\ref{lemma: realizability condition of N1a}, which can be referred to \cite[Section~3]{WC_sup} for details.
\end{proof}

\begin{lemma}   \label{lemma: realizability condition of N4a}
Consider a bicubic impedance $Z(s)$ in the form of \eqref{eq: three-degree impedance}, where $a_i, d_j > 0$ for $i, j = 0, 1, 2, 3$, $\delta(Z(s)) = 3$, and there is not any pole or zero on $j \mathbb{R} \cup  \infty$. Then, $Z(s)$ is realizable as one of the five-element series-parallel configurations in Fig.~\ref{fig: N4} (whose  one-terminal-pair labeled graph is $\mathcal{N}_{4a}$, $\mathcal{N}_{4b}$, $\mathcal{N}_{4c}$, or $\mathcal{N}_{4d}$), if and only if
one of the following four conditions holds:
\begin{enumerate}
  \item[1.] $\mathcal{B}_{13} \Delta_1  = a_1^2 \mathcal{B}_{23} > 0$ and $a_1 \mathcal{B}_{33} = a_3 \mathcal{B}_{13} > 0$;
  \item[2.] $\mathcal{B}_{13} \Delta_2  = d_1^2 \mathcal{B}_{23} < 0$ and $d_1 \mathcal{B}_{33} = d_3 \mathcal{B}_{13} < 0$;
  \item[3.] $\mathcal{B}_{13} \Delta_1  = a_2^2 \mathcal{B}_{12} < 0$ and $a_2 \mathcal{B}_{11} = a_0 \mathcal{B}_{13} < 0$;
  \item[4.] $\mathcal{B}_{13} \Delta_2  = d_2^2 \mathcal{B}_{12} > 0$ and $d_2 \mathcal{B}_{11} = d_0 \mathcal{B}_{13} > 0$.
\end{enumerate}
Moreover, if Condition~1 holds, then $Z(s)$ is realizable as  in Fig.~\ref{fig: N4a} whose element values can be expressed as
\begin{equation}  \label{eq: N4a element values}
\begin{split}
c_1  = \frac{d_3}{a_3}, ~~
c_2  = \frac{\mathcal{B}_{13}}{a_0 a_3}, ~~
k_1  = \frac{\mathcal{B}_{13}}{a_1 a_3},  ~~
k_2  = \frac{a_1 \mathcal{B}_{23}}{a_0 a_3^2},  ~~
b_1  = \frac{\mathcal{B}_{23}}{a_0 a_3}.
\end{split}
\end{equation}
\end{lemma}
\begin{proof}
The method is similar to that of Lemma~\ref{lemma: realizability condition of N1a}, which can be referred to \cite[Section~4]{WC_sup} for details.
\end{proof}

\begin{lemma}   \label{lemma: realizability condition of N5a}
Consider a bicubic impedance $Z(s)$ in the form of \eqref{eq: three-degree impedance}, where $a_i, d_j > 0$ for $i, j = 0, 1, 2, 3$, $\delta(Z(s)) = 3$, and there is not any pole or zero on $j \mathbb{R} \cup  \infty$. Then, $Z(s)$ is realizable as one of the five-element series-parallel configurations in Fig.~\ref{fig: N5} (whose  one-terminal-pair labeled graph is $\mathcal{N}_{5a}$, $\mathcal{N}_{5b}$, $\mathcal{N}_{5c}$, or $\mathcal{N}_{5d}$),
if and only if one of the following four conditions holds:
\begin{enumerate}
  \item[1.] $a_3^2 d_0^2 \Delta_2 = d_2^2 \mathcal{B}_{12} \mathcal{B}_{13} > 0$ and
$a_0 d_2^2 \mathcal{B}_{12} = a_3 d_0^2 (a_1d_2 - a_3d_0) > 0$;
  \item[2.] $a_0^2 d_3^2 \Delta_1 = a_2^2 \mathcal{B}_{12} \mathcal{B}_{13} > 0$ and
$a_2^2 d_0 \mathcal{B}_{12} = a_0^2 d_3 (a_0d_3 - a_2d_1) < 0$;
  \item[3.] $a_0^2 d_3^2 \Delta_2 = d_1^2 \mathcal{B}_{23} \mathcal{B}_{13} > 0$ and
$a_3 d_1^2 \mathcal{B}_{23} = a_0 d_3^2 (a_0d_3 - a_2d_1) < 0$;
  \item[4.] $a_3^2 d_0^2 \Delta_1 = a_1^2 \mathcal{B}_{23} \mathcal{B}_{13} > 0$ and
$a_1^2 d_3 \mathcal{B}_{23} = a_3^2 d_0 (a_1d_2 - a_3d_0) > 0$.
\end{enumerate}
Moreover, if Condition~1 holds, then $Z(s)$ is realizable as the configuration in Fig.~\ref{fig: N5a} whose element values can be expressed as
\begin{equation}  \label{eq: N5a element values}
\begin{split}
c_1 = \frac{\mathcal{B}_{13}}{a_0 a_3}, ~~
c_2 = \frac{d_3}{a_3}, ~~
k_1 = \frac{d_2 \mathcal{B}_{13}}{a_3^2d_0},  ~~
k_2 = \frac{d_2d_3 \mathcal{B}_{13}}{a_3^2 \Delta_2}, ~~
b_1 = \frac{d_2^2d_3 \mathcal{B}_{13}}{a_3^2d_0 \Delta_2}.
\end{split}
\end{equation}
\end{lemma}
\begin{proof}
The method is similar to that of Lemma~\ref{lemma: realizability condition of N1a}, which can be referred to \cite[Section~5]{WC_sup} for details.
\end{proof}

Define the notation $\zeta_1$ as
\begin{equation} \label{eq: zeta}
\zeta_1 := \frac{a_0 (\mathcal{B}_{23} \pm \sqrt{\mathcal{M}_{23}^2 - 4a_0a_3d_2d_3})}{2 a_3}.
\end{equation}
Then, the notations $\zeta_2$, $\zeta_3$,  and $\zeta_4$ can be obtained from $\zeta_1$ according to the conversion $a_k \leftrightarrow d_k$ for $k = 0, 1, 2, 3$, the conversion $a_k \leftrightarrow a_{3-k}$  and $d_k \leftrightarrow d_{3-k}$ for $k = 0, 1$,  and the conversion $a_k \leftrightarrow d_{3-k}$ for $k = 0, 1, 2, 3$, respectively. Then, the following lemma can be formulated.

\begin{lemma}  \label{lemma: realizability condition of N6a}
Consider a bicubic impedance $Z(s)$ in the form of \eqref{eq: three-degree impedance}, where $a_i, d_j > 0$ for $i, j = 0, 1, 2, 3$, $\delta(Z(s)) = 3$, and there is not any pole or zero on $j \mathbb{R} \cup  \infty$. Then, $Z(s)$ is realizable as one of the five-element series-parallel configurations in Fig.~\ref{fig: N6} (whose  one-terminal-pair labeled graph is $\mathcal{N}_{6a}$, $\mathcal{N}_{6b}$, $\mathcal{N}_{6c}$, or $\mathcal{N}_{6d}$),
if and only if one of the following four conditions holds:
\begin{enumerate}
  \item[1.] $0 < \zeta_1 < \min\{ a_1d_0, a_0d_1 \}$, $\mathcal{M}_{23}^2 - 4a_0a_3d_2d_3 \geq 0$,
$a_3 \zeta_1^2 - a_0 \mathcal{B}_{23} \zeta_1 - a_0^2d_3(a_1d_1-a_0d_2) = 0$, and
$\zeta_1^3 - \mathcal{M}_{11} \zeta_1^2 + a_0a_1d_0d_1 \zeta_1 - a_0^3 d_0^2 d_3 = 0$;
  \item[2.] $0 < \zeta_2 < \min\{ a_1d_0, a_0d_1 \}$, $\mathcal{M}_{23}^2 - 4a_2a_3d_0d_3 \geq 0$,
$d_3 \zeta_2^2 + d_0 \mathcal{B}_{23} \zeta_2 - a_3 d_0^2 (a_1d_1-a_2d_0) = 0$, and
$\zeta_2^3 - \mathcal{M}_{11} \zeta_2^2 + a_0a_1d_0d_1 \zeta_2 - a_0^2a_3 d_0^3 = 0$;
  \item[3.] $0 < \zeta_3 < \min\{ a_3d_2 , a_2d_3 \}$, $\mathcal{M}_{12}^2 - 4a_0a_3d_0d_1 \geq 0$,
$a_0 \zeta_3^2 + a_3 \mathcal{B}_{12} \zeta_3 - a_3^2d_0(a_2d_2-a_3d_1) = 0$, and
$\zeta_3^3 - \mathcal{M}_{33} \zeta_3^2 + a_2a_3d_2d_3 \zeta_3 -
a_3^3 d_0 d_3^2 = 0$;
  \item[4.] $0 < \zeta_4 < \min\{ a_3d_2 , a_2d_3 \}$, $\mathcal{M}_{12}^2 - 4a_0a_1d_0d_3 \geq 0$,
$d_0 \zeta_4^2 - d_3 \mathcal{B}_{12} \zeta_4 - a_0 d_3^2 (a_2d_2-a_1d_3) = 0$, and
$\zeta_4^3 - \mathcal{M}_{33} \zeta_4^2 +a_2 a_3 d_2 d_3 \zeta_4 - a_0 a_3^2 d_3^3 = 0$.
\end{enumerate}
Moreover, if Condition~1 holds, then $Z(s)$ is realizable as the configuration in Fig.~\ref{fig: N6a} whose element values can be expressed as
\begin{equation}  \label{eq: N6a element values}
\begin{split}
c_1 = \frac{d_0}{a_0}, ~~
c_2 = \frac{d_3}{a_3}, ~~
k_1 = \frac{d_0^2}{a_1d_0 - \zeta_1}, ~~
k_2 = \frac{a_0d_0d_3}{a_3 \zeta_1}, ~~
b_1 = \frac{a_0d_1 - \zeta_1}{a_0^2}.
\end{split}
\end{equation}
\end{lemma}
\begin{proof}
See Appendix~\ref{appendix: N6a} for details.
\end{proof}

Then, combining Lemmas~\ref{lemma: third-order configurations}--\ref{lemma: realizability condition of N6a}, the following   Theorem~\ref{theorem: third-degree five-element series-parallel network} can be proved, which presents a necessary and sufficient condition for a bicubic impedance $Z(s)$ in the form of \eqref{eq: three-degree impedance} to be realizable as a five-element series-parallel network.

\begin{theorem}  \label{theorem: third-degree five-element series-parallel network}
Consider a bicubic impedance $Z(s)$ in the form of \eqref{eq: three-degree impedance}, where $a_i, d_j > 0$ for $i, j = 0, 1, 2, 3$, $\delta(Z(s)) = 3$, and there is not any pole or zero on $j \mathbb{R} \cup  \infty$. Then, $Z(s)$ is realizable as a five-element series-parallel network, if and only if one of the conditions in Lemmas~\ref{lemma: realizability condition of N1a}--\ref{lemma: realizability condition of N6a} holds.
\end{theorem}
\begin{proof}
By Lemma~\ref{lemma: third-order configurations}, the bicubic impedance $Z(s)$ in this theorem  is realizable as a five-element series-parallel network if and only if $Z(s)$ is realizable as one of the configurations in Figs.~\ref{fig: N1}--\ref{fig: N6}. Since the necessary and sufficient conditions for  $Z(s)$  to be realizable as the configurations in Figs.~\ref{fig: N1}--\ref{fig: N6} are shown in Lemmas~\ref{lemma: realizability condition of N1a}--\ref{lemma: realizability condition of N6a}, this theorem can be proved.
\end{proof}

\subsection{Realizations as Five-Element Non-Series-Parallel Networks}

For the realizations as five-element non-series-parallel networks, the following Lemmas~\ref{lemma: third-order non-series-parallel configurations}--\ref{lemma: realizability condition of N11} will be utilized to prove Theorem~\ref{theorem: third-degree five-element non-series-parallel network}.

\begin{lemma}  \label{lemma: third-order non-series-parallel configurations}
Consider a bicubic impedance $Z(s)$ in the form of \eqref{eq: three-degree impedance}, where $a_i, d_j > 0$ for $i, j = 0, 1, 2, 3$, $\delta(Z(s)) = 3$, and there is not any pole or zero on $j \mathbb{R} \cup  \infty$. $Z(s)$ is realizable as a five-element non-series-parallel network, if and only if $Z(s)$ is realizable as one of the configurations in Figs.~\ref{fig: N7}--\ref{fig: N11}.
\end{lemma}
\begin{proof}
See Appendix~\ref{appendix: non-series-parallel configurations} for details, where Lemmas~\ref{lemma: graph constraint} and \ref{lemma: number of elements} are utilized in the proof.
\end{proof}

The above lemma (Lemma~\ref{lemma: third-order non-series-parallel configurations}) presents a set of realization configurations in Figs.~\ref{fig: N7}--\ref{fig: N11} that can cover all the cases of five-element non-series-parallel realizations. Then, the realizability conditions of these configurations need to be derived, which will be shown in the following Lemmas~\ref{lemma: realizability condition of N7a}--\ref{lemma: realizability condition of N11}.

\begin{figure}[thpb]
      \centering
      \subfigure[]{
      \includegraphics[scale=0.9]{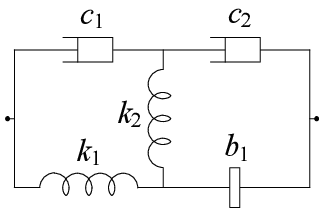}
      \label{fig: N7a}}
      \subfigure[]{
      \includegraphics[scale=0.9]{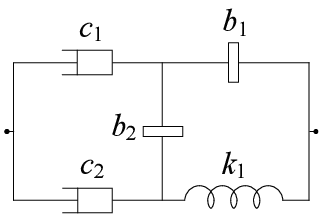}
      \label{fig: N7b}}
      \subfigure[]{
      \includegraphics[scale=0.9]{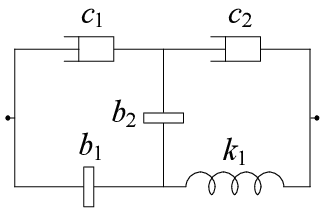}
      \label{fig: N7c}}
      \subfigure[]{
      \includegraphics[scale=0.9]{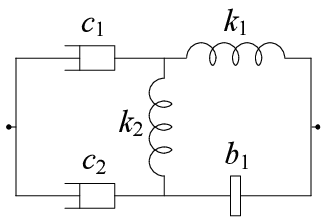}
      \label{fig: N7d}}
      \caption{Five-element non-series-parallel configurations that can realize the bicubic impedance $Z(s)$ in \eqref{eq: three-degree impedance}, whose one-terminal-pair labeled graphs are (a) $\mathcal{N}_{7a}$, (b) $\mathcal{N}_{7b}$, (c) $\mathcal{N}_{7c}$, and (d) $\mathcal{N}_{7d}$, respectively, satisfying $\mathcal{N}_{7b} = \text{Dual}(\mathcal{N}_{7a})$, $\mathcal{N}_{7c} = \text{Inv}(\mathcal{N}_{7a})$, and $\mathcal{N}_{7d} = \text{GDu}(\mathcal{N}_{7a})$.}
      \label{fig: N7}
\end{figure}

\begin{figure}[thpb]
      \centering
      \subfigure[]{
      \includegraphics[scale=0.9]{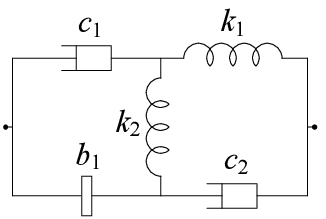}
      \label{fig: N8a}}
      \subfigure[]{
      \includegraphics[scale=0.9]{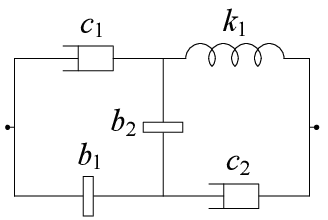}
      \label{fig: N8b}}
      \caption{Five-element non-series-parallel configurations that can realize the bicubic impedance $Z(s)$ in \eqref{eq: three-degree impedance}, whose one-terminal-pair labeled graphs are (a) $\mathcal{N}_{8a}$ and (b) $\mathcal{N}_{8b}$, respectively,  satisfying $\mathcal{N}_{8b} = \text{Dual}(\mathcal{N}_{8a})$, where $\text{Inv}(\mathcal{N}_{8a}) = \text{Dual}(\mathcal{N}_{8a})$
      and $\text{GDu}(\mathcal{N}_{8a}) = \mathcal{N}_{8a}$.}
      \label{fig: N8}
\end{figure}

\begin{lemma}  \label{lemma: realizability condition of N7a}
Consider a bicubic impedance $Z(s)$ in the form of \eqref{eq: three-degree impedance}, where $a_i, d_j > 0$ for $i, j = 0, 1, 2, 3$, $\delta(Z(s)) = 3$, and there is not any pole or zero on $j \mathbb{R} \cup  \infty$. Then, $Z(s)$ is realizable as one of the five-element non-series-parallel configurations in Fig.~\ref{fig: N7} (whose  one-terminal-pair labeled graph is $\mathcal{N}_{7a}$, $\mathcal{N}_{7b}$, $\mathcal{N}_{7c}$, or $\mathcal{N}_{7d}$), if and only if one of the following four conditions holds:
\begin{enumerate}
  \item[1.] $\mathcal{B}_{13} > 0$, $\mathcal{B}_{23} > 0$, $\mathcal{B}_{13}
  (a_2d_2 - \mathcal{B}_{23})  - a_2^2 d_0 d_3 = 0$, and
      $\mathcal{B}_{13} \mathcal{B}_{23} \Delta_1   - a_2^2a_3^2d_0^2 = 0$;
  \item[2.] $\mathcal{B}_{13} < 0$, $\mathcal{B}_{23} < 0$, $\mathcal{B}_{13}
  (a_2d_2 + \mathcal{B}_{23})  + a_0a_3d_2^2 = 0$, and
      $\mathcal{B}_{13}\mathcal{B}_{23} \Delta_2  - a_0^2d_2^2d_3^2 = 0$;
  \item[3.] $\mathcal{B}_{13} < 0$, $\mathcal{B}_{12} < 0$, $\mathcal{B}_{13}
  (a_1d_1 + \mathcal{B}_{12})  + a_1^2d_0d_3 = 0$, and
      $\mathcal{B}_{13}\mathcal{B}_{12} \Delta_1  - a_0^2 a_1^2 d_3^2 = 0$;
  \item[4.] $\mathcal{B}_{13} > 0$, $\mathcal{B}_{12} > 0$, $\mathcal{B}_{13}
  (a_1d_1 - \mathcal{B}_{12})  - a_0a_3d_1^2 = 0$, and
      $\mathcal{B}_{13}\mathcal{B}_{12} \Delta_2  - a_3^2d_1^2d_0^2 = 0$.
\end{enumerate}
Moreover, if Condition~1 holds, then $Z(s)$ is realizable as the configuration in Fig.~\ref{fig: N7a} whose element values can be expressed as
\begin{equation}  \label{eq: N7a element values}
\begin{split}
c_1 = \frac{d_0d_3}{\mathcal{B}_{13}}, ~~
c_2 = \frac{d_0}{a_0}, ~~
k_1 = \frac{\mathcal{B}_{23}}{a_2a_3}, ~~
k_2 = \frac{a_2d_0^2}{a_0 \mathcal{B}_{13}}, ~~
b_1 = \frac{a_0a_3}{\mathcal{B}_{23}}.
\end{split}
\end{equation}
\end{lemma}
\begin{proof}
The method is similar to that of Lemma~\ref{lemma: realizability condition of N1a}, which can be referred to \cite[Section~6]{WC_sup} for details.
\end{proof}

\begin{figure}[thpb]
      \centering
      \subfigure[]{
      \includegraphics[scale=0.9]{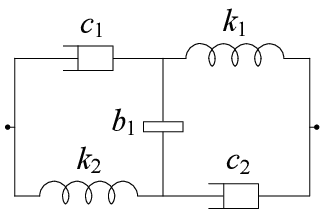}
      \label{fig: N9a}}
      \subfigure[]{
      \includegraphics[scale=0.9]{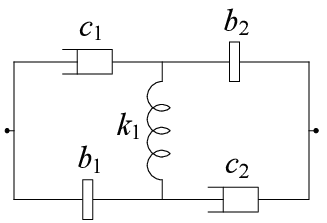}
      \label{fig: N9b}}
      \caption{Five-element non-series-parallel configurations that can realize the bicubic impedance $Z(s)$ in \eqref{eq: three-degree impedance}, whose one-terminal-pair labeled graphs are (a) $\mathcal{N}_{9a}$ and (b) $\mathcal{N}_{9b}$, respectively,  satisfying $\mathcal{N}_{9b} = \text{Dual}(\mathcal{N}_{9a})$, where $\text{Inv}(\mathcal{N}_{9a}) = \text{Dual}(\mathcal{N}_{9a})$
      and $\text{GDu}(\mathcal{N}_{9a}) = \mathcal{N}_{9a}$.}
      \label{fig: N9}
\end{figure}

\begin{figure}[thpb]
      \centering
      \subfigure[]{
      \includegraphics[scale=0.9]{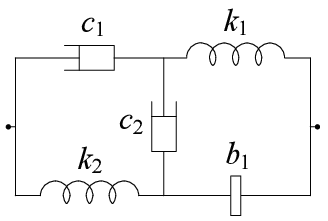}
      \label{fig: N10a}}
      \subfigure[]{
      \includegraphics[scale=0.9]{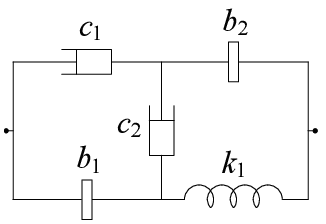}
      \label{fig: N10b}}
      \caption{Five-element non-series-parallel configurations that can realize the bicubic impedance $Z(s)$ in \eqref{eq: three-degree impedance}, whose one-terminal-pair labeled graphs are (a) $\mathcal{N}_{10a}$ and (b) $\mathcal{N}_{10b}$, respectively,  satisfying $\mathcal{N}_{10b} = \text{Dual}(\mathcal{N}_{10a})$, where $\text{Inv}(\mathcal{N}_{10a}) = \text{Dual}(\mathcal{N}_{10a})$
      and $\text{GDu}(\mathcal{N}_{10a}) = \mathcal{N}_{10a}$.}
      \label{fig: N10}
\end{figure}

\begin{figure}[thpb]
      \centering
      \includegraphics[scale=0.9]{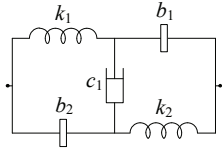}
      \caption{Five-element non-series-parallel configuration that can realize the bicubic impedance $Z(s)$ in \eqref{eq: three-degree impedance}, whose one-terminal-pair labeled graph is $\mathcal{N}_{11}$, where
      $\text{Dual}(\mathcal{N}_{11}) = \text{Inv}(\mathcal{N}_{11}) = \text{GDu}(\mathcal{N}_{11}) = \mathcal{N}_{11}$.}
      \label{fig: N11}
\end{figure}

Define the   notations $\Lambda_{1a}$ and $\Lambda_{1b}$ as
\begin{equation} \label{eq: Lambda1a}
\Lambda_{1a} := \frac{d_3^2 \mathcal{\mathcal{B}}_{11}}{\mathcal{B}_{13} \mathcal{M}_{13}},
\end{equation}
and
\begin{equation}  \label{eq: Lambda1b}
\begin{split}
\Lambda_{1b} :=  \frac{1}{2 a_0^2 \mathcal{B}_{33}}
\Bigg(\mathcal{B}_{13} \mathcal{M}_{13}  - \mathcal{B}_{11} \mathcal{B}_{33} \pm \sqrt{(\mathcal{B}_{13}^2 - \mathcal{B}_{11} \mathcal{B}_{33})(\mathcal{M}_{13}^2 - \mathcal{B}_{11} \mathcal{B}_{33})}  \Bigg)
\end{split}
\end{equation}
Then, the notation $\Lambda_{2a}$ and $\Lambda_{2b}$ can be respectively obtained from
$\Lambda_{1a}$ and $\Lambda_{1b}$ according to the conversion $a_k \leftrightarrow d_k$ for $k = 0, 1, 2, 3$.

\begin{lemma}   \label{lemma: realizability condition of N8a}
Consider a bicubic impedance $Z(s)$ in the form of \eqref{eq: three-degree impedance}, where $a_i, d_j > 0$ for $i, j = 0, 1, 2, 3$, $\delta(Z(s)) = 3$, and there is not any pole or zero on $j \mathbb{R} \cup  \infty$. Then, $Z(s)$ is realizable as one of the five-element non-series-parallel configurations in Fig.~\ref{fig: N8} (whose  one-terminal-pair labeled graph is $\mathcal{N}_{8a}$ or $\mathcal{N}_{8b}$), if and only if one of the following four conditions holds:
\begin{enumerate}
  \item[1.] $\mathcal{B}_{33} = 0$, $\mathcal{B}_{11} \mathcal{B}_{13} > 0$, $a_0^2 a_3 \Lambda_{1a}^3 - a_0a_3d_1 \Lambda_{1a}^2 + a_2d_0d_3 \Lambda_{1a} - d_0d_3^2 = 0$, and
  $a_0^3 \mathcal{M}_{13} \Lambda_{1a}^3 + a_0(2 \mathcal{B}_{11} \mathcal{M}_{13} - a_1a_3d_0^2) \Lambda_{1a}^2 + \mathcal{B}_{11} (d_3 \mathcal{B}_{11} - a_3d_0d_1) \Lambda_{1a} + a_3d_0^3d_3 = 0$;
  \item[2.] $\mathcal{B}_{33} \neq 0$, $(\mathcal{B}_{13}^2 - \mathcal{B}_{11} \mathcal{B}_{33})(\mathcal{M}_{13}^2 - \mathcal{B}_{11} \mathcal{B}_{33}) \geq 0$, $a_0^2 \Lambda_{1b} > \max\{0, -\mathcal{B}_{11} \}$,
      $a_0^2 a_3 \Lambda_{1b}^3 - a_0a_3d_1 \Lambda_{1b}^2 + a_2d_0d_3 \Lambda_{1b} - d_0d_3^2 = 0$, and
  $a_0^3 \mathcal{M}_{13} \Lambda_{1b}^3 + a_0(2 \mathcal{B}_{11} \mathcal{M}_{13} - a_1a_3d_0^2) \Lambda_{1b}^2 + \mathcal{B}_{11} (d_3 \mathcal{B}_{11} - a_3d_0d_1) \Lambda_{1b} + a_3d_0^3d_3 = 0$;
  \item[3.] $\mathcal{B}_{33} = 0$, $\mathcal{B}_{11} \mathcal{B}_{13} > 0$,  $d_0^2 d_3 \Lambda_{2a}^3 - a_1d_0d_3 \Lambda_{2a}^2 + a_0a_3d_2 \Lambda_{2a} - a_0a_3^2 = 0$, and
  $d_0^3 \mathcal{M}_{13} \Lambda_{2a}^3 - d_0(2 \mathcal{B}_{11} \mathcal{M}_{13} + a_0^2d_1d_3) \Lambda_{2a}^2 + \mathcal{B}_{11} (a_3 \mathcal{B}_{11} + a_0a_1d_3) \Lambda_{2a} + a_0^3a_3d_3 = 0$;
  \item[4.] $\mathcal{B}_{33} \neq 0$, $(\mathcal{B}_{13}^2 - \mathcal{B}_{11} \mathcal{B}_{33})(\mathcal{M}_{13}^2 - \mathcal{B}_{11} \mathcal{B}_{33}) \geq 0$, $d_0^2 \Lambda_{2b} > \max\{0, \mathcal{B}_{11} \}$,  $d_0^2 d_3 \Lambda_{2b}^3 - a_1d_0d_3 \Lambda_{2b}^2 + a_0a_3d_2 \Lambda_{2b} - a_0a_3^2 = 0$, and
  $d_0^3 \mathcal{M}_{13} \Lambda_{2b}^3 - d_0(2 \mathcal{B}_{11} \mathcal{M}_{13} + a_0^2d_1d_3) \Lambda_{2b}^2 + \mathcal{B}_{11} (a_3 \mathcal{B}_{11} + a_0a_1d_3) \Lambda_{2b} + a_0^3a_3d_3 = 0$.
\end{enumerate}
Moreover, if  Condition~1 or 2 holds, then $Z(s)$ is realizable as the configuration in Fig.~\ref{fig: N8a} whose element values can be expressed as
\begin{equation}  \label{eq: N8a element values}
\begin{split}
c_1 = \frac{d_0}{a_0},  ~~ c_2 = \frac{d_3}{a_3}, ~~ k_1 = \frac{d_0^2}{a_0^2 \gamma + \mathcal{B}_{11}}, ~~
k_2 = \frac{\gamma (a_0^2 \gamma + \mathcal{B}_{11})}{a_3d_0}, ~~
b_1 = \gamma,
\end{split}
\end{equation}
where $\gamma = \Lambda_{1a}$ when Condition~1 holds and $\gamma = \Lambda_{1b}$ when  Condition~2 holds.
\end{lemma}
\begin{proof}
See Appendix~\ref{appendix: N8a} for details.
\end{proof}

Define the notations $\Gamma_1$, $\Phi_1$, and $\Psi_1$ as
\begin{equation}  \label{eq: Gamma1}
\Gamma_1 := \frac{a_3d_0 \pm \sqrt{a_3d_0(a_3d_0 - 4a_0d_3)}}{2 d_0d_3},
\end{equation}
\begin{equation} \label{eq: Phi1}
\begin{split}
&\Phi_1 :=
\frac{1}{2a_0a_3d_0d_2\Gamma_1} \Bigg( a_1a_3d_0d_2 \Gamma_1^2
- (d_0^2d_3^2\Gamma_1^4 - a_0^2a_3^2)     \\
& ~~~~~ \pm
\sqrt{(a_1a_3d_0d_2\Gamma_1^2 - (d_0d_3\Gamma_1^2 - a_0a_3)^2)   (a_1a_3d_0d_2\Gamma_1^2
- (d_0d_3\Gamma_1^2 + a_0a_3)^2)} \Bigg),
\end{split}
\end{equation}
and
\begin{equation} \label{eq: Psi}
\Psi_1 := \frac{a_1d_2 + a_3d_0 \pm \sqrt{(a_1d_2 - a_3d_0)^2 - 4a_0a_3d_0d_3}}{2d_0d_2} - \frac{d_3}{d_2}\Gamma_1.
\end{equation}
Then, the notations $\Gamma_2$, $\Phi_2$, and $\Psi_2$ can be respectively obtained from
$\Gamma_1$, $\Phi_1$, and $\Psi_1$ according to the conversion $a_k \leftrightarrow d_k$ for $k = 0, 1, 2, 3$.

\begin{lemma}   \label{lemma: realizability condition of N9a}
Consider a bicubic impedance $Z(s)$ in the form of \eqref{eq: three-degree impedance}, where $a_i, d_j > 0$ for $i, j = 0, 1, 2, 3$, $\delta(Z(s)) = 3$, and there is not any pole or zero on $j \mathbb{R} \cup  \infty$. Then, $Z(s)$ is realizable as one of the five-element non-series-parallel configurations in Fig.~\ref{fig: N9} (whose  one-terminal-pair labeled graph is $\mathcal{N}_{9a}$ or $\mathcal{N}_{9b}$), if and only if one of the following two conditions holds:
\begin{enumerate}
  \item[1.] $0 < \Phi_1/\Gamma_1 < a_1/a_0$, $a_3d_0 - 4a_0d_3 \geq 0$,
  $(a_1a_3d_0d_2\Gamma_1^2 - (d_0d_3\Gamma_1^2 - a_0a_3)^2)(a_1a_3d_0d_2\Gamma_1^2
- (d_0d_3\Gamma_1^2 + a_0a_3)^2) \geq 0$,
$a_0(d_0d_3\Gamma_1^2 - a_0a_3)\Phi_1^3 - (a_1d_0d_3 \Gamma_1^2 + a_0a_3d_1
\Gamma_1 - 2a_0a_1a_3) \Gamma_1 \Phi_1^2 + a_1a_3(d_1\Gamma_1 - a_1)\Gamma_1^2 \Phi_1 - a_0a_3d_3 \Gamma_1^4 = 0$, and
$a_0^2d_0 \Phi_1^4 - 2 a_0a_1d_0 \Gamma_1 \Phi_1^3 +
d_0(a_1^2+a_0a_2)\Gamma_1^2 \Phi_1^2 + (a_0d_0d_3 \Gamma_1^2 - a_1a_2
d_0 \Gamma_1 - a_0^2a_3)\Gamma_1^2 \Phi_1 + a_0a_1a_3 \Gamma_1^3 = 0$;
  \item[2.] $0 <  \Phi_2/\Gamma_2 < d_1/d_0$, $a_0 d_3 - 4a_3 d_0 \geq 0$,
$(a_0a_2d_1d_3 \Gamma_2^2 - (a_0 a_3\Gamma_2^2 - d_0d_3)^2)(a_0a_2d_1d_3 \Gamma_2^2 - (a_0a_3\Gamma_2^2 + d_0d_3)^2) \geq 0$,
$d_0(a_0a_3\Gamma_2^2 - d_0d_3)\Phi_2^3 - (a_0a_3d_1 \Gamma_2^2 + a_1d_0d_3
\Gamma_2 - 2d_0d_1d_3) \Gamma_2 \Phi_2^2 + d_1 d_3(a_1\Gamma_2 - d_1) \Gamma_2^2 \Phi_2 -  a_3d_0d_3   \Gamma_2^4 = 0$, and
$a_0d_0^2   \Phi_2^4 - 2 a_0d_0d_1   \Gamma_2 \Phi_2^3 +
a_0(d_1^2+d_0d_2)\Gamma_2^2 \Phi_2^2 + (a_0a_3d_0   \Gamma_2^2 - a_0d_1d_2
\Gamma_2 - d_0^2d_3)\Gamma_2^2 \Phi_2 + d_0d_1d_3 \Gamma_2^3 = 0$.
\end{enumerate}
Moreover, if  Condition~1 holds, then $Z(s)$ is realizable as the configuration in Fig.~\ref{fig: N9a}   whose element values can be expressed as
\begin{equation}  \label{eq: N9a element values}
\begin{split}
c_1 =  \Gamma_1^{-1}, ~ c_2 = \frac{d_0d_3 \Gamma_1}{a_0a_3}, ~  k_1 =  \Phi_1^{-1}, ~
k_2 = \frac{d_0d_3 \Gamma_1^2}{a_3(a_1 \Gamma_1 - a_0 \Phi_1)},  ~
b_1 = \frac{d_3 \Gamma_1^2}{(a_1 \Gamma_1 - a_0 \Phi_1) \Phi_1}.
\end{split}
\end{equation}
\end{lemma}
\begin{proof}
See Appendix~\ref{appendix: N9a} for details.
\end{proof}

\begin{lemma}   \label{lemma: realizability condition of N10a}
Consider a bicubic impedance $Z(s)$ in the form of \eqref{eq: three-degree impedance}, where $a_i, d_j > 0$ for $i, j = 0, 1, 2, 3$, $\delta(Z(s)) = 3$, and there is not any pole or zero on $j \mathbb{R} \cup  \infty$. Then, $Z(s)$ is realizable as one of the five-element non-series-parallel configurations in Fig.~\ref{fig: N10} (whose  one-terminal-pair labeled graph is $\mathcal{N}_{10a}$ or $\mathcal{N}_{10b}$), if and only if one of the following two conditions holds:
\begin{enumerate}
  \item[1.] $0 < \Psi_1 < a_1/d_0$, $a_3d_0 - 4a_0d_3 \geq 0$,  $(a_1d_2 - a_3d_0)^2 - 4a_0a_3d_0d_3 \geq 0$, $d_0^2 \Psi_1^3 - a_1d_0 \Psi_1^2 + a_2d_0\Gamma_1 \Psi_1 - a_0a_3\Gamma_1 = 0$, and $d_0^2 \Psi_1^3 + d_0(d_1\Gamma_1 - 2a_1) \Psi_1^2 - a_1(d_1 \Gamma_1 - a_1) \Psi_1 + a_0d_3 \Gamma_1^2 = 0$;
  \item[2.] $0 < \Psi_2 < d_1/a_0$, $a_0d_3 - 4a_3d_0 \geq 0$,
  $(a_2 d_1 - a_0 d_3)^2 - 4a_0 a_3 d_0 d_3 \geq 0$, $a_0^2 \Psi_2^3 - a_0 d_1 \Psi_2^2 + a_0 d_2 \Gamma_2 \Psi_2 - d_0 d_3 \Gamma_2 = 0$, and
  $a_0^2 \Psi_2^3 + a_0(a_1 \Gamma_2 - 2 d_1) \Psi_2^2 - d_1 (a_1 \Gamma_2 - d_1) \Psi_2 + a_3 d_0 \Gamma_2^2 = 0$.
\end{enumerate}
Moreover, if  Condition~1 holds, then $Z(s)$ is realizable as the configuration in Fig.~\ref{fig: N10a}   whose element values can be expressed as
\begin{equation}  \label{eq: N10a element values}
\begin{split}
c_1 = \Gamma_1^{-1}, ~  c_2 = \frac{d_0d_3 \Gamma_1}{a_0a_3}, ~
k_1 = \Psi_1^{-1}, ~
k_2 = \frac{d_0d_3 \Gamma_1}{a_3 (a_1 - d_0 \Psi_1)}, ~
b_1 = \frac{d_3 \Gamma_1}{(a_1 - d_0 \Psi_1)\Psi_1}.
\end{split}
\end{equation}
\end{lemma}
\begin{proof}
The method is similar to that of Lemma~\ref{lemma: realizability condition of N9a}, which can be referred to \cite[Section~7]{WC_sup} for details.
\end{proof}

\begin{lemma}   \label{lemma: realizability condition of N11}
Consider a bicubic impedance $Z(s)$ in the form of \eqref{eq: three-degree impedance}, where $a_i, d_j > 0$ for $i, j = 0, 1, 2, 3$, $\delta(Z(s)) = 3$, and there is not any pole or zero on $j \mathbb{R} \cup  \infty$. Then, $Z(s)$ is realizable as  the five-element non-series-parallel configuration  in Fig.~\ref{fig: N11} (whose  one-terminal-pair labeled graph is $\mathcal{N}_{11}$), if and only if $\mathcal{B}_{13} = 0$, and there exists a positive root $T > 0$ for the equation $a_0d_0d_3 T^3 + (a_1d_0d_3 + a_0d_1d_3 - a_2d_0d_2) T^2 +
(a_1 d_1 d_3 - a_2 d_0 d_3 - a_3 d_0 d_2) T - a_3 d_0 d_3 = 0$ such that
$a_0d_2 T^2 + (a_1d_2 - 3a_0d_3) T + a_1d_3 \geq 0$, $(a_2^2d_2 - 4a_0a_3d_3) T^3 + (a_2^2 d_3 + 2a_2a_3d_2 - 4a_1a_3d_3) T^2 + a_3(a_3d_2 + 2a_2d_3) T + a_3^2d_3 \geq 0$, $a_1 T + a_2 - a_0(y_1z_1 + y_2z_2) = 0$, and $d_1 T + d_2 - d_0(y_1z_2 + y_2z_1) = 0$, where $y_1$ and $y_2$ are two positive roots of the following equation in $y$:
\begin{equation}  \label{eq: N11 L1 L2 equation}
d_0(d_2T + d_3) y^2 - (a_0T + a_1)(d_2T+d_3) y + a_3T(a_0T + a_1) = 0,
\end{equation}
and $z_1$ and $z_2$ are two positive roots of the following equation in $z$:
\begin{equation}  \label{eq: N11 C1 C2 equation}
a_3 T (a_0 T + a_1) z^2 - (a_2T+a_3)(d_2T+d_3) z + d_3T(d_2 T + d_3) = 0.
\end{equation}
Moreover, if  the condition of this lemma holds, then $Z(s)$ is realizable as the configuration in Fig.~\ref{fig: N11}, where
\begin{equation}  \label{eq: N11 element values}
c_1 = \frac{d_3}{a_3},~~ k_1 = y_1^{-1}, ~~ k_2 = y_2^{-1}, ~~ b_1 = z_1, ~~ b_2 = z_2.
\end{equation}
\end{lemma}
\begin{proof}
See Appendix~\ref{appendix: N11} for details.
\end{proof}

Then, combining Lemmas~\ref{lemma: third-order non-series-parallel configurations}--\ref{lemma: realizability condition of N11}, the following   Theorem~\ref{theorem: third-degree five-element non-series-parallel network} can be proved, which presents a necessary and sufficient condition for a bicubic impedance $Z(s)$ in the form of \eqref{eq: three-degree impedance} to be realizable as a five-element non-series-parallel network.

\begin{theorem}
\label{theorem: third-degree five-element non-series-parallel network}
Consider a bicubic impedance $Z(s)$ in the form of \eqref{eq: three-degree impedance}, where $a_i, d_j > 0$ for $i, j = 0, 1, 2, 3$, $\delta(Z(s)) = 3$, and there is not any pole or zero on $j \mathbb{R} \cup  \infty$. Then, $Z(s)$ is realizable as a five-element non-series-parallel network, if and only if one of the conditions in Lemmas~\ref{lemma: realizability condition of N7a}--\ref{lemma: realizability condition of N11} holds.
\end{theorem}
\begin{proof}
By Lemma~\ref{lemma: third-order non-series-parallel configurations}, the bicubic impedance $Z(s)$ in this theorem  is realizable as a five-element non-series-parallel network if and only if $Z(s)$ is realizable as one of the configurations in Figs.~\ref{fig: N7}--\ref{fig: N11}. Since the necessary and sufficient conditions for  $Z(s)$  to be realizable as the configurations in Figs.~\ref{fig: N7}--\ref{fig: N11} are shown in Lemmas~\ref{lemma: realizability condition of N7a}--\ref{lemma: realizability condition of N11}, this theorem can be proved.
\end{proof}
\section{Numerical Examples and Positive-Real Controller Designs for Inerter-Based Control Systems}
\label{sec: examples}

In this section, two examples  in the     positive-real controller designs for a quarter-car suspension system will be presented for illustrations.

It is shown that  the bicubic impedances satisfying the conditions of this paper (realizable with five elements) can provide better ride comfort performances compared with the biquadratic positive-real impedances (realizable with no more than nine elements by the Bott-Duffin procedure).

Consider a  quarter-car   suspension system  as shown in Fig.~\ref{fig: suspension-model}, where the admittance $K(s) = 1/Z(s)$ of a passive mechanical network can be regarded as the positive-real controller as shown in Fig.~\ref{fig: suspension-control}. Here, the spring with   stiffness $k_t$ denotes   the
      vehicle tyre,  the sprung mass $m_s$ denotes
the vehicle body, the unsprung mass $m_u$ denotes the vehicle wheel,  $z_s$ denotes the displacement of the sprung mass, $z_u$ denotes the displacement of the unsprung mass, and $z_r$ denotes the road displacement.

\begin{figure}[thpb]
\centering
      \subfigure[]{
      \includegraphics[scale=1.1]{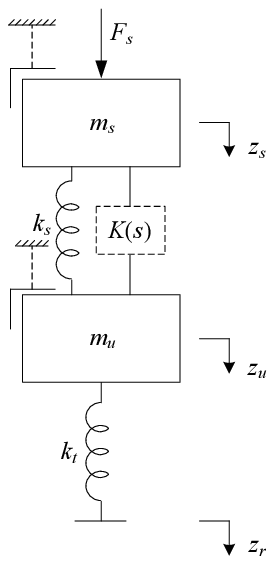}
      \label{fig: suspension-model}} ~~~~~
      \subfigure[]{
      \includegraphics[scale=0.85]{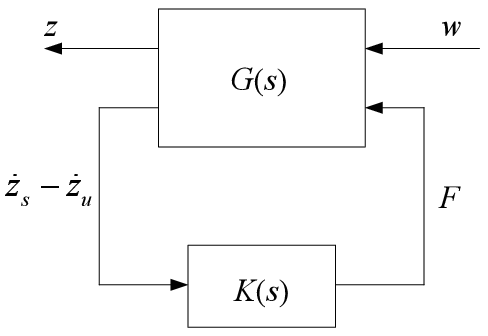}
      \label{fig: suspension-control}}
\caption{(a) A quarter-car vehicle suspension system model \cite{PS06,SW04},  where the spring $k_s$ and the passive network whose admittance $K(s) = 1/Z(s)$ constitute the suspension part.
(b) Control synthesis diagram with positive-real controller $K(s) = 1/Z(s)$. }
\end{figure}

As shown in \cite{PS06}, by Newton's Second Law, the motion equations of the quarter-car suspension system in Fig.~\ref{fig: suspension-model}
 can be formulated as
\begin{equation}  \label{eq: suspension model equations}
m_s \ddot{z}_s = F_s - F - k_s (z_s - z_u),  ~~~
m_u \ddot{z}_u = F + k_s (z_s - z_u) + k_t (z_r - z_u),
\end{equation}
where $F$ is the output of the controller $K(s)$ satisfying
$\hat{F} = K(s) (s\hat{z}_s - s\hat{z}_u)$. Here, $\hat{F}$, $\hat{z}_s$, and $\hat{z}_u$ denote the Laplace transforms of $F$, $z_s$, and $z_u$, respectively. Furthermore, let $w = [F_s, z_r]^T$ denote the input of the generalized plant, and let $z = [\dot{z}_s, z_s]^T$ denote the performance output of the generalized plant. As a consequence, by letting the state vector satisfy
$x = [\dot{z}_s, z_s, \dot{z}_u, z_u]^T$, the motion equations in \eqref{eq: suspension model equations} can be written in the state-space form as follows:
\begin{equation}  \label{eq: state space generalized plant}
\dot{x} = A x + B \left[
                    \begin{array}{c}
                      w \\
                      F \\
                    \end{array}
                  \right], ~~~
\left[
                    \begin{array}{c}
                      z \\
                      \dot{z}_s - \dot{z}_u \\
                    \end{array}
                  \right] =
C x,
\end{equation}
where
\begin{equation}  \label{eq: A B C}
A = \left[
      \begin{array}{cccc}
        0 & -\frac{k_s}{m_s} & 0 & \frac{k_s}{m_s} \\
        1 & 0 & 0 & 0 \\
        0 & \frac{k_s}{m_u} & 0 & -\frac{k_s + k_t}{m_u} \\
        0 & 0 & 1 & 0 \\
      \end{array}
    \right], ~~
B = \left[
      \begin{array}{ccc}
        \frac{1}{m_s} & 0 & -\frac{1}{m_s}   \\
        0 & 0 & 0  \\
        0 & \frac{k_t}{m_u} & \frac{1}{m_u}   \\
        0 & 0 & 0   \\
      \end{array}
    \right], ~~
C = \left[
      \begin{array}{cccc}
        1 & 0 & 0 & 0   \\
        0 & 1 & 0 & 0  \\
        1 & 0 & -1 & 0   \\
      \end{array}
    \right].
\end{equation}

Assume that $\{ A_k, B_k, C_k, D_k \}$ is a minimal realization of the positive-real controller $K(s)$, that is,
$K(s) = C_k (sI - A_k)^{-1} B_k + D_k$. Then,
\begin{equation}  \label{eq: state space controller}
\dot{x}_k = A_k x_k + B_k (\dot{z}_s - \dot{z}_u), ~~~
F = C_k x_k + D_k (\dot{z}_s - \dot{z}_u),
\end{equation}
where the dimension of $x_k$ is equal to the McMillan degree of $K(s)$, which is denoted as $n_k$.

Assume that $F_s = 0$. Then,
combining \eqref{eq: state space generalized plant}--\eqref{eq: state space controller}, the closed-loop system whose input is $z_r$ and output is $\dot{z}_s$ can be obtained as
\begin{equation}  \label{eq: state space closed-loop system}
\dot{x}_{cl}
= A_{cl} x_{cl}  + B_{cl} z_r, ~~~
\dot{z}_s = C_{cl} x_{cl},
\end{equation}
where $x_{cl} = [x, x_k]^T$, and
\begin{equation}  \label{eq: Acl Bcl Ccl closed-loop system}
A_{cl} = \left[
           \begin{array}{cc}
             A + B_3 D_k C_3 & B_3 C_k \\
             B_k C_3 & A_k \\
           \end{array}
         \right], ~~~
B_{cl} = \left[
           \begin{array}{c}
             B_2  \\
             \boldsymbol{0}_{n_k}   \\
           \end{array}
         \right], ~~~
C_{cl} = \left[
           \begin{array}{cc}
             C_1  &
             \boldsymbol{0}_{n_k}^T   \\
           \end{array}
         \right].
\end{equation}
Here, $B_2$ and $B_3$ denote the second and third columns of $B$, $C_1$ and $C_3$ denote the first and third rows of $C$, and
$\boldsymbol{0}_{n_k}$ denotes the zero column vector whose dimension is $n_k$.

As shown in \cite{SW04}, the ride comfort index, which is the root-mean-square value  of $\ddot{z}_s$, can be expressed as
\begin{equation} \label{eq: ride comfort index}
J_1 = 2 \pi \sqrt{V \kappa} \|s^{-1} T_{z_r \rightarrow \ddot{z}_s} \|_2 = 2 \pi \sqrt{V \kappa} \|T_{z_r \rightarrow \dot{z}_s} \|_2,
\end{equation}
where $V$ denotes the vehicle speed, $\kappa$ denotes the road roughness parameter, and $T_{z_r \rightarrow \dot{z}_s}$ denotes the transfer function from $z_r$ to $\dot{z}_s$.
The following lemma shown in \cite[pg.~25]{DFT09} can be utilized to derive an equivalent form of $J_1$ in \eqref{eq: ride comfort index}.

\begin{lemma} \cite[pg.~25]{DFT09}  \label{lemma: H2 norm}
Consider a SISO closed-loop system \eqref{eq: state space closed-loop system}. If $A_{cl}$ is stable, that is, $\Re(\lambda_i (A_{cl})) < 0$, then
the $H_2$ norm from the input $z_r$ to the output $\dot{z}_s$ satisfies
\begin{equation*}
\|T_{z_r \rightarrow \dot{z}_s} \|_2^2 = C_{cl} P C_{cl},
\end{equation*}
where the positive definite matrix $P > 0$ is the unique solution of the Lyapunov equation
\begin{equation} \label{eq: lyp equation}
A_{cl} P + P A_{cl}^T = - B_{cl} B_{cl}^T.
\end{equation}
\end{lemma}

Furthermore, assuming that   $A_{cl}$ is stable, by Lemma~\ref{lemma: H2 norm},    the ride comfort index $J_1$ as in \eqref{eq: ride comfort index} can be equivalent to
\begin{equation} \label{eq: J1 equivalent form}
J_1 =  2 \pi \sqrt{V \kappa} (C_{cl} P C_{cl})^{1/2},
\end{equation}
where the positive definite matrix $P > 0$ is the unique solution of the Lyapunov equation in \eqref{eq: lyp equation}. It is obvious that $P$ is related to $K(s)$.
Therefore, the optimization problem of ride comfort $J_1$ is listed in the following procedure.

\begin{procedure}  \label{procedure: design process}
Consider a quarter-car   suspension system as in Fig.~\ref{fig: suspension-model}, whose motion equations satisfy the state-space form in \eqref{eq: state space generalized plant}. Assuming that $F_s = 0$, the steps of designing a positive-real controller $K(s)$ to minimize the ride comfort performance $J_1$ in \eqref{eq: ride comfort index} (or \eqref{eq: J1 equivalent form}) for the
closed-loop system in \eqref{eq: state space closed-loop system} are as follows.
\begin{enumerate}
  \item[1.] Choose the McMillan degree $n$ of the positive-real controller $K(s)$, which is the admittance of a passive damper-spring-inerter network. Then, the impedance can be written as $Z(s) := 1/K(s) = (a_n s^n +  \cdots + a_1 s + a_0)/(d_n s^n +   \cdots + d_1 s + d_0)$, where
        $a_i, d_j \geq 0$ for $i, j = 0, 1, ..., n$.  Determine the positive-real condition and choose the further constraint conditions of the $n$th-order $Z(s) = 1/K(s)$, which can guarantee $K(s)$ to be realizable as a specific class of passive damper-spring-inerter networks.
  \item[2.] Calculate a minimal realization $\{ A_k, B_k, C_k, D_k \}$ of $K(s)$ satisfying \eqref{eq: state space controller}, which is related to $a_i, d_j$ for $i, j = 0, 1, ..., n$.
  \item[3.] Then, optimize the following problem:
  \begin{equation*}
\begin{split}
&\mathop{\min}_{a_i,d_j} H_1:= C_{cl} P C_{cl},    \\
&\text{s.t.} ~~~ A_{cl} ~ \text{is stable}, ~  \\
& ~~~~~~~ P ~ \text{is the solution of \eqref{eq: lyp equation}},  \\ & ~~~~~~~ Z(s)=1/K(s) ~\text{is a class of positive-real impedances in Step~1},
\end{split}
  \end{equation*}
  where the optimization variables are the nonnegative coefficients $a_i, d_j$ of $K(s)$, and the optimal positive-real controller can be obtained.
  \item[4.] Calculate the optimal ride comfort performance by \eqref{eq: J1 equivalent form}, that is, $J_1 = 2 \pi \sqrt{V \kappa H_1}$.
  \item[5.] Making use of the results of network synthesis, realize the positive-real controller $K(s) = 1/Z(s)$ corresponding to the optimal performance as a damper-spring-inerter network.
\end{enumerate}
\end{procedure}

For the case (Case~A) when $Z(s) = 1/K(s)$ is a bicubic (third-order)  impedance  as in \eqref{eq: three-degree impedance},
where  $a_i, d_j > 0$ for $i, j = 0, 1, 2, 3$,
one can further assume that $a_i$ and $d_j$  satisfy the conditions in Theorem~\ref{theorem: third-degree five-element series-parallel network} or \ref{theorem: third-degree five-element non-series-parallel network}. Then, the class of positive-real controllers in Step~1 of Procedure\ref{procedure: design process} is chosen as above for this case.  This means that any damper-spring-inerter realization of the optimal positive-real controller $K(s) = 1/Z(s)$ contains no more than five elements, and the conditions in
Lemmas~\ref{lemma: realizability condition of N1a}--\ref{lemma: realizability condition of N6a}
and \ref{lemma: realizability condition of N7a}--\ref{lemma: realizability condition of N11}
are regarded as the optimization constraints in Step~3 of Procedure~\ref{procedure: design process}.

For the case (Case~B) when $Z(s) = 1/K(s)$ is any biquadratic (second-order) positive-real  impedance
\begin{equation} \label{eq: biquadratic form}
Z(s) = \frac{a_2 s^2 + a_1 s + a_0}{d_2 s^2 + d_1 s + d_0},
\end{equation}
where $a_i, d_j > 0$ for $i, j = 0, 1, 2$ and
$(\sqrt{a_2 d_0} - \sqrt{a_0 d_2})^2 \leq a_1 d_1$,  any damper-spring-inerter realization of the optimal positive-real controller $K(s) = 1/Z(s)$ contains no more than nine elements by using the \emph{Bott-Duffin procedure} \cite{BD49}. Then, the class of positive-real controllers in Step~1 of Procedure\ref{procedure: design process} is chosen as above for this case.

Let the parameters of the suspension model satisfy $m_s = 250$~kg, $m_u = 35$~kg, $k_t = 150$~kN/m, $V = 25$~m/s, and  $\kappa = 5\times 10^{-7}$~$\text{m}$/cycle, which are the same as those in \cite{SW04}.
Following Procedure~\ref{procedure: design process} where  the optimization solver
\emph{fmincon} in MATLAB is utilized in Step~3,  the optimal results of ride comfort $J_1$ for Case~A (solid line) and Case~B (dashed line) can be obtained  as shown in
Fig.~\ref{fig: J1},
where the static stiffness $k_s$ is a fixed value ranging from $10$~kN/m to $120$~kN/m. It is shown that for different values of static stiffness $k_s$ the optimal performance values $J_1$ for Case A is enhanced compared with the values for Case~B, and the percentage improvement can be over $9 \%$ for small values of $k_s$.  This means that the third-order positive-real controller $K(s)$ that is realizable as a five-element network using the results of this paper can provide better ride comfort performances than the second-order positive-real controller $K(s)$ that is realizable as a series-parallel (resp. non-series-parallel) network containing no more than nine (resp. eight) elements by using the Bott-Duffin procedure (resp. Pantell's modified Bott-Duffin procedure). Therefore, the class of third-order positive-real controllers realizable as a five-element network can provide both lower physical complexity and better system performances than the conventional second-order positive-real controllers, which can also illustrate the significance of this paper.

\begin{figure}[thpb]
      \centering
      \subfigure[]{
      \includegraphics[scale=0.52]{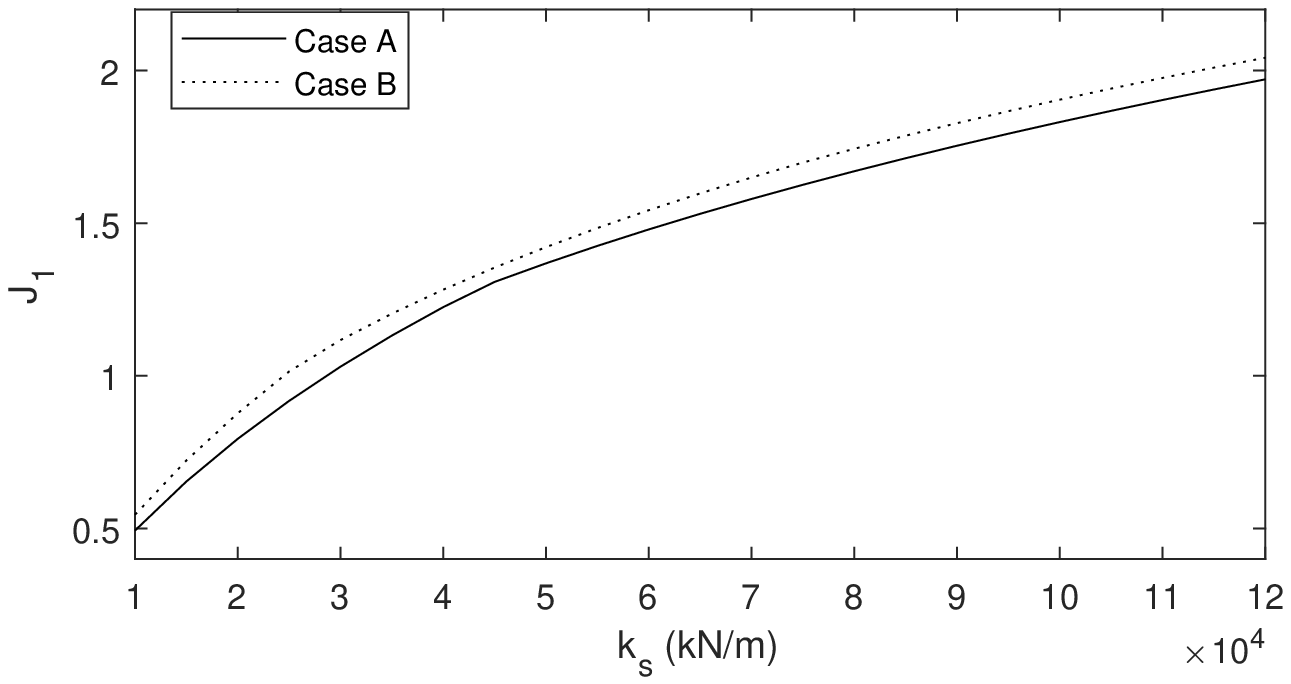}
      \label{fig: performances}}
      \subfigure[]{
      \includegraphics[scale=0.52]{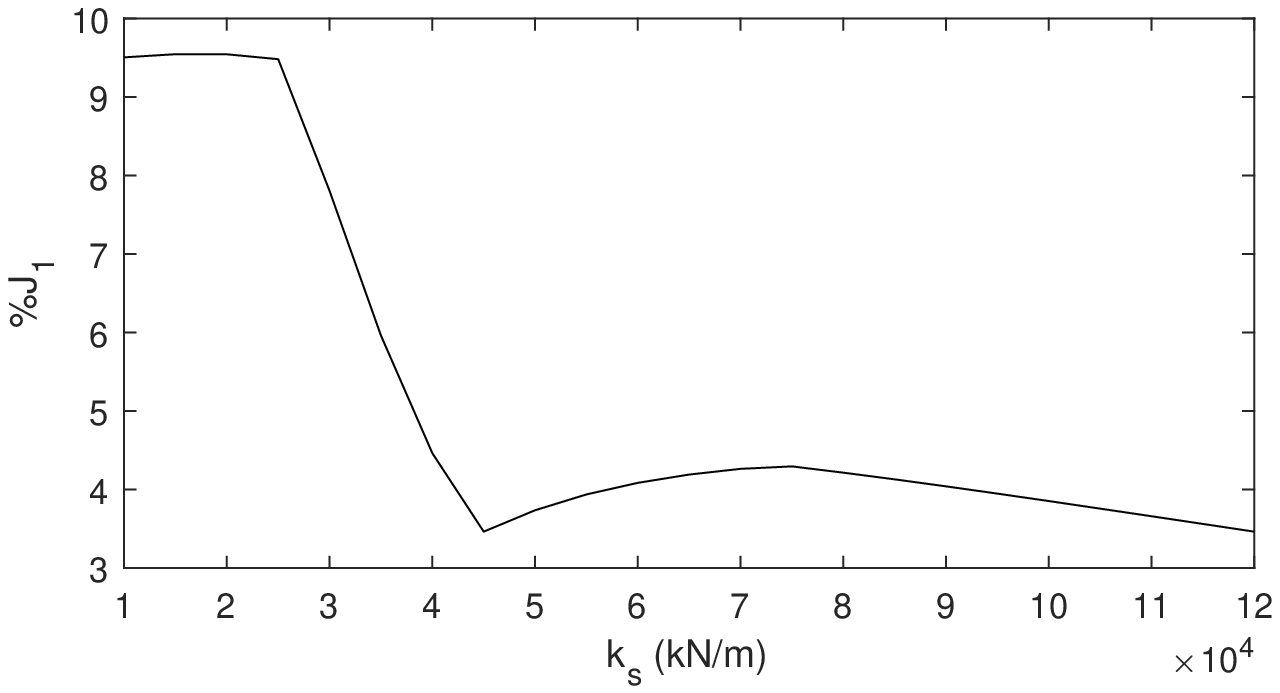}
      \label{fig: percentage-improvement}}
      \caption{(a) The optimal performances $J_1$ for the case when $Z(s) = 1/K(s)$ is any bicubic impedance satisfying the condition of  Theorem~3  (Case~A,  solid line), and the case when $Z(s) = 1/K(s)$ is any positive-real biquadratic impedance (Case~B,  dashed line),
       where the static stiffness $k_s$ ranges from $10$~kN/m to $120$~kN/m.
       (b) The percentage improvement of optimal performances for Case A and Case B, which is $(J_1^{(2)} - J_1^{(1)})/J_1^{(2)}  \times 100 \%$, where $J_1^{(1)}$ and $J_1^{(2)}$ are optimal performance values corresponding to Cases~A and B, respectively.}
      \label{fig: J1}
\end{figure}

The following two examples show the realization results of the positive-real controllers in the
above optimization designs when the   static stiffness $k_s$ satisfies $k_s = 25$~kN/m and $k_s = 70$~kN/m, respectively.

\begin{example}  \label{example: 01}
When $k_s = 25$~kN/m, the optimal value of $J_1$ for Case~A satisfies $J_1^{(1)} = 0.9182$, and the corresponding bicubic impedance $Z(s)$ is as in \eqref{eq: three-degree impedance}  with  $a_3 = 5.994 \times 10^{-4}$, $a_2 = 0.07188$, $a_1 = 1.529$, $a_0 = 14.818$, $d_3 = 1$, $d_2 = 5.005 \times 10^{-8}$, $d_1 = a_1 d_3 /a_3 = 2.55 \times 10^3$, and $d_0 = a_0 d_2/a_2 = 1.031 \times 10^{-5}$. Then, it can be checked that
$\mathcal{B}_{13} = a_3 d_0 - a_0 d_3 = -14.818 < 0$, $\mathcal{B}_{12} = a_2 d_0 - a_0 d_2 = 0$, $\mathcal{B}_{23} = a_3 d_1 - a_1 d_3 = 0$, and $\Delta_1 = a_1 a_2 - a_0 a_3 = 0.10102 > 0$, which implies that the condition of Lemma~5 holds. Therefore, $Z(s)$ is realizable as the configuration in Fig.~1(b) with $c_1 = 1.668 \times 10^3$~Ns/m, $c_2 = 6.96 \times 10^{-7}$~Ns/m, $b_1 = 172.097$~kg, $b_2 = 15.131$~kg, and $k_1 = 3.858 \times 10^4$~N/m, which is shown in Fig.~\ref{fig: Example 01 - a}. By the Bott-Duffin synthesis procedure,  $Z(s)$  is   realizable as the configuration in Fig.~\ref{fig: Example 01 - b} with $c_1 = 6.956 \times 10^{-7}$~Ns/m, $c_2 = 1.667 \times 10^3$~Ns/m, $c_3 = 1.668 \times 10^3$~Ns/m, $c_4 = 3.999 \times 10^{12}$~Ns/m, $k_1 = 5.41 \times 10^8$~N/m, $k_2 = 3.859 \times 10^4$~N/m, $k_3 = 1.838 \times 10^5$~N/m, $k_4 = 1.617 \times 10^4$~N/m, $b_1 = 5.141 \times 10^{-3}$~kg, $b_2 = 15.13$~kg, $b_3 = 72.066$~kg, and $b_4 = 172.028$~kg.
For Case~B, the optimal value of $J_1$ satisfies
$J_1^{(2)} = 1.0144$, and the corresponding biquadratic  positive-real impedance $Z(s)$ is as in  \eqref{eq: biquadratic form} with $a_2 = 1$, $a_1 = 226.559$, $a_0 = 1.34 \times 10^4$,
$d_2 = 5.083 \times 10^3$, $d_1 = 7.6  \times 10^4$, and $d_0 = 1.684  \times 10^7$. By using the Bott-Duffin synthesis procedure, $Z(s)$ is realizable as a nine-element series-parallel configuration as the configuration in Fig.~\ref{fig: Example 01 - c} with element values satisfying $c_m = 5.476 \times 10^{11}$~Ns/m, $c_1 = 1.257 \times 10^3$~Ns/m, $c_2 = 5.083  \times 10^3$~Ns/m, $k_1 = 5.664 \times 10^5$~N/m, $k_2 = 6.564 \times 10^4$~N/m,   $k_3 = 6.487 \times 10^5$~N/m, $b_1 = 11.28$~kg, $b_2 = 9.85$~kg, and $b_3 = 97.341$~kg.
\end{example}

\begin{figure}[thpb]
\centering
\subfigure[]{
\includegraphics[scale=0.9]{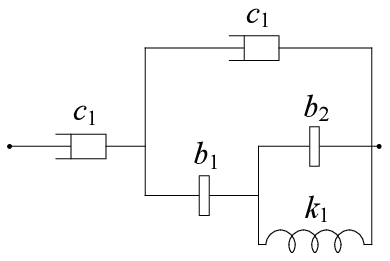}
\label{fig: Example 01 - a}}
\subfigure[]{
\includegraphics[scale=0.9]{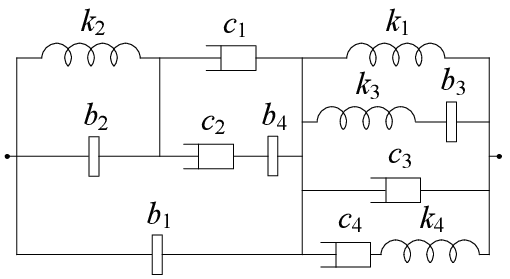}
\label{fig: Example 01 - b}}
\subfigure[]{
\includegraphics[scale=0.9]{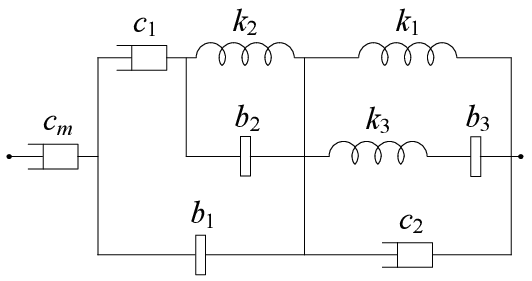}
\label{fig: Example 01 - c}}
\caption{(a) A five-element network realization of the optimal bicubic impedance in Example~1, where the configuration is as in  Fig.~1(b) and the element values satisfy $c_1 = 1.668 \times 10^3$~Ns/m, $c_2 = 6.96 \times 10^{-7}$~Ns/m, $b_1 = 172.097$~kg, $b_2 = 15.131$~kg, and $k_1 = 3.858 \times 10^4$~N/m.
(b) A  Bott-Duffin realization of the optimal bicubic impedance in Example~1, where the element values satisfy $c_1 = 6.956 \times 10^{-7}$~Ns/m, $c_2 = 1.667 \times 10^3$~Ns/m, $c_3 = 1.668 \times 10^3$~Ns/m, $c_4 = 3.999 \times 10^{12}$~Ns/m, $k_1 = 5.41 \times 10^8$~N/m, $k_2 = 3.859 \times 10^4$~N/m, $k_3 = 1.838 \times 10^5$~N/m, $k_4 = 1.617 \times 10^4$~N/m, $b_1 = 5.141 \times 10^{-3}$~kg, $b_2 = 15.13$~kg, $b_3 = 72.066$~kg, and $b_4 = 172.028$~kg.
(c)  A Bott-Duffin realization configuration of the optimal biquadratic positive-real impedance in Example~1, where
the element values satisfy $c_m = 5.476 \times 10^{11}$~Ns/m, $c_1 = 1.257 \times 10^3$~Ns/m, $c_2 = 5.083  \times 10^3$~Ns/m, $k_1 = 5.664 \times 10^5$~N/m, $k_2 = 6.564 \times 10^4$~N/m,   $k_3 = 6.487 \times 10^5$~N/m, $b_1 = 11.28$~kg, $b_2 = 9.85$~kg, and $b_3 = 97.341$~kg.}
\label{fig: Example 01}
\end{figure}

\begin{example}  \label{example: 02}
When $k_s = 70$~kN/m, the optimal value of $J_1$ for Case~A satisfies $J_1^{(1)} = 1.579$, and the corresponding bicubic impedance $Z(s)$ is as in (1) with
$a_3 = 279.553$, $a_2 = 4.239 \times 10^3$, $a_1 = 2.398 \times 10^4$, $a_0 = 2.232 \times 10^5$, $d_3 = 1$, $d_2 = 9.3105$, $d_1 = 141.471$, and $d_0 = 798.595$. Then, it can be checked that
the condition of Lemma~16 holds, where $T = 0.1078$, $y_1 = 30.221$, $y_2 = 29.937$, $z_1 = 1.772\times 10^{-4}$, and
$z_2 = 8.421 \times 10^{-4}$.
Therefore, $Z(s)$ is realizable as the configuration in Fig.~11 with
$c_1 =  3.577  \times 10^{-3}$~Ns/m, $k_1 = 0.03309$~N/m, $k_2 = 0.033403$~N/m, $b_1 = 1.772 \times 10^{-4}$~kg, and $b_2 = 8.421 \times 10^{-4}$~kg, which is shown in Fig.~\ref{fig: Example 02 - a}. By the Bott-Duffin synthesis procedure,  $Z(s)$  is   realizable as the configuration in Fig.~\ref{fig: Example 02 - b} with
$c_1 = 3.577 \times 10^{-3}$~Ns/m, $c_2 = 1.261 \times 10^{-5}$~Ns/m, $c_3 = 3.577 \times 10^{-3}$~Ns/m, $c_4 = 1.261 \times 10^{-5}$~Ns/m, $k_1 = 0.03367$~N/m, $k_2 = 0.02124$~N/m, $k_3 = 0.05203$~N/m, $k_4 = 1.168 \times 10^{-4}$~N/m, $b_1 = 3.813 \times 10^{-4}$~kg,
$b_2 = 2.468 \times 10^{-4}$~kg, $b_3 = 6.045 \times 10^{-4}$~kg, and $b_4 = 1.357 \times 10^{-6}$~kg.
For Case~B, the optimal value of $J_1$ satisfies
$J_1^{(2)} = 1.6498$, and the corresponding biquadratic  positive-real impedance $Z(s)$ is as in  \eqref{eq: biquadratic form} with $a_2 = 1$, $a_1 = 11.057$, $a_0 = 109.731$, $d_2 = 2.942 \times 10^3$, $d_1 = 1.798 \times 10^4$, and $d_0 = 1.496 \times 10^4$. By using the Bott-Duffin synthesis procedure, $Z(s)$ is realizable as a nine-element series-parallel configuration in
 Fig.~\ref{fig: Example 02 - c} with element values satisfying
$c_m = 3.537 \times 10^{12}$~Ns/m, $c_1 = 136.323$~Ns/m, $c_2 = 2.942 \times 10^3$~Ns/m, $k_1 = 1.921 \times 10^3$~N/m,
$k_2 =  379.207$~N/m,  $k_3 = 2.498 \times 10^4$~N/m, $b_1 = 208.8002$~kg, $b_2 = 16.054$~kg, and $b_3 = 1.058 \times 10^3$~kg.
\end{example}

\begin{figure}[thpb]
\centering
\subfigure[]{
\includegraphics[scale=0.9]{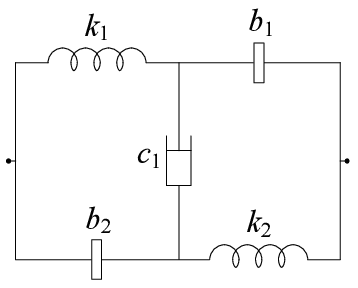}
\label{fig: Example 02 - a}}
\subfigure[]{
\includegraphics[scale=0.9]{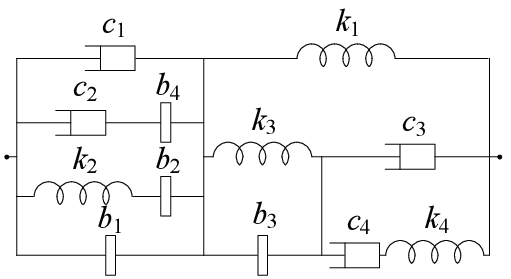}
\label{fig: Example 02 - b}}
\subfigure[]{
\includegraphics[scale=0.9]{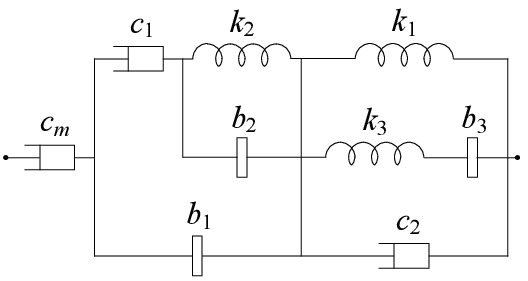}
\label{fig: Example 02 - c}}
\caption{(a) A five-element network realization of the optimal bicubic impedance in Example~2, where the configuration is as in  Fig.~11 and the element values satisfy  $c_1 =  3.577  \times 10^{-3}$~Ns/m, $k_1 = 0.03309$~N/m, $k_2 = 0.033403$~N/m, $b_1 = 1.772 \times 10^{-4}$~kg, and $b_2 = 8.421 \times 10^{-4}$~kg.
(b) A  Bott-Duffin realization of the optimal bicubic impedance in Example~2, where the element values satisfy  $c_1 = 3.577 \times 10^{-3}$~Ns/m, $c_2 = 1.261 \times 10^{-5}$~Ns/m, $c_3 = 3.577 \times 10^{-3}$~Ns/m, $c_4 = 1.261 \times 10^{-5}$~Ns/m, $k_1 = 0.03367$~N/m, $k_2 = 0.02124$~N/m, $k_3 = 0.05203$~N/m, $k_4 = 1.168 \times 10^{-4}$~N/m, $b_1 = 3.813 \times 10^{-4}$~kg,
$b_2 = 2.468 \times 10^{-4}$~kg, $b_3 = 6.045 \times 10^{-4}$~kg, and $b_4 = 1.357 \times 10^{-6}$~kg.
(c)  A Bott-Duffin realization configuration of the optimal biquadratic positive-real impedance in Example~2, where
the element values satisfy $c_m = 3.537 \times 10^{12}$~Ns/m, $c_1 = 136.323$~Ns/m, $c_2 = 2.942 \times 10^3$~Ns/m, $k_1 = 1.921 \times 10^3$~N/m,
$k_2 =  379.207$~N/m,  $k_3 = 2.498 \times 10^4$~N/m, $b_1 = 208.8002$~kg, $b_2 = 16.054$~kg, and $b_3 = 1.058 \times 10^3$~kg.
}
\label{fig: Example 02}
\end{figure}

As shown in Examples~\ref{example: 01} and \ref{example: 02}, the five-element realization results derived in this paper can provide much fewer elements than the Bott-Duffin synthesis procedure, provided that the bicubic impedance satisfies the corresponding realizability conditions.
Moreover, the five-element realizations can even contain fewer elements than the Bott-Duffin realizations of the optimal biquadratic impedance.
Since the realization results in this paper can be directly obtained by testing the realizability conditions and calculating the element expressions,
it is more convenient to obtain the realization networks compared with the Bott-Duffin synthesis procedure.


\section{Conclusion}

This paper has solved the realization problem of a bicubic impedance as a passive damper-spring-inerter network consisting of no more than five elements. The realization results of the specific bicubic impedance contains a pole or zero on $j \mathbb{R} \cup \infty$ with no more than five elements were first obtained. Then, a necessary and sufficient condition was derived for a bicubic impedance containing neither pole nor zero on  $j \mathbb{R} \cup \infty$ to be realizable as a five-element series-parallel network, by proving that 22 configurations classified into six quartets can cover this case and investigating their realizability conditions. Similarly, the synthesis results of five-element non-series-parallel networks were derived, where a necessary and sufficient for the realizability and
11 covering configurations classified into five quartets were obtained. Finally, some numerical examples
together with the optimization designs of positive-real controllers for suspension systems
were presented for illustrations. The results of this paper can theoretically contribute to investigating the minimal realizations of low-order impedances and can be directly utilized to design low-complexity electrical and mechanical networks, which are motivated by inerter-based vibration control.

\vspace{0.2cm}

\begin{appendices}

\section{Proof of Lemma~\ref{lemma: third-order configurations}}

\emph{Sufficiency.}
The sufficiency part is clearly satisfied.

\emph{Necessity.}
The necessity part can be proved by showing that the configurations in Figs.~\ref{fig: N1}--\ref{fig: N6} can  cover all the possible cases.

By Lemma \ref{lemma: lossless passive networks}, to avoid  lossless subnetworks,  $Z(s)$ is always realizable as the configuration belonging to one of the classes in Figs.~\ref{fig: R1-N2} and \ref{fig: R1L1-N2}.
Based on the principles of duality, frequency inversion, and frequency-inverse duality, it suffices to discuss  Figs.~\ref{fig: R1-N2-a} and \ref{fig: R1L1-N2-a}.

\begin{figure}[thpb]
      \centering
      \subfigure[]{
      \includegraphics[scale=1.0]{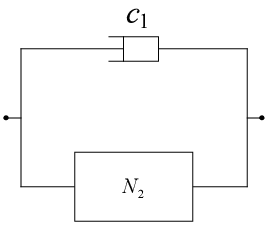}
      \label{fig: R1-N2-a}}
      \subfigure[]{
      \includegraphics[scale=1.0]{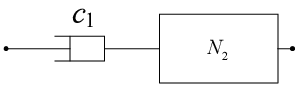}
      \label{fig: R1-N2-b}}
      \caption{Classes of five-element series-parallel configurations, where $N_2$ is a four-element series-parallel subnetwork consisting of one damper and three energy storage elements.}
      \label{fig: R1-N2}
\end{figure}

\begin{figure}[thpb]
      \centering
      \subfigure[]{
      \includegraphics[scale=1.0]{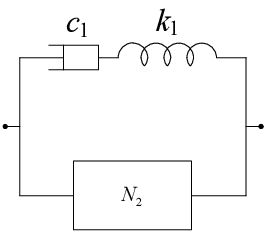}
      \label{fig: R1L1-N2-a}}
      \subfigure[]{
      \includegraphics[scale=1.0]{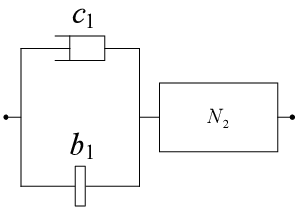}
      \label{fig: R1L1-N2-b}}
      \subfigure[]{
      \includegraphics[scale=1.0]{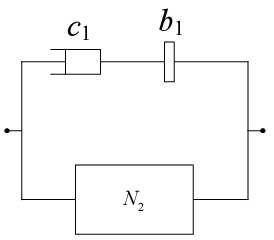}
      \label{fig: R1L1-N2-c}}
      \subfigure[]{
      \includegraphics[scale=1.0]{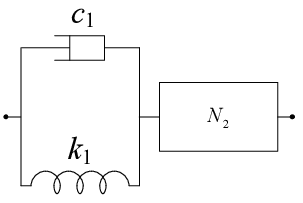}
      \label{fig: R1L1-N2-d}}
      \caption{Classes of five-element series-parallel configurations, where $N_2$ is a three-element series-parallel subnetwork consisting of one damper and two energy storage elements.}
      \label{fig: R1L1-N2}
\end{figure}

For  Fig.~\ref{fig: R1-N2-a}, $N_2$ must consist of one damper and three energy storage elements by Lemmas~\ref{lemma: lossless passive networks} and \ref{lemma: number of elements}. By the principle of frequency inversion,  assume that $N_2$ contains  at least two springs.
Recalling  that $Z(s)$ cannot be realized with fewer than five elements, by Lemma~\ref{lemma: lossless passive networks}, $N_2$ cannot be further decomposed as a parallel connection of two subnetworks. Therefore, the network graph of $N_2$ can only be
one of the graphs in Fig.~\ref{fig: four-element-graph}.
For the   graph  in Fig.~\ref{fig: four-element-graph-03}, to avoid $k$-$\mathcal{P}(a,a')$ and $b$-$\mathcal{P}(a,a')$, Edge~1 and one of Edges~2--4 must correspond to dampers by Lemma~\ref{lemma: graph constraint}.
Since $N_2$ is in parallel with a damper $c_1$,  by the equivalence in Fig.~\ref{fig: Network_Equivalence}, $Z(s)$ is realizable as a four-element series-parallel network, which contradicts the assumption.
For the graph in Fig.~\ref{fig: four-element-graph-04}, recalling that the number of springs is at least two,
one of Edges~1 and 2 and one of Edges~3 and 4 must correspond to springs.
This means that there exists $k$-$\mathcal{P}(a,a')$, which contradicts the assumption by Lemma~\ref{lemma: graph constraint}.
For the graph in Fig.~\ref{fig: four-element-graph-01}, one of Edges~1 and 2 and one of Edges~3 and 4 must correspond to springs, which based on the equivalence in
Fig.~\ref{fig: Network_Equivalence}
implies that the subnetwork $N_2$ can always be equivalent to the subnetwork whose network graph is in Fig.~\ref{fig: four-element-graph-02}. Therefore, one only needs to discuss the graph in  Fig.~\ref{fig: four-element-graph-02}, which can imply all the possible configurations in Figs.~\ref{fig: N1a}, \ref{fig: N2a}, \ref{fig: N3a}, and \ref{fig: N4a} by Lemma~\ref{lemma: graph constraint} and the approach of enumeration.

\begin{figure}[thpb]
      \centering
      \subfigure[]{
      \includegraphics[scale=1.0]{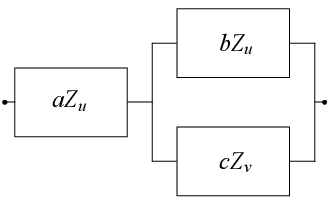}
      \label{fig: Equivalent-a}}
      \subfigure[]{
      \includegraphics[scale=1.0]{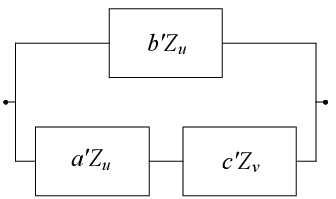}
      \label{fig: Equivalent-b}}
      \caption{Two networks that are equivalent with each other, where $a' = a(a+b)/b$, $b' = a + b$,   $c' = c(a+b)^2/b^2$, and $Z_u$ and $Z_v$ are positive-real impedances (see \cite{Lin65}).}
      \label{fig: Network_Equivalence}
\end{figure}

\begin{figure}[thpb]
      \centering
      \subfigure[]{
      \includegraphics[scale=1.0]{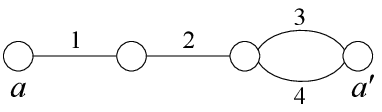}
      \label{fig: four-element-graph-01}}
      \subfigure[]{
      \includegraphics[scale=1.0]{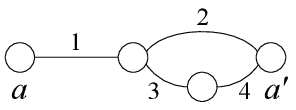}
      \label{fig: four-element-graph-02}}
      \subfigure[]{
      \includegraphics[scale=1.0]{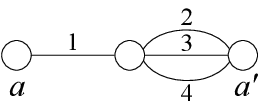}
      \label{fig: four-element-graph-03}}
      \subfigure[]{
      \includegraphics[scale=1.0]{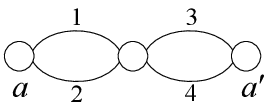}
      \label{fig: four-element-graph-04}}
      \caption{Network graphs of four-element series-parallel subnetworks $N_2$ in Fig.~\ref{fig: R1-N2}, where $a$ and $a'$ are vertices corresponding to two terminals.}
      \label{fig: four-element-graph}
\end{figure}

For   Fig.~\ref{fig: R1L1-N2-a}, $N_2$ must consist of one damper and
two energy storage elements by Lemmas~\ref{lemma: lossless passive networks} and \ref{lemma: number of elements}. By the principle of frequency inversion,  assume that $N_2$ contains  at least one spring. To avoid repeated discussion, the network graph of $N_2$ can   be one of the graphs in Fig.~\ref{fig: three-element-graph}. For the  graph in Fig.~\ref{fig: three-element-graph-01},
the network graph of any configuration for Fig.~\ref{fig: R1L1-N2-a} realizing $Z(s)$ must contain $k$-$\mathcal{C}(a,a')$, which contradicts the assumption by Lemma~\ref{lemma: graph constraint}. For
Fig.~\ref{fig: three-element-graph-02}, all the possible configurations are implied  in Figs.~\ref{fig: N5a} and \ref{fig: N6a} by Lemma~\ref{lemma: graph constraint} and the approach of enumeration.

\begin{figure}[thpb]
      \centering
      \subfigure[]{
      \includegraphics[scale=1.0]{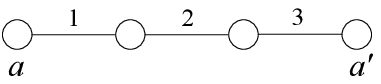}
      \label{fig: three-element-graph-01}}
      \subfigure[]{
      \includegraphics[scale=1.0]{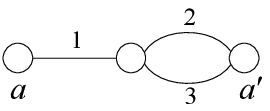}
      \label{fig: three-element-graph-02}}
      \caption{Network graphs of three-element series-parallel subnetworks $N_2$ in Fig.~\ref{fig: R1L1-N2}, where $a$ and $a'$ are vertices corresponding to two terminals.}
      \label{fig: three-element-graph}
\end{figure}

Together with the principles of duality, frequency inversion, and frequency-inverse duality, one can obtain  the  configurations in Figs.~\ref{fig: N1}--\ref{fig: N6} covering all the cases, where the configuration in Fig.~\ref{fig: N1a} (resp. Fig.~\ref{fig: N1b}) whose one-terminal-pair labeled graph is $\mathcal{N}_{1a}$ (resp. $\text{Dual}(\mathcal{N}_{1a})$) can  always be equivalent to the configuration   whose one-terminal-pair labeled graph is $\text{GDu}(\mathcal{N}_{1a})$ (resp. $\text{Inv}(\mathcal{N}_{1a})$) by the equivalence in Fig.~\ref{fig: Network_Equivalence}.

\section{Proof of Lemma~\ref{lemma: realizability condition of N1a}} \label{appendix: N1a}

\setcounter{equation}{0}
\renewcommand{\theequation}{\thesection.\arabic{equation}}

By the principle  of duality, one only needs to prove that  the impedance $Z(s)$ of this lemma is realizable as  in Fig.~\ref{fig: N1a}, if and only if $\mathcal{B}_{13} > 0$, $\mathcal{B}_{12} = 0$, $\mathcal{B}_{23} = 0$,  and $\Delta_1 > 0$.

\emph{Necessity.}
The impedance   of the configuration in Fig.~\ref{fig: N1a} is calculated as
$Z(s) = a(s)/d(s)$, where
$a(s) = c_1^{-1} k_1^{-1} k_2^{-1} b_1  s^3 + c_1^{-1} c_2^{-1} (k_1^{-1} + k_2^{-1}) b_1 s^2 + c_1^{-1} k_1^{-1} s + c_1^{-1} c_2^{-1}$ and
$d(s) = k_1^{-1} k_2^{-1} b_1 s^3 + (c_1^{-1} + c_2^{-1})(k_1^{-1} + k_2^{-1}) b_1 s^2 + k_1^{-1} s + c_1^{-1} + c_2^{-1}$.
Since $Z(s)$ is realizable as the configuration in Fig.~\ref{fig: N1a}, it follows   that
\begin{subequations}
\begin{align}
c_1^{-1} k_1^{-1} k_2^{-1} b_1  &=  x a_3,   \label{eq: N1a eqn1}    \\
c_1^{-1} c_2^{-1} (k_1^{-1} + k_2^{-1}) b_1  &=  x a_2,
 \label{eq: N1a eqn2}    \\
c_1^{-1} k_1^{-1}  &=  x a_1,   \label{eq: N1a eqn3}    \\
c_1^{-1} c_2^{-1}       &=  x a_0,   \label{eq: N1a eqn4}  \\
k_1^{-1} k_2^{-1} b_1   &=  x d_3,   \label{eq: N1a eqn5}    \\
(c_1^{-1} + c_2^{-1})(k_1^{-1} + k_2^{-1}) b_1 &= x d_2,  \label{eq: N1a eqn6}    \\
k_1^{-1}     &=  x d_1,   \label{eq: N1a eqn7}    \\
c_1^{-1} + c_2^{-1}     &=  x d_0,   \label{eq: N1a eqn8}
\end{align}
\end{subequations}
where $x > 0$.
It follows from \eqref{eq: N1a eqn1} and \eqref{eq: N1a eqn5} that the value of $c_1$ can be expressed as in \eqref{eq: N1a element values}. Together with \eqref{eq: N1a eqn3} and \eqref{eq: N1a eqn7}, it is implies that $\mathcal{B}_{23} := a_3 d_1 - a_1 d_3  = 0$. Then, substituting the expression of $c_1$ into  \eqref{eq: N1a eqn4} and \eqref{eq: N1a eqn8} yields
\begin{equation}   \label{eq: N1a x}
x = \frac{a_3^2}{d_3 \mathcal{B}_{13}},
\end{equation}
and the expression of $c_2$ as in \eqref{eq: N1a element values}, which by $c_2 > 0$
implies that $\mathcal{B}_{13} := a_3 d_0 - a_0 d_3 > 0$.
From \eqref{eq: N1a eqn5} and \eqref{eq: N1a eqn7}, one derives that
\begin{equation}  \label{eq: N1a L2C1}
k_2^{-1} b_1  = \frac{d_3}{d_1},
\end{equation}
From \eqref{eq: N1a eqn7} and \eqref{eq: N1a x}, the expression of $k_1$ can be derived as in
\eqref{eq: N1a element values}. By \eqref{eq: N1a eqn2}, \eqref{eq: N1a eqn4}, \eqref{eq: N1a eqn6}, and \eqref{eq: N1a eqn8}, one obtains
$\mathcal{B}_{12} := a_2 d_0 - a_0 d_2 = 0$ and
$(k_1^{-1} + k_2^{-1})b_1  = a_2/a_0$, which together with \eqref{eq: N1a L2C1} and the expression of $k_1$ implies that
\begin{equation} \label{eq: N1a L2 C1 equivalent}
k_2 = \frac{\mathcal{B}_{13}(a_2d_1 - a_0d_3)}{a_0a_3^2d_1}, ~~
b_1  = \frac{d_3 \mathcal{B}_{13} (a_2d_1 - a_0d_3)}{a_0 a_3^2 d_1^2}.
\end{equation}
It follows from $\mathcal{B}_{23} = 0$ that  $a_3 (a_2 d_1 - a_0 d_3) =   d_3 \Delta_1$, which together with \eqref{eq: N1a L2 C1 equivalent} implies that the expressions of $k_2$ and $b_1$ can be further obtained as in \eqref{eq: N1a element values}. Since $\mathcal{B}_{13} > 0$ and $k_2 > 0$, it is implied that
$\Delta_1 := a_1a_2 - a_0a_3  > 0$. Therefore, the necessity part is proved.

\emph{Sufficiency.}
Suppose that $\mathcal{B}_{13} > 0$, $\mathcal{B}_{12} = 0$, $\mathcal{B}_{23} = 0$,  and $\Delta_1 > 0$.
Let the values of the elements satisfy \eqref{eq: N1a element values} and $x$ satisfy \eqref{eq: N1a x}. Then,  $\Delta_1 > 0$ and $\mathcal{B}_{13} > 0$ can guarantee that the element values are positive and finite. Since $\mathcal{B}_{12} = 0$ and $\mathcal{B}_{23} = 0$, it can be verified that conditions~\eqref{eq: N1a eqn1}--\eqref{eq: N1a eqn8} hold. Therefore, $Z(s)$ is realizable as the configuration in Fig.~\ref{fig: N1a}.

\section{Proof of Lemma~\ref{lemma: realizability condition of N6a}}
\label{appendix: N6a}
\setcounter{equation}{0}
\renewcommand{\theequation}{\thesection.\arabic{equation}}

By the principles  of duality, frequency inversion, and frequency-inverse duality,
one only needs to prove that  the impedance $Z(s)$ of this lemma is realizable as the configuration in Fig.~\ref{fig: N6a}, if and only if
 Condition~1 of this lemma holds, where $\zeta_1$ is defined in \eqref{eq: zeta}.

\emph{Necessity.}
The impedance   of the configuration in Fig.~\ref{fig: N6a} is calculated as
$Z(s) = a(s)/d(s)$, where
$a(s) = c_2^{-1} k_1^{-1} k_2^{-1} b_1 s^3 + (c_1^{-1}c_2^{-1}b_1 + k_1^{-1}) k_2^{-1} s^2 + (c_1^{-1}k_2^{-1} + c_2^{-1}k_1^{-1}) s + c_1^{-1}c_2^{-1}$ and $d(s) = k_1^{-1}k_2^{-1}b_1 s^3 + (c_1^{-1}k_2^{-1} + c_2^{-1}k_1^{-1} + c_2^{-1}k_2^{-1}) s^2 + (c_1^{-1}c_2^{-1}b_1 + k_2^{-1}) s + c_2^{-1}$.
Since $Z(s)$ is realizable as the configuration in Fig.~\ref{fig: N6a}, it follows that
\begin{subequations}
\begin{align}
c_2^{-1} k_1^{-1} k_2^{-1} b_1  &=  x a_3,   \label{eq: N6a eqn1}    \\
(c_1^{-1}c_2^{-1}b_1 + k_1^{-1}) k_2^{-1}  &=  x a_2, \label{eq: N6a eqn2}    \\
c_1^{-1}k_2^{-1} + c_2^{-1}k_1^{-1}  &=  x a_1,   \label{eq: N6a eqn3}    \\
c_1^{-1}c_2^{-1}  &=  x a_0,   \label{eq: N6a eqn4}   \\
k_1^{-1}k_2^{-1}b_1  &=   x d_3, \label{eq: N6a eqn5}    \\
c_1^{-1}k_2^{-1} + c_2^{-1}k_1^{-1} + c_2^{-1}k_2^{-1}  &=  x d_2,  \label{eq: N6a eqn6}    \\
c_1^{-1}c_2^{-1}b_1 + k_2^{-1}  &=  x d_1,  \label{eq: N6a eqn7}    \\
c_2^{-1}  &=  x d_0,   \label{eq: N6a eqn8}
\end{align}
\end{subequations}
where $x > 0$. Then, it follows from \eqref{eq: N6a eqn1} and \eqref{eq: N6a eqn5} that
the expression of $c_2$ can be obtained as in \eqref{eq: N6a element values}, which together with \eqref{eq: N6a eqn4} and \eqref{eq: N6a eqn8} implies
\begin{equation} \label{eq: N6a x}
x = \frac{a_3}{d_0d_3},
\end{equation}
and the expression of $c_1$ as in \eqref{eq: N6a element values}. It is implied from \eqref{eq: N6a eqn3} and \eqref{eq: N6a eqn7} that
\begin{equation} \label{eq: N6a L1 C1}
k_1 = \frac{a_3d_0}{a_1a_3 - a_0d_3 k_2^{-1}}, ~~~
b_1 = \frac{a_3d_1 - d_0d_3 k_2^{-1}}{a_0a_3}.
\end{equation}
Substituting \eqref{eq: N6a x}, \eqref{eq: N6a L1 C1}, and the expressions of $c_1$ and $c_2$ into \eqref{eq: N6a eqn6} yields $d_0^2 d_3 k_2^{-2} - d_0 \mathcal{B}_{23} k_2^{-1} - a_3 (a_1d_1-a_0d_2) = 0$, which implies that $\mathcal{M}_{23}^2 - 4a_0a_3d_2d_3 \geq 0$ and
$k_2 = a_0 d_0 d_3 / (a_3 \zeta_1)$ as in \eqref{eq: N6a element values}, where $\zeta_1$ is defined in \eqref{eq: zeta}.
Furthermore, substituting $k_2 = a_0 d_0 d_3 / (a_3 \zeta_1)$ into \eqref{eq: N6a L1 C1} implies that the element values of $k_1$ and $b_1$ can be expressed as in \eqref{eq: N6a element values}.
By the assumption that $k_1 > 0$, $k_2 > 0$, and $b_1 > 0$, it is implied that $0 < \zeta_1 < \min\{a_1d_0, a_0d_1\}$.
Substituting  the element values in \eqref{eq: N6a element values} and $x$ in
\eqref{eq: N6a x}
into \eqref{eq: N6a eqn2} and \eqref{eq: N6a eqn5} can imply
$a_3 \zeta_1^2 - a_0 \mathcal{B}_{23} \zeta_1 - a_0^2d_3(a_1d_1-a_0d_2) = 0$ and
$\zeta_1^3 - \mathcal{M}_{11} \zeta_1^2 + a_0a_1d_0d_1 \zeta_1 - a_0^3 d_0^2 d_3 = 0$. Therefore, the necessity part is proved.

\emph{Sufficiency.} Suppose that Condition~1 of this lemma holds. Let the values of the elements satisfy
\eqref{eq: N6a element values} and $x$ satisfy \eqref{eq: N6a x}. Then, $0 < \zeta_1 < \min\{ a_1d_0, a_0d_1 \}$ and $\mathcal{M}_{23}^2 - 4a_0a_3d_2d_3 \geq 0$
can guarantee that  the element values are positive and finite. Since $a_3 \zeta_1^2 - a_0 \mathcal{B}_{23} \zeta_1 - a_0^2d_3(a_1d_1-a_0d_2) = 0$  and
$\zeta_1^3 - \mathcal{M}_{11} \zeta_1^2 + a_0a_1d_0d_1 \zeta_1 - a_0^3 d_0^2 d_3 = 0$, it can be verified that
conditions~\eqref{eq: N6a eqn1}--\eqref{eq: N6a eqn8} hold.
Therefore, $Z(s)$ is realizable as the configuration in Fig.~\ref{fig: N6a}.

\section{Proof of Lemma~\ref{lemma: third-order non-series-parallel configurations}}
\label{appendix: non-series-parallel configurations}

\emph{Sufficiency.} The sufficiency part is clearly satisfied.

\emph{Necessity.}
Since there is no  pole or zero on $j \mathbb{R} \cup  \infty$, there are at most four energy storage elements.
Together with Lemma~\ref{lemma: number of elements}, the number of energy storage elements is either three or four.  By Lemma~\ref{lemma: graph constraint} and the approach of enumeration, all the possible configurations are  as in Figs.~\ref{fig: N7}--\ref{fig: N11} and \ref{fig: N12}.
It remains to discussing the realizability of $Z(s)$ in this lemma  as the  five-element configurations containing four energy storage elements
in Figs.~\ref{fig: N11} and \ref{fig: N12}.

\begin{figure}[thpb]
      \centering
      \subfigure[]{
      \includegraphics[scale=0.9]{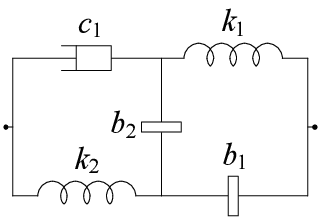}
      \label{fig: N12a}}
      \subfigure[]{
      \includegraphics[scale=0.9]{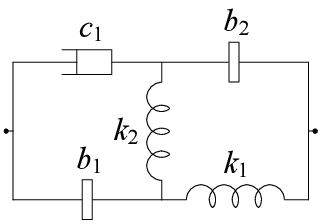}
      \label{fig: N12b}}
      \caption{Five-element non-series-parallel configurations that cannot realize the bicubic impedance $Z(s)$ in \eqref{eq: three-degree impedance}, whose one-terminal-pair labeled graphs are (a) $\mathcal{N}_{12a}$ and (b) $\mathcal{N}_{12b}$, respectively,  satisfying $\mathcal{N}_{12b} = \text{Dual}(\mathcal{N}_{12a})$, where $\text{Inv}(\mathcal{N}_{12a}) = \text{Dual}(\mathcal{N}_{12a})$
      and $\text{GDu}(\mathcal{N}_{12a}) = \mathcal{N}_{12a}$.}
      \label{fig: N12}
\end{figure}

The impedance of the five-element configuration containing four energy storage elements in Fig.~\ref{fig: N11} is calculated as $Z(s) = a(s)/d(s)$, where $a(s) = c_1^{-1} k_1^{-1} k_2^{-1} b_1 b_2 s^4 + k_1^{-1}k_2^{-1} (b_1 + b_2) s^3 + c_1^{-1} (k_1^{-1}b_1 + k_2^{-1}b_2) s^2 + (k_1^{-1} + k_2^{-1}) s + c_1^{-1}$ and $d(s) = k_1^{-1}k_2^{-1}b_1b_2 s^4 + c_1^{-1} (k_1^{-1} + k_2^{-1}) b_1 b_2 s^3 + (k_1^{-1} b_2 + k_2^{-1} b_1) s^2 + c_1^{-1}(b_1 + b_2) s + 1$.
Choosing $k_1 = 1/16$~$\text{N}/\text{m}$, $k_2 = 1$~$\text{N}/\text{m}$, $b_1 = 1$~$\text{kg}$, $b_2 = 1$~$\text{kg}$, and
$c_1 = 1/2$~$\text{Ns}/\text{m}$,   the impedance of  Fig.~\ref{fig: N11} is $Z(s) = (16 s^3+8 s^2+13 s+2)/(8 s^3+13 s^2+2 s+1)$, which  satisfies the assumption.
Therefore,
the configuration in Fig.~\ref{fig: N11} can    realize the bicubic impedance in \eqref{eq: three-degree impedance} for some element values.

The impedance of the five-element configuration containing four energy storage elements in Fig.~\ref{fig: N12a} is calculated as $Z(s) = a(s)/d(s)$, where $a(s) = c_1^{-1}k_1^{-1}k_2^{-1}b_1b_2 s^4 + k_1^{-1}k_2^{-1}(b_1 + b_2) s^3 + c_1^{-1} (k_1^{-1} b_2 + k_2^{-1} b_1 + k_2^{-1} b_2) s^2 + k_1^{-1} s + c_1^{-1}$ and $d(s) = k_1^{-1}k_2^{-1}b_1b_2 s^4 + c_1^{-1} k_1^{-1} b_1b_2 s^3 + (k_1^{-1}b_1 + k_2^{-1}b_1 + k_2^{-1}b_2) s^2 + c_1^{-1} (b_1 + b_2) s + 1$.
Assume that  a bicubic impedance of this lemma   is realizable as the configuration in Fig.~\ref{fig: N12a}.
Then, the \emph{resultant} \cite[Chapter XV]{Gan80} of $a(s)$ and $d(s)$ in $s$   calculated as
$R_0(a, b, s) = k_1^{-3} k_2^{-1} b_1 b_2 h_2^2$ must be zero, where
$h_2 := k_1^{-1} b_1 b_2^3 c_1^{-4} + k_2^{-1} (b_1 + b_2) (k_2^{-1}(b_1 + b_2)^2 - 3k_1^{-1}b_1b_2) c_1^{-2} + k_1^{-2}k_2^{-1}b_1^2$.
If $k_2^{-1} (b_1 + b_2)^2 - 3 k_1^{-1}b_1b_2 \geq 0$, then $h_2 > 0$, which implies that $R_0(a, b, s) > 0$.
This contradicts the assumption.
If $k_2^{-1} (b_1 + b_2)^2 - 3k_1^{-1}b_1b_2 < 0$, then  the discriminant of the equation $h_2 = 0$ in $c_1^{-2}$ is calculated as
$(k_2^{-1}(b_1+b_2)(k_2^{-1}(b_1+b_2)^2 -
3k_1^{-1} b_1 b_2))^2 - 4 k_1^{-1} b_1 b_2^3 k_1^{-2} k_2^{-1} b_1^2 = k_2^{-1} (k_2^{-1} (b_1+b_2)^2 -
4 k_1^{-1} b_1 b_2)(k_2^{-1} (b_1+b_2)^2 - k_1^{-1} b_1 b_2)^2 \leq 0$.
Since $h_2 = 0$ must have a positive root in $c_2^{-2}$, it is implied that
$k_2^{-1}(b_1 + b_2)^2 - k_1^{-1} b_1 b_2 = 0$,  which together with $h_2 = 0$ implies
$c_1^2 = b_2^{-2}/(k_2^{-1}(b_1 + b_2))$. Therefore,  the impedance becomes
$Z(s) = c_1^{-1}(c_1^{-2}b_1b_2^2 s^2 + c_1^{-1}b_2^2 s + b_1)/(b_1(c_1^{-2}b_2^2 s^2 + c_1^{-1}b_1 s + 1))$, whose McMillan degree is at most two. By contradiction, any bicubic impedance of this lemma cannot be realized as the configuration in Fig.~\ref{fig: N12a}, which cannot be realized as the configuration in Fig.~\ref{fig: N12b} by the principle of duality.

\section{Proof of Lemma~\ref{lemma: realizability condition of N8a}}
\label{appendix: N8a}
\setcounter{equation}{0}
\renewcommand{\theequation}{\thesection.\arabic{equation}}

By  the principle  of duality,
one only needs to prove that  the impedance $Z(s)$ of this lemma is realizable as the configuration in Fig.~\ref{fig: N8a}, if and only if Condition~1 or Condition~2 of this lemma holds, where $\Lambda_{1a}$ and $\Lambda_{1b}$  are defined in \eqref{eq: Lambda1a} and \eqref{eq: Lambda1b}, respectively.

\emph{Necessity.}
The impedance   of the configuration in Fig.~\ref{fig: N8a} is calculated as $Z(s) = a(s)/d(s)$, where
$a(s) = c_2^{-1} k_1^{-1} k_2^{-1} b_1 s^3 + (c_1^{-1}c_2^{-1}(k_1^{-1} + k_2^{-1})b_1 + k_1^{-1}k_2^{-1}) s^2 + (c_1^{-1}(k_1^{-1}+k_2^{-1}) + c_2^{-1}k_1^{-1}) s + c_1^{-1}c_2^{-1}$  and
$d(s) = k_1^{-1}k_2^{-1}b_1 s^3 + (c_1^{-1}(k_1^{-1}+k_2^{-1})+c_2^{-1}k_2^{-1})b_1 s^2 + (c_1^{-1}c_2^{-1}b_1 + k_1^{-1} + k_2^{-1}) s + c_2^{-1}$.
Then,
\begin{subequations}
\begin{align}
c_2^{-1} k_1^{-1} k_2^{-1} b_1   &=  x a_3,   \label{eq: N8a eqn1}    \\
c_1^{-1}c_2^{-1}(k_1^{-1} + k_2^{-1})b_1 + k_1^{-1}k_2^{-1}  &=  x a_2, \label{eq: N8a eqn2}    \\
c_1^{-1}(k_1^{-1}+k_2^{-1}) + c_2^{-1}k_1^{-1}   &=  x a_1,   \label{eq: N8a eqn3}    \\
c_1^{-1}c_2^{-1}    &=  x a_0,   \label{eq: N8a eqn4}  \\
k_1^{-1}k_2^{-1}b_1     &=   x d_3, \label{eq: N8a eqn5}    \\
(c_1^{-1}(k_1^{-1}+k_2^{-1})+c_2^{-1}k_2^{-1})b_1    &=  x d_2,  \label{eq: N8a eqn6}   \\
c_1^{-1}c_2^{-1}b_1 + k_1^{-1} + k_2^{-1}  &=  x d_1,  \label{eq: N8a eqn7}    \\
c_2^{-1}  &=  x d_0,   \label{eq: N8a eqn8}
\end{align}
\end{subequations}
where $x > 0$. It follows from \eqref{eq: N8a eqn1} and \eqref{eq: N8a eqn5} that the expression of $c_2$ can be obtained as in  \eqref{eq: N8a element values}, together with \eqref{eq: N8a eqn4} and \eqref{eq: N8a eqn8} yields the expression of $c_1$ as in \eqref{eq: N8a element values} and
\begin{equation} \label{eq: N8a x}
x = \frac{a_3}{d_0d_3}.
\end{equation}
Substituting \eqref{eq: N8a x} and the expressions of $c_1$ and $c_2$ into \eqref{eq: N8a eqn7} yields
\begin{equation} \label{eq: N8a L1plusL2}
k_1^{-1} + k_2^{-1} = \frac{a_3(d_1 - a_0 b_1)}{d_0d_3},
\end{equation}
which together with \eqref{eq: N8a eqn3} and \eqref{eq: N8a eqn5} implies the expressions of $k_1$ and $k_2$ as in \eqref{eq: N8a element values}. Substituting \eqref{eq: N8a x},   and the expressions of $c_2$, $k_1$ and $k_2$ into \eqref{eq: N8a eqn2} implies
\begin{equation} \label{eq: N8a equations}
c_1^{-1}(k_1^{-1} + k_2^{-1})b_1 = \frac{a_2 b_1 - d_3}{d_0 b_1}.
\end{equation}
Substituting \eqref{eq: N8a x}, \eqref{eq: N8a equations}, and the expressions of $c_2$ and $k_2$ into \eqref{eq: N8a eqn6} implies $a_0^2 \mathcal{B}_{33} b_1^2 - (\mathcal{B}_{13} \mathcal{M}_{13} - \mathcal{B}_{11} \mathcal{B}_{33}) b_1 + d_3^2 \mathcal{B}_{11} = 0$.
If $\mathcal{B}_{33} = 0$, then $b_1$  must satisfy $b_1 =  \Lambda_{1a}$, where $\Lambda_{1a}$ is defined in \eqref{eq: Lambda1a}. Then, $b_1 > 0$ implies that $\mathcal{B}_{11} \mathcal{B}_{13} > 0$.
Together with the element values in \eqref{eq: N8a element values} and $x$ in \eqref{eq: N8a x}, it follows from \eqref{eq: N8a eqn2} and \eqref{eq: N8a L1plusL2} that  $a_0^2 a_3 \Lambda_{1a}^3 - a_0a_3d_1 \Lambda_{1a}^2 + a_2d_0d_3 \Lambda_{1a} - d_0d_3^2 = 0$  and
  $a_0^3 \mathcal{M}_{13} \Lambda_{1a}^3 + a_0(2 \mathcal{B}_{11} \mathcal{M}_{13} + a_1a_3d_0^2) \Lambda_{1a}^2 + \mathcal{B}_{11} (d_3 \mathcal{B}_{11} - a_3d_0d_1) \Lambda_{1a} + a_3d_0^3d_3 = 0$.
If $\mathcal{B}_{33} \neq 0$, then $(\mathcal{B}_{13}^2 - \mathcal{B}_{11} \mathcal{B}_{33})(\mathcal{M}_{13}^2 - \mathcal{B}_{11} \mathcal{B}_{33}) \geq 0$, and $b_1$ can be solved as $b_1 = \Lambda_{1b}$, where $\Lambda_{1b}$ is defined in \eqref{eq: Lambda1b},
which by $k_1 > 0$ and $k_2 > 0$ implies that  $a_0^2 \Lambda_{1b} > \max\{0, -\mathcal{B}_{11} \}$. Similarly, together with the element values in \eqref{eq: N8a element values} and $x$ in \eqref{eq: N8a x}, it follows from \eqref{eq: N8a eqn2} and \eqref{eq: N8a L1plusL2} that  $a_0^2 a_3 \Lambda_{1b}^3 - a_0a_3d_1 \Lambda_{1b}^2 + a_2d_0d_3 \Lambda_{1b} - d_0d_3^2 = 0$  and
  $a_0^3 \mathcal{M}_{13} \Lambda_{1b}^3 + a_0(2 \mathcal{B}_{11} \mathcal{M}_{13} + a_1a_3d_0^2) \Lambda_{1b}^2 + \mathcal{B}_{11} (d_3 \mathcal{B}_{11} - a_3d_0d_1) \Lambda_{1b} + a_3d_0^3d_3 = 0$.

\emph{Sufficiency.} Suppose that Condition~1 of this lemma holds. Let $c_1$, $c_2$, $k_1$, and $k_2$ satisfy
\eqref{eq: N8a element values},  $b_1 = \Lambda_{1a}$, and $x$ satisfy \eqref{eq: N8a x}. Then, $\mathcal{B}_{11} \mathcal{B}_{13} > 0$ implies that $\Lambda_{1a} = d_3^3 \mathcal{B}_{11}/(\mathcal{B}_{13} \mathcal{M}_{13}) > 0$ and $a_0^2 \Lambda_{1a} + \mathcal{B}_{11} = a_3^2d_0^2 \mathcal{B}_{11}/(\mathcal{B}_{13} \mathcal{M}_{13}) > 0$, which
can guarantee that  the element values are positive and finite.
Since $\mathcal{B}_{33} = 0$, $a_0^2 a_3 \Lambda_{1a}^3 - a_0a_3d_1 \Lambda_{1a}^2 + a_2d_0d_3 \Lambda_{1a} - d_0d_3^2 = 0$, and
  $a_0^3 \mathcal{M}_{13} \Lambda_{1a}^3 + a_0(2 \mathcal{B}_{11} \mathcal{M}_{13} + a_1a_3d_0^2) \Lambda_{1a}^2 + \mathcal{B}_{11} (d_3 \mathcal{B}_{11} - a_3d_0d_1) \Lambda_{1a} + a_3d_0^3d_3 = 0$, it can be verified that conditions~\eqref{eq: N8a eqn1}--\eqref{eq: N8a eqn8} hold.
Therefore, $Z(s)$ is realizable as the configuration in Fig.~\ref{fig: N8a}.
Similar discussions can be made    when  Condition~2 of this lemma holds.

\section{Proof of Lemma~\ref{lemma: realizability condition of N9a}}
\label{appendix: N9a}
\setcounter{equation}{0}
\renewcommand{\theequation}{\thesection.\arabic{equation}}

By  the principle  of duality,
one only needs to prove that  the impedance $Z(s)$ of this lemma is realizable as the configuration in Fig.~\ref{fig: N9a}, if and only if
Condition~1 of this lemma holds, where $\Gamma_1$ and $\Phi_1$ are defined in \eqref{eq: Gamma1} and \eqref{eq: Phi1}, respectively.

\emph{Necessity.}
The impedance   of the configuration in Fig.~\ref{fig: N9a} is calculated as $Z(s) = a(s)/d(s)$, where
$a(s) = (c_1^{-1} + c_2^{-1})k_1^{-1}k_2^{-1}b_1 s^3 + (c_1^{-1}c_2^{-1}(k_1^{-1}+k_2^{-1})b_1 + k_1^{-1}k_2^{-1}) s^2 + (c_1^{-1}k_2^{-1} + c_2^{-1}k_1^{-1}) s + c_1^{-1} c_2^{-1}$ and $d(s) = k_1^{-1}k_2^{-1}b_1 s^3 + (c_1^{-1}k_1^{-1} + c_2^{-1} k_2^{-1}) b_1 s^2 + (c_1^{-1}c_2^{-1}b_1 + k_1^{-1} + k_2^{-1}) s + c_1^{-1} + c_2^{-1}$.
Then,
\begin{subequations}
\begin{align}
(c_1^{-1} + c_2^{-1})k_1^{-1}k_2^{-1}b_1    &=  x a_3,   \label{eq: N9a eqn1}    \\
c_1^{-1}c_2^{-1}(k_1^{-1}+k_2^{-1})b_1 + k_1^{-1}k_2^{-1}  &=  x a_2, \label{eq: N9a eqn2}    \\
c_1^{-1}k_2^{-1} + c_2^{-1}k_1^{-1}  &=  x a_1,   \label{eq: N9a eqn3}    \\
c_1^{-1} c_2^{-1}  &=  x a_0,   \label{eq: N9a eqn4}   \\
k_1^{-1}k_2^{-1}b_1  &=   x d_3, \label{eq: N9a eqn5}    \\
(c_1^{-1}k_1^{-1} + c_2^{-1} k_2^{-1}) b_1  &=  x d_2,  \label{eq: N9a eqn6}    \\
c_1^{-1}c_2^{-1}b_1 + k_1^{-1} + k_2^{-1}  &=  x d_1,  \label{eq: N9a eqn7}    \\
c_1^{-1} + c_2^{-1}  &=  x d_0,   \label{eq: N9a eqn8}
\end{align}
\end{subequations}
where $x > 0$. It follows from \eqref{eq: N9a eqn1} and \eqref{eq: N9a eqn5} that
\begin{equation} \label{eq: N9a R1 + R2}
c_1^{-1} + c_2^{-1} = \frac{a_3}{d_3},
\end{equation}
which together with \eqref{eq: N9a eqn8} implies that
\begin{equation} \label{eq: N9a x}
x = \frac{a_3}{d_0d_3}.
\end{equation}
By \eqref{eq: N9a eqn4}, \eqref{eq: N9a R1 + R2}, and \eqref{eq: N9a x}, one implies that
$d_0 d_3 c_1^{-2} - a_3 d_0 c_1^{-1} + a_0 a_3 = 0$,
  $c_2^{-1} = a_0a_3/(d_0d_3c_1^{-1})$, and  $c_2^{-1} = a_3/d_3 - c_1^{-1}$.
Furthermore, $a_3 d_0 - 4a_0 d_3 \geq 0$, and one obtains
$c_1 = 1/\Gamma_1$ and $c_2 = d_0d_3 \Gamma_1/(a_0a_3)$ as in \eqref{eq: N9a element values}, where
$\Gamma_1$ is defined in \eqref{eq: Gamma1}. Substituting \eqref{eq: N9a eqn4} and \eqref{eq: N9a x} into \eqref{eq: N9a eqn3} and \eqref{eq: N9a eqn5} yields
\begin{equation} \label{eq: N9a L2 & C1}
k_2 = \frac{d_0d_3 c_1^{-2}}{a_3(a_1 c_1^{-1} - a_0 k_1^{-1})}, ~~
b_1 = \frac{d_3 c_1^{-2}}{(a_1 c_1^{-1} - a_0 k_1^{-1}) k_1^{-1}}.
\end{equation}
Substituting  \eqref{eq: N9a x}, \eqref{eq: N9a L2 & C1}, $c_1 = 1/\Gamma_1$, and $c_2 = d_0d_3 \Gamma_1/(a_0a_3)$ into
\eqref{eq: N9a eqn6}
yields
$a_0a_3d_0d_2 \Gamma_1 k_1^{-2} + (d_0^2d_3^2 \Gamma_1^4 - a_1a_3d_0d_2 \Gamma_1^2 -
a_0^2a_3^2) k_1^{-1} + a_0a_1a_3^2 \Gamma_1 = 0$, which implies that
$(a_1a_3d_0d_2\Gamma_1^2 - (d_0d_3\Gamma_1^2 - a_0a_3)^2)(a_1a_3d_0d_2\Gamma_1^2
- (d_0d_3\Gamma_1^2 + a_0a_3)^2) \geq 0$ and $k_1 = 1/\Phi_1$  as in \eqref{eq: N9a element values}, where $\Phi_1$ is defined in
\eqref{eq: Phi1}.
Furthermore, substituting   $c_1 = 1/\Gamma_1$, $c_2 = d_0d_3 \Gamma_1/(a_0a_3)$, and $k_1 = 1/\Phi_1$   into \eqref{eq: N9a L2 & C1} implies
the expressions of $k_2$ and $b_1$ as in \eqref{eq: N9a element values}. The assumption that the element values are positive and finite implies that $0 < \Phi_1/\Gamma_1 < a_1/a_0$.
By    the  element values in \eqref{eq: N9a element values} and $x$ in
\eqref{eq: N9a x}, it follows from \eqref{eq: N9a eqn2} and \eqref{eq: N9a eqn7} that
$a_0^2d_0 \Phi_1^4 - 2 a_0a_1d_0 \Gamma_1 \Phi_1^3 +
d_0(a_1^2+a_0a_2)\Gamma_1^2 \Phi_1^2 + (a_0d_0d_3 \Gamma_1^2 - a_1a_2
d_0 \Gamma_1 - a_0^2a_3)\Gamma_1^2 \Phi_1 + a_0a_1a_3 \Gamma_1^3 = 0$ and
$a_0(d_0d_3\Gamma_1^2 - a_0a_3)\Phi_1^3 - (a_1d_0d_3 \Gamma_1^2 + a_0a_3d_1
\Gamma_1 - 2a_0a_1a_3) \Gamma_1 \Phi_1^2 + a_1a_3(d_1\Gamma_1 - a_1)\Gamma_1^2 \Phi_1 - a_0a_3d_3 \Gamma_1^4 = 0$
 hold.

\emph{Sufficiency.} Suppose that Condition~1 of this lemma holds. Let the values of the elements satisfy
\eqref{eq: N9a element values} and $x$ satisfy \eqref{eq: N9a x}.
Since $0 < \Phi_1/\Gamma_1 < a_1/a_0$, $a_3d_0 - 4a_0d_3 \geq 0$,
  $(a_1a_3d_0d_2\Gamma_1^2 - (d_0d_3\Gamma_1^2 - a_0a_3)^2)(a_1a_3d_0d_2\Gamma_1^2
- (d_0d_3\Gamma_1^2 + a_0a_3)^2) \geq 0$, it is clear that the element values as in \eqref{eq: N9a element values} can be positive and finite.
Since $a_0(d_0d_3\Gamma_1^2 - a_0a_3)\Phi_1^3 - (a_1d_0d_3 \Gamma_1^2 + a_0a_3d_1
\Gamma_1 - 2a_0a_1a_3) \Gamma_1 \Phi_1^2 + a_1a_3(d_1\Gamma_1 - a_1)\Gamma_1^2 \Phi_1 - a_0a_3d_3 \Gamma_1^4 = 0$, and
$a_0^2d_0 \Phi_1^4 - 2 a_0a_1d_0 \Gamma_1 \Phi_1^3 +
d_0(a_1^2+a_0a_2)\Gamma_1^2 \Phi_1^2 + (a_0d_0d_3 \Gamma_1^2 - a_1a_2
d_0 \Gamma_1 - a_0^2a_3)\Gamma_1^2 \Phi_1 + a_0a_1a_3 \Gamma_1^3 = 0$,   it can be verified that
conditions~\eqref{eq: N9a eqn1}--\eqref{eq: N9a eqn8} hold.
Therefore, $Z(s)$ is realizable as the configuration in Fig.~\ref{fig: N9a}.

\section{Proof of Lemma~\ref{lemma: realizability condition of N11}}
\label{appendix: N11}

\setcounter{equation}{0}
\renewcommand{\theequation}{\thesection.\arabic{equation}}

\emph{Necessity.}
The impedance   of the configuration in Fig.~\ref{fig: N11} is calculated as $Z(s) = a(s)/d(s)$, where
$a(s) = k_1^{-1}k_2^{-1}b_1b_2 s^4 + c_1^{-1} (k_1^{-1}+k_2^{-1}) b_1 b_2 s^3 + (k_1^{-1}b_2 + k_2^{-1}b_1) s^2 + c_1^{-1}(b_1 + b_2) s + 1$ and $d(s) = c_1^{-1} k_1^{-1} k_2^{-1} b_1 b_2 s^4 + k_1^{-1} k_2^{-1} (b_1 + b_2) s^3 + c_1^{-1} (k_1^{-1} b_1 + k_2^{-1} b_2) s^2 + (k_1^{-1} + k_2^{-1}) s + c_1^{-1}$.
Then, multiplying the numerator and denominator of the bicubic impedance $Z(s)$ in \eqref{eq: three-degree impedance} with a common factor $(Ts + 1)$ where  $T > 0$, it follows that
\begin{subequations}
\begin{align}
c_1^{-1}k_1^{-1}k_2^{-1}b_1b_2 &= x a_3 T,   \label{eq: N11 eqn1}   \\
k_1^{-1}k_2^{-1}(b_1 + b_2)   &=  x (a_2 T + a_3), \label{eq: N11 eqn2}    \\
c_1^{-1}(k_1^{-1}b_1 + k_2^{-1}b_2) &=  x (a_1 T + a_2),   \label{eq: N11 eqn3}    \\
k_1^{-1} + k_2^{-1} &=  x (a_0 T + a_1),   \label{eq: N11 eqn4}  \\
c_1^{-1} &=  x a_0,       \label{eq: N11 eqn5}    \\
k_1^{-1}k_2^{-1}b_1b_2     &=   x d_3 T, \label{eq: N11 eqn6}    \\
c_1^{-1}(k_1^{-1}+k_2^{-1})b_1b_2     &=  x (d_2T+d_3),  \label{eq: N11 eqn7}    \\
k_1^{-1}b_2 + k_2^{-1}b_1    &=  x (d_1 T + d_2),  \label{eq: N11 eqn8}    \\
c_1^{-1}(b_1 + b_2)    &=  x (d_0 T + d_1),   \label{eq: N11 eqn9}    \\
1 &= x d_0,  \label{eq: N11 eqn10}
\end{align}
\end{subequations}
where $x > 0$. It follows from \eqref{eq: N11 eqn1} and \eqref{eq: N11 eqn6} that the expression of $c_1$ can be expressed as in \eqref{eq: N11 element values}. The expression of $x$ can be directly obtained from \eqref{eq: N11 eqn10} as
\begin{equation} \label{eq: N11 x}
x = \frac{1}{d_0}.
\end{equation}
Then, substituting \eqref{eq: N11 x} and the expression of $c_1$ into \eqref{eq: N11 eqn5} implies that $\mathcal{B}_{13} = 0$.
Furthermore, it follows from \eqref{eq: N11 eqn4} that
\begin{equation}  \label{eq: N11 L1plusL2}
k_1^{-1} + k_2^{-1} = \frac{a_0 T + a_1}{d_0},
\end{equation}
which together with \eqref{eq: N11 eqn7}, \eqref{eq: N11 x}, and the expression of $c_1$ implies that
\begin{equation}  \label{eq: N11 C1C2}
b_1 b_2 = \frac{d_3 (d_2 T + d_3)}{a_3 (a_0 T + a_1)}.
\end{equation}
Then, substituting \eqref{eq: N11 x} and \eqref{eq: N11 C1C2} into \eqref{eq: N11 eqn6} yields
\begin{equation}  \label{eq: N11 L1L2}
k_1^{-1} k_2^{-1} = \frac{a_3 T (a_0 T + a_1)}{d_0 (d_2 T + d_3)},
\end{equation}
which together with \eqref{eq: N11 eqn2} and \eqref{eq: N11 x} implies that
\begin{equation}  \label{eq: N11 C1plusC2}
b_1 + b_2 = \frac{(a_2 T + a_3)(d_2 T + d_3)}{a_3 T (a_0 T + a_1)}.
\end{equation}
By \eqref{eq: N11 L1plusL2} and \eqref{eq: N11 L1L2}, it is implied that $k_1 = 1/y_1$ and $k_2 = 1/y_2$ are two positive roots of   equation \eqref{eq: N11 L1 L2 equation} in $y$, whose discriminant must be nonnegative. Therefore, it follows that
$a_0d_2 T^2 + (a_1d_2 - 3a_0d_3) T + a_1d_3 \geq 0$.
Similarly, by \eqref{eq: N11 C1C2} and \eqref{eq: N11 C1plusC2}, it is implied that $b_1 = z_1$ and $b_2 = z_2$ are two positive roots of   equation
\eqref{eq: N11 C1 C2 equation} in $z$, whose discriminant must be nonnegative. Therefore, it follows that $(a_2^2d_2 - 4a_0a_3d_3) T^3 + (a_2^2 d_3 + 2a_2a_3d_2 - 4a_1a_3d_3) T^2 + a_3(a_3d_2 + 2a_2d_3) T + a_3^2d_3 \geq 0$. Substituting \eqref{eq: N11 x}, \eqref{eq: N11 C1plusC2}, and the expression of $c_1$ into   \eqref{eq: N11 eqn9} implies $a_0d_0d_3 T^3 + (a_1d_0d_3 + a_0d_1d_3 - a_2d_0d_2) T^2 +
(a_1 d_1 d_3 - a_2 d_0 d_3 - a_3 d_0 d_2) T - a_3 d_0 d_3 = 0$. Substituting \eqref{eq: N11 eqn5}, \eqref{eq: N11 x}, $k_1 = 1/y_1$,
$k_2 = 1/y_2$, $b_1 = z_1$, and $b_2 = z_2$ into \eqref{eq: N11 eqn3} and \eqref{eq: N11 eqn8} implies $a_1 T + a_2 - a_0(y_1z_1 + y_2z_2) = 0$ and $d_1 T + d_2 - d_0(y_1z_2 + y_2z_1) = 0$.

\emph{Sufficiency.} Suppose that the condition of this lemma holds. Then,  $a_0d_2 T^2 + (a_1d_2 - 3a_0d_3) T + a_1d_3 \geq 0$ can imply that    equation \eqref{eq: N11 L1 L2 equation} in $y$ has two positive roots denoted as $y_1$ and $y_2$, and $(a_2^2d_2 - 4a_0a_3d_3) T^3 + (a_2^2 d_3 + 2a_2a_3d_2 - 4a_1a_3d_3) T^2 + a_3(a_3d_2 + 2a_2d_3) T + a_3^2d_3 \geq 0$ can imply that equation  \eqref{eq: N11 C1 C2 equation} in $z$ has two positive roots denoted as $z_1$ and $z_2$.
Let the element values satisfy \eqref{eq: N11 element values}, and $x$ satisfy \eqref{eq: N11 x}, which implies that the element values can be positive and finite. Since $\mathcal{B}_{13} = 0$,
 $a_0d_0d_3 T^3 + (a_1d_0d_3 + a_0d_1d_3 - a_2d_0d_2) T^2 +
(a_1 d_1 d_3 - a_2 d_0 d_3 - a_3 d_0 d_2) T - a_3 d_0 d_3 = 0$, $a_1 T + a_2 - a_0(y_1z_1 + y_2z_2) = 0$, and $d_1 T + d_2 - d_0(y_1z_2 + y_2z_1) = 0$ imply that \eqref{eq: N11 L1plusL2}--\eqref{eq: N11 C1plusC2} hold,
it can be verified that
conditions~\eqref{eq: N11 eqn1}--\eqref{eq: N11 eqn10} hold.
Therefore, $Z(s)$ is realizable as the configuration in Fig.~\ref{fig: N11}.

\end{appendices}

\newpage

\setcounter{figure}{0}
\setcounter{equation}{0}
\setcounter{section}{0}

\numberwithin{equation}{section}
\numberwithin{figure}{section}
\numberwithin{lemma}{section}
\numberwithin{theorem}{section}
\numberwithin{definition}{section}

\begin{center}
\LARGE Supplementary Material to: Passive Mechanical Realizations of Bicubic Impedances with No More Than Five Elements for Inerter-Based Control Design
\end{center}

\begin{center}
\large Kai~Wang  ~ and ~Michael~Z.~Q.~Chen 
\end{center}

\vspace{0.5cm}

\section{Introduction}
This report presents the proofs of some results in the paper    entitled   ``Passive network realizations of bicubic impedances with no more than five elements for inerter-based control design'' \cite{WC_sub}, which are omitted from the paper for brevity. It is assumed that
the numbering of lemmas, theorems, equations and figures in this report agrees with that in the original paper.

\section{Proof of Lemma~7}

By the principles of duality, frequency inversion, and frequency-inverse duality,
one only needs to prove that  the impedance $Z(s)$ of this lemma is realizable as the configuration in Fig.~2(a), if and only if  $\mathcal{B}_{13} = 0$, $\Delta_1 > 0$, and   $a_0 \mathcal{B}_{33} = a_1 \mathcal{B}_{23} > 0$.

\emph{Necessity.}
Suppose that $Z(s)$ is realizable as the configuration in Fig.~2(a). The impedance   of the configuration in Fig.~2(a) is calculated as
$Z(s) = a(s)/d(s)$, where
$a(s) = c_1^{-1} k_1^{-1} k_2^{-1} b_1 s^3 + c_1^{-1}c_2^{-1}k_1^{-1}b_1 s^2 + c_1^{-1}(k_1^{-1} + k_2^{-1}) s + c_1^{-1} c_2^{-1}$
and
$d(s) = k_1^{-1}k_2^{-1}b_1 s^3 + (c_1^{-1}k_1^{-1} + c_1^{-1}k_2^{-1} + c_2^{-1}k_1^{-1})b_1 s^2 + (c_1^{-1}c_2^{-1}b_1 + k_1^{-1} + k_2^{-1}) s + c_2^{-1}$.
Then, it follows   that
\begin{subequations}
\begin{align}
c_1^{-1} k_1^{-1} k_2^{-1} b_1  &=  x a_3,   \label{eq: N2a eqn1}    \\
c_1^{-1}c_2^{-1}k_1^{-1}b_1  &=  x a_2,   \label{eq: N2a eqn2}    \\
c_1^{-1}(k_1^{-1} + k_2^{-1})  &=  x a_1,   \label{eq: N2a eqn3}    \\
c_1^{-1} c_2^{-1}  &=  x a_0,   \label{eq: N2a eqn4}   \\
k_1^{-1}k_2^{-1}b_1  &=   x d_3, \label{eq: N2a eqn5}    \\
(c_1^{-1}k_1^{-1} + c_1^{-1}k_2^{-1} + c_2^{-1}k_1^{-1})b_1  &=  x d_2,  \label{eq: N2a eqn6}    \\
c_1^{-1}c_2^{-1}b_1 + k_1^{-1} + k_2^{-1}  &=   x d_1,   \label{eq: N2a eqn7}    \\
c_2^{-1}  &=   x d_0,   \label{eq: N2a eqn8}
\end{align}
\end{subequations}
where $x > 0$. Then, it follows from \eqref{eq: N2a eqn1} and \eqref{eq: N2a eqn5} that the expression of $c_1$ can be obtained as in (5), which together with  \eqref{eq: N2a eqn4}  and \eqref{eq: N2a eqn8} can imply that $\mathcal{B}_{13} := a_3 d_0 - a_0 d_3 = 0$.
Substituting  \eqref{eq: N2a eqn2},   \eqref{eq: N2a eqn3}, and the expression of $c_1$ into \eqref{eq: N2a eqn6} yields the expression of $b_1$ as in (5), which
by $b_1 > 0$
implies that $\mathcal{B}_{33} > 0$.
Substituting
\eqref{eq: N2a eqn3},   \eqref{eq: N2a eqn8}, and the expressions of $c_1$ and $b_1$
into \eqref{eq: N2a eqn7} yields
$a_3d_0 \mathcal{B}_{33} - a_1d_3 \mathcal{B}_{23} := a_3 d_0 (a_3 d_2 - a_2 d_3) - a_1d_3 (a_3d_1 - a_1d_3) = 0$, which is further equivalent to $a_0 \mathcal{B}_{33} = a_1 \mathcal{B}_{23}$ by $\mathcal{B}_{13} = 0$.
By   \eqref{eq: N2a eqn2},   \eqref{eq: N2a eqn8}, and the expressions of $c_1$ and $b_1$, one can derive the expression of $k_1$ as in (5),
which together with   \eqref{eq: N2a eqn3}, \eqref{eq: N2a eqn5},
$\mathcal{B}_{13} = 0$, and the expression of $b_1$
implies the expression of $k_2$ as in (5), and
\begin{equation}  \label{eq: N2a x}
x = \frac{a_1 a_2^2 a_3}{d_0 \mathcal{B}_{33} \Delta_1}.
\end{equation}
By $\mathcal{B}_{33} > 0$ and $x > 0$,
it is implied that $\Delta_1 > 0$.
Then, it follows from \eqref{eq: N2a eqn8}
 and \eqref{eq: N2a x} that  the value of $c_2$ can be expressed as in (5).
Therefore, the necessity part is proved.

\emph{Sufficiency.} Suppose that $\mathcal{B}_{13} = 0$, $\Delta_1 > 0$, and   $a_0 \mathcal{B}_{33} = a_1 \mathcal{B}_{23} > 0$.
Let the values of the elements satisfy (5)
 and $x$ satisfy \eqref{eq: N2a x}. Then,
$\Delta_1 > 0$ and $\mathcal{B}_{33} > 0$ can guarantee that the element values are positive and finite. Since $\mathcal{B}_{13} = 0$ and $a_0 \mathcal{B}_{33} = a_1 \mathcal{B}_{23}$,   it can be verified that conditions~\eqref{eq: N2a eqn1}--\eqref{eq: N2a eqn8} hold. Therefore, $Z(s)$ is realizable as the configuration in Fig.~2(a).

\section{Proof of Lemma~8}

By the principles  of duality, frequency inversion, and frequency-inverse duality,
one only needs to prove that  the impedance $Z(s)$ of this lemma is realizable as the configuration in Fig.~3(a), if and only if $\mathcal{B}_{33} \Delta_1 = a_2a_3 \mathcal{B}_{13} > 0$ and $a_2 \mathcal{B}_{33} = a_3 \mathcal{B}_{23} > 0$ (Condition~1 of this lemma).

\emph{Necessity.} Suppose that $Z(s)$ is realizable as the configuration in
Fig.~3(a).
The impedance   of the configuration in Fig.~3(a) is calculated as
$Z(s) = a(s)/d(s)$, where
$a(s) = c_1^{-1} k_1^{-1} k_2^{-1} b_1 s^3 + c_1^{-1}c_2^{-1}k_1^{-1}b_1 s^2 + c_1^{-1}(k_1^{-1}+k_2^{-1}) s + c_1^{-1}c_2^{-1}$ and
$d(s) = k_1^{-1}k_2^{-1}b_1 s^3 + (c_1^{-1}k_2^{-1} + c_2^{-1}k_1^{-1})b_1 s^2 + (c_1^{-1}c_2^{-1}b_1 + k_1^{-1} + k_2^{-1}) s + c_1^{-1} + c_2^{-1}$.
Then, it follows   that
\begin{subequations}
\begin{align}
c_1^{-1} k_1^{-1} k_2^{-1} b_1  &=  x a_3,   \label{eq: N3a eqn1}    \\
c_1^{-1}c_2^{-1}k_1^{-1}b_1  &=  x a_2,   \label{eq: N3a eqn2}    \\
c_1^{-1}(k_1^{-1}+k_2^{-1})  &=  x a_1,   \label{eq: N3a eqn3}    \\
c_1^{-1}c_2^{-1}  &=  x a_0,   \label{eq: N3a eqn4}   \\
k_1^{-1}k_2^{-1}b_1  &=   x d_3, \label{eq: N3a eqn5}    \\
(c_1^{-1}k_2^{-1} + c_2^{-1}k_1^{-1})b_1  &=  x d_2,  \label{eq: N3a eqn6}    \\
c_1^{-1}c_2^{-1}b_1 + k_1^{-1} + k_2^{-1}  &=   x d_1,   \label{eq: N3a eqn7}    \\
c_1^{-1} + c_2^{-1}  &=   x d_0,   \label{eq: N3a eqn8}
\end{align}
\end{subequations}
where $x > 0$. Then, it follows from \eqref{eq: N3a eqn1} and \eqref{eq: N3a eqn5} that
the value of
$c_1$ can be expressed as in (6).
Substituting the expression of $c_1$ into \eqref{eq: N3a eqn4} and \eqref{eq: N3a eqn8}, one can obtain
\begin{equation} \label{eq: N3a x}
x = \frac{a_3^2}{d_3 \mathcal{B}_{13}}
\end{equation}
and the expression of $c_2$ as in (6).
Then, it is implied from $c_2 > 0$
that $\mathcal{B}_{13}:= a_3d_0 - a_0d_3 > 0$. Furthermore, substituting    \eqref{eq: N3a eqn3},   \eqref{eq: N3a x}, and the expressions of $c_1$ and $c_2$
into \eqref{eq: N3a eqn7} yields the expression of $b_1$ as in (6),
which by $b_1 > 0$ implies that $\mathcal{B}_{23} := a_3d_1 - a_1d_3 > 0$.
Similarly, substituting   \eqref{eq: N3a x}  and the expressions of $c_1$, $c_2$, and $b_1$
into \eqref{eq: N3a eqn2} and \eqref{eq: N3a eqn5} implies the expressions of $k_1$ and $k_2$ as in (6).
Based on the element values in (6) and $x$ in \eqref{eq: N3a x}, it can be derived that  \eqref{eq: N3a eqn3} and \eqref{eq: N3a eqn6} are equivalent to
$\mathcal{B}_{33} \Delta_1  = a_2a_3 \mathcal{B}_{13}$
and
$a_2 \mathcal{B}_{33} = a_3 \mathcal{B}_{23}$. Therefore, the necessity part is proved.

\emph{Sufficiency.} Suppose that Condition~1 of this lemma holds.
Let the values of the elements satisfy
(6) and $x$ satisfy \eqref{eq: N3a x}. Then,
$\mathcal{B}_{13} > 0$ and $\mathcal{B}_{23} > 0$ can guarantee that  the element values are positive and finite. Since $a_2 \mathcal{B}_{33} = a_3 \mathcal{B}_{23}$ and
$\Delta_1 \mathcal{B}_{33} = a_2a_3 \mathcal{B}_{13}$, it  can be verified that conditions~\eqref{eq: N3a eqn1}--\eqref{eq: N3a eqn8} hold.
Therefore, $Z(s)$ is realizable as the configuration in Fig.~3(a).

\section{Proof of Lemma~9}

By the principles  of duality, frequency inversion, and frequency-inverse duality,
one only needs to prove that  the impedance $Z(s)$ of this lemma is realizable as the configuration in Fig.~4(a), if and only if
$\mathcal{B}_{13} \Delta_1  = a_1^2 \mathcal{B}_{23} > 0$ and $a_1 \mathcal{B}_{33} = a_3 \mathcal{B}_{13} > 0$ (Condition~1 of this lemma).

\emph{Necessity.} Suppose that $Z(s)$ is realizable as the configuration in
Fig.~4(a).
The impedance   of the configuration in Fig.~4(a) is calculated as
$Z(s) = a(s)/d(s)$, where $a(s) = c_1^{-1}k_1^{-1}k_2^{-1}b_1 s^3 + c_1^{-1}c_2^{-1}(k_1^{-1}+k_2^{-1})b_1 s^2 + c_1^{-1}k_1^{-1} s + c_1^{-1}c_2^{-1}$ and $d(s) = k_1^{-1}k_2^{-1}b_1 s^3 + (c_1^{-1}k_2^{-1} + c_2^{-1}k_1^{-1} + c_2^{-1}k_2^{-1}) b_1 s^2 + (c_1^{-1}c_2^{-1}b_1 + k_1^{-1}) s + c_1^{-1} + c_2^{-1}$.
Then, it follows   that
\begin{subequations}
\begin{align}
c_1^{-1}k_1^{-1}k_2^{-1}b_1    &=  x a_3,   \label{eq: N4a eqn1}    \\
c_1^{-1}c_2^{-1}(k_1^{-1}+k_2^{-1})b_1    &=  x a_2, \label{eq: N4a eqn2}    \\
c_1^{-1}k_1^{-1}    &=  x a_1,   \label{eq: N4a eqn3}    \\
c_1^{-1}c_2^{-1}    &=  x a_0,   \label{eq: N4a eqn4}   \\
k_1^{-1}k_2^{-1}b_1    &=   x d_3, \label{eq: N4a eqn5}    \\
(c_1^{-1}k_2^{-1} + c_2^{-1}k_1^{-1} + c_2^{-1}k_2^{-1}) b_1    &=  x d_2,  \label{eq: N4a eqn6}    \\
c_1^{-1}c_2^{-1}b_1 + k_1^{-1}  &=  x d_1,  \label{eq: N4a eqn7}    \\
c_1^{-1} + c_2^{-1}  &=  x d_0,   \label{eq: N4a eqn8}
\end{align}
\end{subequations}
where $x > 0$. Then, it follows from \eqref{eq: N4a eqn1} and \eqref{eq: N4a eqn5} that the expression of $c_1$ satisfies (7),  which together with \eqref{eq: N4a eqn4} and \eqref{eq: N4a eqn8} can further imply that
\begin{equation}   \label{eq: N4a x}
x = \frac{a_3^2}{d_3 \mathcal{B}_{13}},
\end{equation}
and the expression of $c_2$ satisfies (7). Therefore, it is implied from $x > 0$ that $\mathcal{B}_{13}:= a_3d_0 - a_0d_3 > 0$. Substituting  \eqref{eq: N4a x} and the expression of $c_1$ into \eqref{eq: N4a eqn3} yields  the expression of $k_1$ as in (7). Similarly, substituting \eqref{eq: N4a x} and the expressions of $c_1$, $c_2$, and $k_1$  into \eqref{eq: N4a eqn5} and \eqref{eq: N4a eqn7} can imply that the values of $k_2$ and $b_1$ can be expressed as in (7),   which by $k_2 > 0$ implies  that $\mathcal{B}_{23} := a_3d_1 - a_1d_3 > 0$.
Based on the element values  in  (7) and $x$ in \eqref{eq: N4a x}, it can be derived that
\eqref{eq: N4a eqn2} and \eqref{eq: N4a eqn6} can be equivalent to
$\mathcal{B}_{13} \Delta_1 = a_1^2 \mathcal{B}_{23}$
and
$a_1 \mathcal{B}_{33} = a_3 \mathcal{B}_{13}$. Therefore, the necessity part is proved.

\emph{Sufficiency.} Suppose that Condition~1 of this lemma holds.
Let the values of the elements satisfy
(7) and $x$ satisfy \eqref{eq: N4a x}. Then,
$\mathcal{B}_{13} > 0$ and $\mathcal{B}_{23} > 0$ can guarantee that  the element values are positive and finite. Since
$a_1 \mathcal{B}_{33} = a_3 \mathcal{B}_{13}$ and $\mathcal{B}_{13} \Delta_1  = a_1^2 \mathcal{B}_{23}$,
it  can be verified that conditions~\eqref{eq: N4a eqn1}--\eqref{eq: N4a eqn8} hold.
Therefore, $Z(s)$ is realizable as the configuration in Fig.~4(a).

\section{Proof of Lemma~10}

By the principles  of duality, frequency inversion, and frequency-inverse duality,
one only needs to prove that  the impedance $Z(s)$ of this lemma is realizable as the configuration in Fig.~5(a), if and only if
$a_3^2 d_0^2 \Delta_2 = d_2^2 \mathcal{B}_{12} \mathcal{B}_{13} > 0$ and
$a_0 d_2^2 \mathcal{B}_{12} = a_3 d_0^2 (a_1d_2 - a_3d_0) > 0$ (Condition~1 of this lemma).

\emph{Necessity.} Suppose that $Z(s)$ is realizable as the configuration in
Fig.~5(a).
The impedance   of the configuration in Fig.~5(a) is calculated as
$Z(s) = a(s)/d(s)$, where
$a(s) = c_2^{-1}k_1^{-1}k_2^{-1}b_1 s^3 + (c_1^{-1}c_2^{-1}b_1 + k_1^{-1})k_2^{-1} s^2 + (c_1^{-1}k_2^{-1}+c_2^{-1}k_1^{-1}) s + c_1^{-1}c_2^{-1}$ and $d(s) = k_1^{-1}k_2^{-1}b_1 s^3 + (c_1^{-1}+c_2^{-1})k_2^{-1}b_1 s^2 + (k_1^{-1} + k_2^{-1}) s + c_1^{-1} + c_2^{-1}$.
Then, it follows   that
\begin{subequations}
\begin{align}
c_2^{-1}k_1^{-1}k_2^{-1}b_1    &=  x a_3,   \label{eq: N5a eqn1}    \\
(c_1^{-1}c_2^{-1}b_1 + k_1^{-1})k_2^{-1}   &=  x a_2, \label{eq: N5a eqn2}    \\
c_1^{-1}k_2^{-1}+c_2^{-1}k_1^{-1}  &=  x a_1,   \label{eq: N5a eqn3}    \\
c_1^{-1}c_2^{-1}  &=  x a_0,   \label{eq: N5a eqn4}   \\
k_1^{-1}k_2^{-1}b_1  &=   x d_3, \label{eq: N5a eqn5}    \\
(c_1^{-1}+c_2^{-1})k_2^{-1}b_1  &=  x d_2,  \label{eq: N5a eqn6}    \\
k_1^{-1} + k_2^{-1}  &=  x d_1,  \label{eq: N5a eqn7}    \\
c_1^{-1} + c_2^{-1}  &=  x d_0,   \label{eq: N5a eqn8}
\end{align}
\end{subequations}
where $x > 0$. Then, it follows from \eqref{eq: N5a eqn1} and \eqref{eq: N5a eqn5} that the expression of $c_2$ satisfies (8),  which together with \eqref{eq: N5a eqn4} and \eqref{eq: N5a eqn8} can further imply that
\begin{equation}   \label{eq: N5a x}
x = \frac{a_3^2}{d_3 \mathcal{B}_{13}},
\end{equation}
and the expression of $c_1$ satisfies (8). Therefore, it is implied from $x > 0$
that $\mathcal{B}_{13}:= a_3d_0 - a_0d_3 > 0$. By \eqref{eq: N5a eqn5}, \eqref{eq: N5a eqn6}, and the expressions of $c_1$ and $c_2$, the expression of $k_1$ can be derived as in (8). Then, substituting  \eqref{eq: N5a x} and the expression of $k_1$  into \eqref{eq: N5a eqn7} yields the expression of $k_2$ as in (8), which by
$k_2 > 0$
implies $\Delta_2 > 0$.
Furthermore, substituting \eqref{eq: N5a x} and the expressions of $k_1$ and $k_2$ into \eqref{eq: N5a eqn5} implies the expression of $b_1$ as in (8). Based on the element values  in  (8)
and $x$ in \eqref{eq: N5a x}, it can be derived that
\eqref{eq: N5a eqn2} and \eqref{eq: N5a eqn3} can be equivalent to
$a_3^2 d_0^2 \Delta_2 = d_2^2 \mathcal{B}_{12} \mathcal{B}_{13}$ and
$a_0 d_2^2 \mathcal{B}_{12} = a_3 d_0^2 (a_1d_2 - a_3d_0)$. Therefore, recalling that $\mathcal{B}_{13} > 0$ and $\Delta_2 > 0$,  it is implied that $\mathcal{B}_{12} > 0$. Therefore, the necessity part is proved.

\emph{Sufficiency.} Suppose that Condition~1 of this lemma holds.
Let the values of the elements satisfy
(8) and $x$ satisfy \eqref{eq: N5a x}. Then, it is implied that $\Delta_2 > 0$ and $\mathcal{B}_{13} > 0$, which can guarantee that  the element values are positive and finite. Since
$a_3^2 d_0^2 \Delta_2 = d_2^2 \mathcal{B}_{12} \mathcal{B}_{13}$ and
$a_0 d_2^2 \mathcal{B}_{12} = a_3 d_0^2 (a_1d_2 - a_3d_0)$,
it  can be verified that conditions~\eqref{eq: N5a eqn1}--\eqref{eq: N5a eqn8} hold.
Therefore, $Z(s)$ is realizable as the configuration in Fig.~5(a).

\section{Proof of Lemma~13}

By the principles  of duality, frequency inversion, and frequency-inverse duality,
one only needs to prove that  the impedance $Z(s)$ of this lemma is realizable as the configuration in Fig.~7(a), if and only if
$\mathcal{B}_{13} > 0$, $\mathcal{B}_{23} > 0$, $\mathcal{B}_{13}(a_2d_2 - \mathcal{B}_{23}) - a_2^2 d_0 d_3 = 0$, and
      $\mathcal{B}_{13} \mathcal{B}_{23} \Delta_1   - a_2^2a_3^2d_0^2 = 0$
(Condition~1 of this lemma).

\emph{Necessity.} Suppose that $Z(s)$ is realizable as the configuration in
Fig.~7(a). The impedance   of the configuration in Fig.~7(a) is calculated as $Z(s) = a(s)/d(s)$, where
$a(s) = (c_1^{-1} + c_2^{-1})k_1^{-1}k_2^{-1}b_1 s^3 + c_1^{-1}c_2^{-1}k_1^{-1}b_1 s^2 + (c_1^{-1}+c_2^{-1})(k_1^{-1}+k_2^{-1}) s + c_1^{-1}c_2^{-1}$ and $d(s) = k_1^{-1}k_2^{-1}b_1 s^3 + (c_1^{-1}k_2^{-1}+c_2^{-1}(k_1^{-1}+k_2^{-1}))b_1 s^2 + (c_1^{-1}c_2^{-1}b_1 + k_1^{-1} + k_2^{-1}) s + c_1^{-1}$.
Then, it follows that
\begin{subequations}
\begin{align}
(c_1^{-1} + c_2^{-1})k_1^{-1}k_2^{-1}b_1  &=  x a_3,   \label{eq: N7a eqn1}    \\
c_1^{-1}c_2^{-1}k_1^{-1}b_1  &=  x a_2, \label{eq: N7a eqn2}    \\
(c_1^{-1}+c_2^{-1})(k_1^{-1}+k_2^{-1})  &=  x a_1,   \label{eq: N7a eqn3}    \\
c_1^{-1}c_2^{-1}  &=  x a_0,   \label{eq: N7a eqn4}   \\
k_1^{-1}k_2^{-1}b_1  &=   x d_3, \label{eq: N7a eqn5}    \\
(c_1^{-1}k_2^{-1}+c_2^{-1}(k_1^{-1}+k_2^{-1}))b_1  &=  x d_2,  \label{eq: N7a eqn6}    \\
c_1^{-1}c_2^{-1}b_1 + k_1^{-1} + k_2^{-1}  &=  x d_1,  \label{eq: N7a eqn7}    \\
c_1^{-1}  &=  x d_0,   \label{eq: N7a eqn8}
\end{align}
\end{subequations}
where $x > 0$. Then, it follows from \eqref{eq: N7a eqn4} and \eqref{eq: N7a eqn8} that the expression of $c_2$ can be obtained as in (11). Combining \eqref{eq: N7a eqn1} and \eqref{eq: N7a eqn5}, it is implied that
\begin{equation} \label{eq: N7a R1 + R2}
c_1^{-1} + c_2^{-1} = \frac{a_3}{d_3}.
\end{equation}
Substituting the expression of $c_2$ into \eqref{eq: N7a R1 + R2} yields the expression of $c_1$ as in (11), which together with \eqref{eq: N7a eqn8} implies
\begin{equation}  \label{eq: N7a x}
x = \frac{\mathcal{B}_{13}}{d_0^2d_3}.
\end{equation}
Therefore, it follows from $x > 0$
that $\mathcal{B}_{13} > 0$. Together with  \eqref{eq: N7a x} and the expressions of $c_1$ and $c_2$, the expression of $b_1$ can be obtained from \eqref{eq: N7a eqn3} and \eqref{eq: N7a eqn7} as in (11), which by $b_1 > 0$
implies  $\mathcal{B}_{23} > 0$. Then, substituting \eqref{eq: N7a x} and the expressions of $c_1$, $c_2$, and $b_1$ into \eqref{eq: N7a eqn2} can yield the expression of $k_1$ as in (11), which together with \eqref{eq: N7a eqn5} yields the expression of $k_2$ as in (11). Together with the element values in (11) and $x$ in \eqref{eq: N7a x}, it follows from \eqref{eq: N7a eqn3} and \eqref{eq: N7a eqn6} that   $\mathcal{B}_{13} \mathcal{B}_{23} \Delta_1 - a_2^2a_3^2d_0^2 = 0$ and $\mathcal{B}_{13}(a_2d_2 - \mathcal{B}_{23}) - a_2^2 d_0 d_3 = 0$. Therefore, the necessity part is proved.

\emph{Sufficiency.} Suppose that Condition~1 of this lemma holds. Let the values of the elements satisfy
(11) and $x$ satisfy \eqref{eq: N7a x}. Then, $\mathcal{B}_{13} > 0$ and $\mathcal{B}_{23} > 0$ can guarantee that  the element values are positive and finite.  Since $\mathcal{B}_{13}(a_2d_2 - \mathcal{B}_{23}) - a_2^2 d_0 d_3 = 0$ and $\mathcal{B}_{13} \mathcal{B}_{23} \Delta_1 - a_2^2a_3^2d_0^2 = 0$, it can be verified that
conditions~\eqref{eq: N7a eqn1}--\eqref{eq: N7a eqn8} hold.
Therefore, $Z(s)$ is realizable as the configuration in Fig.~7(a).

\section{Proof of Lemma~16}

By  the principle  of duality,
one only needs to prove that  the impedance $Z(s)$ of this lemma is realizable as the configuration in Fig.~10(a), if and only if
$0 < \Psi_1 < a_1/d_0$, $a_3d_0 - 4a_0d_3 \geq 0$,  $(a_1d_2 - a_3d_0)^2 - 4a_0a_3d_0d_3 \geq 0$, $d_0^2 \Psi_1^3 - a_1d_0 \Psi_1^2 + a_2d_0\Gamma_1 \Psi_1 - a_0a_3\Gamma_1 = 0$, and $d_0^2 \Psi_1^3 + d_0(d_1\Gamma_1 - 2a_1) \Psi_1^2 - a_1(d_1 \Gamma_1 - a_1) \Psi_1 + a_0d_3 \Gamma_1^2 = 0$ (Condition~1 of this lemma), where
$\Gamma_1$ and $\Psi_1$ are defined in (15) and (17), respectively.

\emph{Necessity.} Suppose that $Z(s)$ is realizable as the configuration in
Fig.~10(a). The impedance   of the configuration in Fig.~10(a) is calculated as $Z(s) = a(s)/d(s)$, where
$a(s) = (c_1^{-1} + c_2^{-1})k_1^{-1}k_2^{-1}b_1 s^3 + (c_1^{-1}c_2^{-1}b_1 + k_1^{-1})k_2^{-1} s^2 + ((c_1^{-1}+c_2^{-1})k_1^{-1} + c_1^{-1}k_2^{-1}) s + c_1^{-1}c_2^{-1}$ and $d(s) = k_1^{-1}k_2^{-1}b_1 s^3 + ((c_1^{-1}+c_2^{-1})k_1^{-1}+c_2^{-1}k_2^{-1})b_1 s^2 +
(c_1^{-1}c_2^{-1}b_1 + k_2^{-1}) s + c_1^{-1} + c_2^{-1}$.  Then, it follows that
\begin{subequations}
\begin{align}
(c_1^{-1} + c_2^{-1})k_1^{-1}k_2^{-1}b_1     &=  x a_3,   \label{eq: N10a eqn1}    \\
(c_1^{-1}c_2^{-1}b_1 + k_1^{-1})k_2^{-1}  &=  x a_2,   \label{eq: N10a eqn2}    \\
(c_1^{-1}+c_2^{-1})k_1^{-1} + c_1^{-1}k_2^{-1} &=  x a_1,   \label{eq: N10a eqn3}    \\
c_1^{-1}c_2^{-1}  &=  x a_0,   \label{eq: N10a eqn4}   \\
k_1^{-1}k_2^{-1}b_1  &=  x d_3,  \label{eq: N10a eqn5}    \\
((c_1^{-1}+c_2^{-1})k_1^{-1}+c_2^{-1}k_2^{-1})b_1  &=  x d_2,  \label{eq: N10a eqn6}    \\
c_1^{-1}c_2^{-1}b_1 + k_2^{-1}  &=  x d_1,   \label{eq: N10a eqn7}    \\
c_1^{-1} + c_2^{-1}  &=      x d_0,  \label{eq: N10a eqn8}
\end{align}
\end{subequations}
where $x > 0$. It follows from \eqref{eq: N10a eqn1} and \eqref{eq: N10a eqn5} that
\begin{equation} \label{eq: N10a R1 + R2}
c_1^{-1} + c_2^{-1} = \frac{a_3}{d_3},
\end{equation}
which together with \eqref{eq: N10a eqn8} implies that
\begin{equation} \label{eq: N10a x}
x = \frac{a_3}{d_0d_3}.
\end{equation}
By \eqref{eq: N10a eqn4}, \eqref{eq: N10a R1 + R2}, and \eqref{eq: N10a x}, it is implied that
$d_0d_3 c_1^{-2} - a_3d_0 c_1^{-1} + a_0a_3 = 0$ and $c_2^{-1} = a_0a_3/(d_0d_3 c_1^{-1}) = a_3/d_3 -  c_1^{-1}$.
Furthermore, $a_3 d_0 - 4a_0 d_3 \geq 0$, and one obtains
$c_1 = 1/\Gamma_1$ and $c_2 = d_0d_3 \Gamma_1/(a_0a_3)$  as in (19), where
$\Gamma_1$ is defined in (15). Together with
\eqref{eq: N10a R1 + R2}, and \eqref{eq: N10a x}, it follows from  \eqref{eq: N10a eqn3}
and \eqref{eq: N10a eqn5} that
\begin{equation}   \label{eq: N10a L2 & C1}
k_2 = \frac{d_0d_3 c_1^{-1}}{a_3 (a_1 - d_0 k_1^{-1})}, ~~
b_1 = \frac{d_3 c_1^{-1}}{(a_1 - d_0 k_1^{-1}) k_1^{-1}}.
\end{equation}
Then, substituting \eqref{eq: N10a R1 + R2}--\eqref{eq: N10a L2 & C1} into \eqref{eq: N10a eqn6} implies
$d_0d_2 k_1^{-2} + (d_0d_3 c_1^{-1} - d_0d_3 c_2^{-1} - a_1 d_2) k_1^{-1} + a_1d_3 c_2^{-1} = 0$,
which together with \eqref{eq: N10a eqn4}, \eqref{eq: N10a R1 + R2}, and \eqref{eq: N10a x} implies that $(a_1d_2 - a_3d_0)^2 - 4a_0a_3d_0d_3 \geq 0$ and $k_1 = 1/\Psi_1$. By \eqref{eq: N10a L2 & C1}, the expressions of $k_2$ and $b_1$  can be directly obtained as in (19), which by $k_1 > 0$ and $k_2 > 0$
implies that $0 < \Psi_1 < a_1/d_0$. By $c_1 = 1/\Gamma_1$, $k_1 = 1/\Psi_1$,  \eqref{eq: N10a eqn4},  \eqref{eq: N10a x}, and
\eqref{eq: N10a L2 & C1}, it follows from \eqref{eq: N10a eqn2} and \eqref{eq: N10a eqn7} that $d_0^2 \Psi_1^3 - a_1d_0 \Psi_1^2 + a_2d_0\Gamma_1 \Psi_1 - a_0a_3\Gamma_1 = 0$, and $d_0^2 \Psi_1^3 + d_0(d_1\Gamma_1 - 2a_1) \Psi_1^2 - a_1(d_1 \Gamma_1 - a_1) \Psi_1 + a_0d_3 \Gamma_1^2 = 0$ hold. Therefore, the necessity part is proved.

\emph{Sufficiency.} Suppose that Condition~1 of this lemma holds. Let the values of the elements satisfy
(19) and $x$ satisfy \eqref{eq: N10a x}. Since $0 < \Psi_1 < a_1/d_0$, $a_3d_0 - 4a_0d_3 \geq 0$, and $(a_1d_2 - a_3d_0)^2 - 4a_0a_3d_0d_3 \geq 0$, it is clear that the element values can be positive and finite.
Since  $d_0^2 \Psi_1^3 - a_1d_0 \Psi_1^2 + a_2d_0\Gamma_1 \Psi_1 - a_0a_3\Gamma_1 = 0$, and $d_0^2 \Psi_1^3 + d_0(d_1\Gamma_1 - 2a_1) \Psi_1^2 - a_1(d_1 \Gamma_1 - a_1) \Psi_1 + a_0d_3 \Gamma_1^2 = 0$,  it can be verified that
conditions~\eqref{eq: N10a eqn1}--\eqref{eq: N10a eqn8} hold.
Therefore, $Z(s)$ is realizable as the configuration in Fig.~10(a).

\section{Conclusion}

In this report, the proofs of some results in the original
paper \cite{WC_sub} have been presented,  which are omitted from the paper for brevity.



\vspace{0.2cm}



\begin{thebibliography}{00}





\bibitem{AV73}
B.D.O. Anderson, S. Vongpanitlerd,  Network Analysis and Synthesis: A Modern Systems Theory Approach, Prentice Hall, New Jersey, 1973.

\bibitem{CWC19}
M.Z.Q. Chen, K. Wang,  G. Chen, Passive Network Synthesis: Advances with Inerter,   World Scientific, Singapore, 2020.


\bibitem{You15}
D.C. Youla,  Theory and Synthesis of Linear Passive Time-Invariant
Networks,  Cambridge University Press, Cambridge, 2015.


\bibitem{BD49}
R. Bott, R. J. Duffin,  Impedance synthesis without use of transformers,
Journal of Applied Physics  20 (8) (1949) 816.



\bibitem{Smi02}
M.C. Smith,  Synthesis of mechanical networks:
The inerter, IEEE Trans. Automatic Control  47 (10) (2002) 1648--1662.



\bibitem{ACWJ18}
N. Alujevi$\acute{\text{c}}$, D. $\check{\text{C}}$akmak, H. Wolf, M. Joki$\acute{\text{c}}$,
Passive and active vibration isolation systems using inerter,
Journal of Sound and Vibration 418 (2018) 163--183.



\bibitem{JW19}
E.D.A. John, D.J. Wagg, Design and testing of a frictionless mechanical inerter device using living-hinges, Journal of the Franklin Institute 356 (14) (2019) 7650--7668.


\bibitem{SZ19}
X. Shi, S. Zhu, A comparative study of vibration isolation performance using negative stiffness and inerter
dampers, Journal of the Franklin Institute 356 (14) (2019) 7922--7946.




\bibitem{CLLNSC17}
L. Chen, C. Liu, W. Liu, J. Nie, Y. Shen,   G. Chen,  Network synthesis and parameter optimization for vehicle suspension with inerter, Advances in Mechanical Engineering 9 (1) (2017) 1--7.




\bibitem{HC18}
Y. Hu, M.Z.Q. Chen,  Low-complexity passive vehicle suspension design based on element-number-restricted networks and low-order admittance networks,
Journal of Dynamic Systems, Measurement, and Control 140 (10) (2018) 101014.




\bibitem{PS06}
C. Papageorgiou, M.C. Smith, Positive real synthesis using matrix inequalities for mechanical networks: Application to vehicle suspension,
IEEE Trans. Control Systems Technology 14 (3) (2006) 423--435.

\bibitem{NSYZDZL19}
D. Ning, S. Sun, J. Yu, M. Zheng, H. Du, N. Zhang, W. Li, A rotary variable admittance device and its application
in vehicle seat suspension vibration control, Journal of the Franklin Institute 356 (14) (2019)
7873--7895.




\bibitem{SW04}
M.C. Smith, F.C. Wang,  Performance benefits in passive vehicle suspensions employing inerters, Vehicle System Dynamics 42 (4) (2004) 235--257.



\bibitem{ZJN17}
S.Y. Zhang, J.Z. Jiang,  S.A. Neild, Passive vibration control: A structure-immittance approach,
Proceedings of the Royal Society A 473 (2201) (2017) 20170011.





\bibitem{WLLSC09}
F.C. Wang, M.K. Liao, B.H. Liao, W.J. Su,  H.A. Chan, The performance improvements of train suspension systems with mechanical networks employing inerters,  Vehicle System Dynamics 47 (7) (2009) 805--830.




\bibitem{CL19}
L. Cao, C. Li,  Tuned tandem mass dampers-inerters with broadband high effectiveness for structures under white noise base excitations,  Structural Control and Health Monitoring  26 (4) (2019) e2319.




\bibitem{PRRK19}
F. Palacios-Qui$\tilde{\text{n}}$onero,
J. Rubi$\acute{\text{o}}$-Masseg$\acute{\text{u}}$, J.M. Rossell,
 H.R. Karimi,
Design of inerter-based multi-actuator systems for vibration control of adjacent structures, Journal of the Franklin Institute 356 (14) (2019) 7785--7809.



\bibitem{YS16}
K. Yamamoto, M.C. Smith, Bounded disturbance amplification for mass chains with passive interconnection,
IEEE Trans. Automatic Control 61 (6) (2016) 1565--1574.




\bibitem{WNZ19}
X. Wei, B.F. Ng,  X. Zhao, Aeroelastic load control of large and flexible wind turbines through mechanically driven flaps,
Journal of the Franklin Institute 356 (14) (2019) 7810--7835.





\bibitem{LBHD11}
J. Lavaei, A. Babakhani, A. Hajimiri,   J.C. Doyle,  Solving large-scale hybrid circuit-antenna problems, IEEE Trans. Circuits and Systems I: Regular Papers 58 (2) (2011) 374--387.






\bibitem{DGYM20}
R. Deaton, M. Garzon, R. Yasmin,   T. Moorse,  A model for self-assembling circuits with voltage-controlled growth, International Journal of Circuit Theory and Applications 48 (7) (2020) 1017--1031.




\bibitem{DZHD17}
R. Drummond, S. Zhao, D.~A. Howey,   S.R. Duncan,  Circuit synthesis of electrochemical supercapacitor models, Journal of Energy Storage 10 (2017) 48--55.






\bibitem{BHBS16}
S. Bilbao, B. Hamilton, J. Botts,   L. Savioja,  Finite volume time domain room acoustics simulation under general impedance boundary conditions,
IEEE/ACM Trans. Audio, Speech, and Language Processing 24 (1) (2016) 161--173.




\bibitem{Sae14}
K. Saeed,  Carath$\acute{\text{e}}$odory-Toeplitz based mathematical methods and their algorithmic applications in biometric image processing,  Applied Numerical Mathematics  75 (2014) 2--21.




\bibitem{PM19}
R. Pates,  E. Mallada,  Robust scale-free synthesis for frequency control in power systems,
IEEE Trans. Control of Network Systems 6 (3) (2019) 1174--1184.




\bibitem{Hak20}
M. Hakimi-Moghaddam,  Positive real and strictly positive real MIMO systems: Theory and application, International Journal of Dynamics and Control 8 (2020) 448--458.





\bibitem{LX18}
M. Liu, J. Xiong,  Bilinear transformation for discrete-time positive real and negative imaginary systems,  IEEE Trans. Automatic Control 63 (12) (2018) 4264--4269.





\bibitem{CS09(2)}
M.Z.Q. Chen, M.C. Smith, A note on tests for positive-real functions,
IEEE Trans. Automatic Control 54 (2) (2009) 390--393.





\bibitem{CS09}
M.Z.Q. Chen, M.C. Smith, Restricted complexity network realizations for passive mechanical control,
IEEE Trans. Automatic Control 54 (10) (2009) 2290--2301.




\bibitem{CWSL13}
M.Z.Q. Chen, K. Wang, Z. Shu, C. Li, Realizations of a special
class of admittances with strictly lower complexity than canonical forms,
IEEE Trans. Circuits and Systems I: Regular Papers 60 (9) (2013) 2465--2473.







\bibitem{Hug17}
T.H. Hughes, Why RLC realizations of certain impedances need many
more energy storage elements than expected,
IEEE Trans. Automatic Control 62 (9) (2017) 4333--4346.





\bibitem{Hug20}
T.H. Hughes,  Minimal series-parallel network realizations of bicubic impedances,
IEEE Transactions on Automatic Control, in press, DOI: 10.1109/TAC.2020.2968859.



\bibitem{JS11}
J.Z. Jiang, M.C. Smith, Regular positive-real functions and
five-element network synthesis for electrical and mechanical networks,
IEEE Trans. Automatic Control 56 (6) (2011) 1275--1290.



\bibitem{LQ19}
G. Liang, Z. Qi, Synthesis of passive fractional-order LC n-port with three element orders,
IET Circuits, Devices and Systems 13 (1) (2019) 61--72.











\bibitem{ST17}
M.S. Sarafraz, M.S. Tavazoei, Passive realization of fractional-order impedances by a fractional element and RLC components: Conditions and procedure,
IEEE Trans. Circuits and Systems I: Regular Papers
64 (3) (2017) 585--595.





\bibitem{WCH14}
K. Wang, M.Z.Q. Chen,  Y. Hu, Synthesis of biquadratic impedances with at most four passive elements,  Journal of the Franklin Institute 351 (3) (2014) 1251--1267.








\bibitem{WJ19}
K. Wang, X. Ji,  Passive controller realization of a bicubic admittance containing a pole at s = 0 with no more than five elements for inerter-based mechanical control,
Journal of the Franklin Institute 356 (14) (2019) 7896--7921.





\bibitem{WC18}
K. Wang, M.Z.Q. Chen, C. Li,  G. Chen, Passive controller realization of a biquadratic impedance with double poles and zeros as a seven-element series-parallel network for effective mechanical control,
IEEE Trans. Automatic Control 63 (9) (2018) 3010--3015.




\bibitem{WC20}
K. Wang,  M.Z.Q. Chen, On realizability of specific biquadratic impedances as three-reactive seven-element series-parallel networks for inerter-based mechanical control,
IEEE Trans. Automatic Control 66 (1) (2021) 340--345.







\bibitem{YKKP14}
B.S. Yarman, R. Kopru, N. Kuman,  C. Prakash, High precision synthesis of a Richards immittance via parametric approach,
IEEE Trans. Circuits and Systems I: Regular Papers 61 (4) (2014) 1055--1067.





\bibitem{ZJWN17}
S.Y. Zhang, J.Z. Jiang, H.L. Wang,  S. Neild, Synthesis of essential-regular bicubic impedances,
International Journal of Circuit Theory and Applications 45 (11) (2017) 1482--1496.




\bibitem{Smi17}
M.C. Smith, Kalman's last decade: Passive network synthesis,
IEEE Control Systems Magazine 37 (2) (2017) 175--177.



\bibitem{WC_sup}
K. Wang,  M.Z.Q. Chen, Supplementary material to: Passive network realizations of bicubic impedances with no more than five elements for inerter-based control design (technical report to be available in arXiv.org).


\bibitem{Fuh12}
P.A. Fuhrmann, A Polynomial Approach to Linear Algebra, Second Edition, Springer, New York, 2012.



\bibitem{SR61}
S. Seshu,  M.B. Reed, Linear Graphs and Electrical Networks, Addison-Wesley, Boston, 1961.




\bibitem{Ses59}
S. Seshu,  Minimal realizations of the biquadratic minimum function, IRE Trans. Circuit Theory 6 (4) (1959) 345--350.



\bibitem{DFT09}
J.C. Doyle, B.A. Francis, A.R. Tannenbaum, Feedback Control Theory, Dover Publications, New York, 2009.


\bibitem{Lin65}
P.M. Lin, A theorem on equivalent one-port networks, IEEE Trans. Circuit Theory 12 (4) (1965) 619--621.




\bibitem{Gan80}
F.R. Gantmacher, The Theory of Matrices, vol. II, Chelsea, New York, 1980.

























\end{thebibliography}

\begin{thebibliography}{99}



\bibitem{WC_sub}
K. Wang, M.Z.Q. Chen,  Passive mechanical realizations of bicubic impedances with no more than five elements for inerter-based controll  design, Journal of the Franklin Institute 358 (10) (2021) 5353--5385.





\bibitem{PS06}
C. Papageorgiou, M.C. Smith, Positive real synthesis using matrix inequalities for mechanical networks: Application to vehicle suspension,
IEEE Transactions on Control Systems Technology 14 (3) (2006) 423--435.


\bibitem{SW04}
M.C. Smith, F.C. Wang,  Performance benefits in passive vehicle suspensions employing inerters,
Vehicle System Dynamics 42 (4) (2004) 235--257.


\bibitem{BD49}
R. Bott, R.J. Duffin,  Impedance synthesis without use of transformers,
Journal of Applied Physics 20 (8) (1949) 816.



\end{thebibliography}
\end{document}